\documentclass[11pt]{amsart}
\usepackage{setspace}
\usepackage{xcolor}
\usepackage{makecell}
\usepackage{amssymb,amscd,amsthm,verbatim,amsmath,color,fancyhdr,mathrsfs,amsfonts,amssymb,commath, graphicx,bbm}
\usepackage[bookmarks=false]{hyperref}
\usepackage[english]{babel}
\usepackage{parskip}
\usepackage{enumitem}
\usepackage[letterpaper,top=2cm,bottom=2cm,left=3cm,right=3cm,marginparwidth=1.75cm]{geometry}

\DeclareMathOperator{\Id}{\mathrm{Id}}

\newcommand{\Was}[1]{\mathbb{W}_{#1}}

\newtheorem{corollary}[]{Corollary}

\newcommand{\Div}{\mathrm{div}_{g}}

\DeclareMathOperator*{\argmin}{argmin}

\newcommand{\cP}{\mathcal{P}}

\newcommand{\cC}{\mathcal{C}}
\newcommand{\cV}{\mathcal{V}}
\newcommand{\cU}{\mathcal{U}}

\newcommand{\cL}{\mathcal{L}}
\newcommand{\cM}{\mathcal{M}}

\newcommand{\Ent}{\mathrm{Ent}}
\newcommand{\Leb}{\mathrm{Leb}}

\newcommand{\Tr}{\text{Tr}}

\newcommand{\Schro}{\text{Schr\"{o}dinger}}
\newcommand{\Ito}{\text{It\^{o}}}
\newcommand{\Exp}[1]{\mathrm{E}_{#1}}

\newcommand{\vol}{\mathrm{vol}}
\newcommand{\diffref}{\mathbf{m}}
\newcommand{\ricci}{\mathrm{Ric}}
\newcommand{\eps}{\varepsilon}

\newtheorem{remark}{Remark}

\newtheorem*{definition}{Definition}

\newtheorem{theorem}{Theorem}
\newtheorem{informaltheorem}{Informal Theorem}
\newtheorem{assumption}{Assumption}
\newtheorem{proposition}{Proposition}

\newtheorem*{conjecture}{Conjecture}
\newcommand{\commentout}[1]{}
\newcommand{\rr}{\mathbb{R}}

\title[Diffusion Approximations to Schr\"{o}dinger Bridges]{Diffusion Approximations to Schr\"{o}dinger Bridges on Manifolds}

\author{Garrett Mulcahy}
\address{Garrett Mulcahy\\ Department of Mathematics \\ University of Washington\\ Seattle WA 98195, USA\\ {Email: gmulcahy@uw.edu}}
\author{Soumik Pal}
\address{Soumik Pal\\ Department of Mathematics \\ University of Washington\\ Seattle WA 98195, USA\\ {Email: soumik@uw.edu}}

	

\thanks{This research is partially supported by the following grants. Both authors are supported by NSF grants DMS-2134012 and DMS-2133244 and PIMS PRN 01 (Kantorovich Initiative). Mulcahy is also supported by the NSF Graduate Research Fellowship Program under Grant No.\ DGE-2140004.}

\date{\today}

\begin{document}

\begin{abstract}
We present a collection of explicit diffusion approximations to small temperature $\Schro$ bridges on manifolds. Our most precise results are when both marginals are the same and the $\Schro$ bridge is on a manifold with a reference process given by a reversible diffusion. In the special case that the reference process is the manifold Brownian motion, we use the small time heat kernel asymptotics to show that the gradient of the corresponding $\Schro$ potential converges in $L^2$, as the temperature vanishes, to a manifold analogue of the score function of the marginal. As an application of the previous result we show that the Euclidean $\Schro$ bridge, computed for the quadratic cost, between two different marginal distributions can be approximated by a transformation of a two point distribution of a stationary Mirror Langevin diffusion. 
\end{abstract}

\maketitle

\section{Introduction}\label{sec:introduction}
Let $(M,g)$ denote a smooth (in the sense of $C^{\infty}$) connected, complete Riemannian manifold without boundary, equipped with the Levi-Civita connection denoted $\nabla_{g}$. The manifold need not be compact. Fix some $V \in C^2(M)$ (further assumptions will be detailed later) and consider the SDE
\begin{align}\label{eq:intro-ref-proc}
    dX_t = -\frac{1}{2}\nabla_g V(X_t) dt + dB_t^{M}, \quad t \geq 0,
\end{align}
where $(B_t^{M}, t \geq 0)$ is Brownian motion on $M$, and $\vol$ refers to the volume measure on the manifold defined in \eqref{eq:vol-measure-coords}.  
% Observe that we release the SDE from its reversing measure. 
We impose the initial distribution $X_0 \sim e^{-V}\vol$, which is the reversible measure of the diffusion and may not be a probability measure. We write $e^{-V}\vol$ to denote the measure whose Radon-Nikodym derivative with respect to $\vol$ is equal to $e^{-V}$. We use $\cM_{+}(X)$ and $\cP(X)$ to denote the set of Borel positive and probability measures on a topological space $X$, respectively. 
Let $R_{0\eps} \in \cM_{+}(M \times M)$ denote the law of $(X_0,X_\eps)$. 

Now, fix $\mu \in \cP_2(M)$ satisfying assumptions we detail later. The same marginal $\eps$-(static) $\Schro$ bridge \cite{schroLeonard13} from $\mu$ to itself is defined as
\begin{align}\label{eq:intro-sb-defn}
    \pi^{\eps} &:= \argmin\limits_{\pi \in \Pi(\mu,\mu)} H(\pi|R_{0\eps}),
\end{align}
where $H(\cdot|\cdot)$ is relative entropy (Kullback-Leibler divergence) and $\Pi(\mu,\mu) \subset \cP(M \times M)$ is the set of couplings of $\mu$ with itself. The relative entropy of $\mu_1 \in \cP(M)$ with respect to $\mu_2 \in \cM_{+}(M)$ is defined as
\begin{align}\label{eq:intro-rel-ent-defn}
    H(\mu_1|\mu_2) :=
        \Exp{\mu_1}\left[\log \left(\frac{d\mu_1}{d\mu_2}\right)\right], \quad \text{if $\mu_1 << \mu_2$}
\end{align}
and equal to $+\infty$ otherwise. Another information theoretic quantity that we consider is the \textbf{(relative) Fisher information}, defined as
\begin{align}\label{eq:intro-fi-defn}
    I_{g}(\mu_2|\mu_1) := 
        \Exp{\mu_1}\left[\norm{ \nabla_g \log\left(\frac{d\mu_1}{d\mu_2}\right)}^2_g\right], \quad \text{if $\mu_1 << \mu_2$}
\end{align}
and equal to $+\infty$ otherwise.

In this paper, we are interested in understanding the behavior of $\pi^{\eps}$ when $\eps$ (which we refer to as temperature) is small as well as the influence of the geometry of $M$ on this behavior. 

The key heuristic is the following: at small temperature, the same marginal $\Schro$ bridge has many attributes in common with a diffusion. To be precise, suppose that $\mu = \exp(-U)\vol$ for some $U \in C^2(M)$ (further assumptions will be detailed later), and consider the SDE
\begin{align}\label{eq:intro-target-process}
    dY_t = -\frac{1}{2}\nabla_g U(Y_t) dt + dB_t^{M}, \quad Y_0 \sim e^{-U}\vol.
\end{align}
The content of Theorem \ref{thm:sym-rel-ent} is that the joint distribution of $(Y_0,Y_{\eps})$ is a tight approximation in symmetric relative entropy to $\pi^{\eps}$. 
\begin{informaltheorem}[Informal Statement of Theorem \ref{thm:sym-rel-ent} below]
    Let $\ell_{\eps} = \mathrm{Law}(Y_0,Y_{\eps})$, where $(Y_t, t \geq 0)$ is as defined in \eqref{eq:intro-target-process}. Under assumptions, it holds that
    \begin{align}\label{eq:sym-rel-informal}
        H(\ell_{\eps}|\pi^{\eps})+H(\pi^{\eps}|\ell_{\eps}) = o(\eps^2).
    \end{align}
\end{informaltheorem}

From the $V$ and $U$ defined in \eqref{eq:intro-ref-proc} and \eqref{eq:intro-target-process}, respectively, define the \textbf{reciprocal characteristics} of these processes by
\begin{align}
    \cV = \frac{1}{8}\norm{\nabla_g V}_g^2 - \frac{1}{4}\Delta_g V \text{ and } \cU = \frac{1}{8}\norm{\nabla_g U}_g^2 - \frac{1}{4}\Delta_g U,
\end{align}
where $\Delta_{g}$ is the Laplace-Beltrami operator on $(M,g)$, defined later in \eqref{eq:laplace-beltrami-coords}.  
This quantity has been identified as a crucial object in the study of Langevin diffusions \cite{LKreener,conforti-recip-characteristics18,vonrenesse-conf18} and has an interpretation as acceleration in the Wasserstein space \cite{conforti-second-order-sb19}. Recently, \cite{chaintron2025propagation} outlined the role that the convexity of $\cU$ plays in propagating a generalized notion of convexity along solutions to Hamilton-Jacobi-Bellman equations. In our work, the difference $(\cU-\cV)$ appears in the upper bound of \eqref{eq:sym-rel-informal}. Note that while the choice of $V$ in \eqref{eq:intro-ref-proc} does not affect the choice of diffusion approximation \eqref{eq:intro-target-process}, it does affect the upper bound in \eqref{eq:sym-rel-informal} 

In Theorem \ref{thm:generator-transfer}, this diffusion approximation is leveraged to show that the family of conditional distributions $(\pi^{\eps}(\cdot|x), \eps > 0, x \in M)$ admits an approximate generator that agrees with that of the diffusion in \eqref{eq:intro-target-process}. We emphasize that there is no apriori Markov structure to this family of integral operators because the underlying coupling changes with each $\eps$. Nonetheless, these operators stitch together in a way that is approximately Markov as $\eps$ vanishes. 
\begin{informaltheorem}[Informal Statement of Theorem \ref{thm:generator-transfer}]
    Let $\cL$ denote the generator of the diffusion in \eqref{eq:intro-target-process}. Under assumptions, for $\xi: M \to \mathbb{R}$ Lipshitz and in the domain of $\cL$, it holds in $L^2(\mu)$
    \begin{align*}
        \lim\limits_{\eps \downarrow 0}\frac{1}{\eps}\left(\Exp{\pi^{\eps}}[\xi(Z)|X=x]-\xi(x)\right) = \cL\xi(x).
    \end{align*}
\end{informaltheorem}
The proof shows an interaction between probability and geometry of the manifold. We capture this interaction by assuming a curvature dimension condition, defined below in \ref{eq:cd-defn}. 

% It is known that t
The Radon-Nikodym derivative of the $\Schro$ bridge with respect to its reference measure is known to admit a product decomposition \cite[Section 3]{schroLeonard13}. By the symmetry of the marginal measures, for $\pi^{\eps}$ defined in \eqref{eq:intro-sb-defn} there exists a unique $a^{\eps}: M \to [0,+\infty)$ such that 
\begin{align}\label{eq:intro-potent-defn}
    \frac{d\pi^{\eps}}{dR_{0\eps}}(x,z) &= a^{\eps}(x)a^{\eps}(z).
\end{align}
The quantity $\eps \log a^{\eps}$ is called in the literature a $\Schro$ potential (also entropic potential), and its limit and well as that of its gradient as $\eps \downarrow 0$ are of great interest for their connection to the unregularized optimal transport problem \cite{pooladian2022entropic,nutz-weisel-22,chiarini2022gradient}. 
It is known that $\eps \nabla_g \log a^{\eps} \to 0$, the gradient of the Kantorovich potential, in $L^2(\mu)$ as $\eps \downarrow 0$ \cite[Theorem 1.1]{chiarini2022gradient}. In \cite[Theorem 2]{AHMP25}, it is shown that for a large class of $\rho \in \cP_2(\mathbb{R}^{d})$ the next order limit is $\nabla \log a^{\eps} \to \frac{1}{2}\nabla \log \rho$ in $L^2(\rho)$ as $\eps \downarrow 0$. The quantity $\nabla \log \rho$ is called the \textbf{score function}, and it is of interest in machine learning and statistics applications. In Theorem \ref{thm:score-function}, we extend this result to the manifold setting. 
\begin{informaltheorem}[Informal Statement of Theorem \ref{thm:score-function}]
    Let $\pi^{\eps}$ be computed be as in \eqref{eq:intro-sb-defn}, but now we insist that the reference measure is the manifold Wiener measure. Let $a^{\eps}$ be as in \eqref{eq:intro-potent-defn}. Under assumptions, it holds in $L^2(\mu)$ that
    \begin{align}\label{eq:informal-schro-limit}
        \lim\limits_{\eps \downarrow 0} \nabla_g \log a^{\eps} = \frac{1}{2}\nabla_g \log \left(\frac{d\mu}{d\vol}\right).
    \end{align}
    In the case of Hessian manifold (i.e.\ a global chart with $g = \nabla^2 \varphi$), 
\begin{align}\label{eq:hessian-manifold-bp}
    \lim\limits_{\eps \downarrow 0} \frac{1}{\eps}\left(\Exp{\pi_{\eps}}[Z|X=x] - x\right) \to \frac{1}{2}g^{ij}(x)\frac{\partial}{\partial x^{j}}\left(\log \frac{d\mu}{d\vol} - \frac{1}{2}\log g \right)(x). 
\end{align}
\end{informaltheorem}
Observe that the right hand side of \eqref{eq:informal-schro-limit} is a manifold analogue of the score function. The conditional expectation on the left hand side of \eqref{eq:hessian-manifold-bp} is a manifold analogue of the barycentric projection introduced in \cite{pooladian2022entropic} and can only be defined when the manifold possesses a global chart.

To conclude this paper, we interest ourselves with extending this program of diffusion approximation to the quadratic cost $\Schro$ bridge on $\mathbb{R}^{d}$ with different marginals. To see how this is connected to the rest of the paper, let $e^{-f}$ and $e^{-h}$ be two probability densities on $\rr^d$ and consider the optimal transport problem with quadratic cost between the two, i.e., 
\[
\gamma_{\mathrm{opt}}:=\argmin_{\gamma \in \Pi(e^{-f}, e^{-h})} \frac{1}{2}\int_{\mathbb{R}^{d} \times \mathbb{R}^{d}} \norm{y-x}^2 d\gamma.
\]
By Brenier's theorem \cite{THEbrenier}, there is a unique solution given by a map $\nabla \varphi: \rr^d \rightarrow \rr^d$ that pushforwards $e^{-f}$ to $e^{-h}$. There is one-to-one correspondence between the sets of the couplings $\Pi(e^{-f}, e^{-h})$ and $\Pi(e^{-f}, e^{-f})$, by the identification $\gamma \in \Pi(e^{-f}, e^{-h}) \mapsto (\Id, \nabla \varphi^*)_{\#} \gamma\in \Pi(e^{-f}, e^{-f})$, where $\varphi^*$ is the convex conjugate of $\varphi$. The inverse map is given by the pushforward by the map $(\Id, \nabla \varphi)$. Using this identification, one may express $\gamma_{\mathrm{opt}}$ as the pushforward, via the map $(\Id, \nabla \varphi)$, of the following optimal coupling    
\begin{equation}\label{eq:bregman}
\gamma^*_{\mathrm{opt}}:= \argmin_{\eta \in \Pi(e^{-f}, e^{-f})} \int_{\mathbb{R}^{d} \times \mathbb{R}^{d}} \norm{y- \nabla \varphi(x)}^2 d\eta=\argmin_{\eta \in \Pi(e^{-f}, e^{-f})} \int_{\mathbb{R}^{d} \times \mathbb{R}^{d}} D_{\varphi}[y \mid x] d\eta.
\end{equation}
Here, $D_{\varphi}[y\mid x]$ is the Bregman divergence of the convex function $\varphi$ defined as $D_{\varphi}[y \mid x]= \varphi(y) - \varphi(x) - \nabla \varphi(x)\cdot (y-x)$. Since the only term involving both $x$ and $y$ is $-y\cdot \nabla \varphi(x)$, the second equality in \eqref{eq:bregman} follows. 

Now, $D_{\varphi}[y\mid x]$ is not a Riemannian metric, but, when $y \approx x$, it is close to an actual Riemannian metric which is that of the Hessian manifold (see Section \ref{subsec:geom}) induced by the convex potential $\varphi$. Thus, one may expect that if we consider the Schr\"{o}dinger bridge with this Riemannian metric and a small temperature $\eps$, and then pushforward this bridge by the map $(\Id, \nabla \varphi)$, one will obtain an approximation of the Schr\"odinger bridge corresponding to the Euclidean transport problem between the marginals $e^{-f}$ and $e^{-h}$. 

The same marginal problem on the Hessian manifold is approximated by the two point distribution of the so-called Mirror Langevin diffusion. Then, we apply the  transformation $(\Id, \nabla \varphi)$ to the two marginals of the Mirror Langevin diffusion to approximate the $\Schro$ bridge on $\mathbb{R}^{d}$ with different marginals. 

\begin{informaltheorem}[Informal Statement of Theorem \ref{thm:mld-quad-sb-comparison}]
    Let $e^{-f}, e^{-h} \in \cP_2(\mathbb{R}^{d})$ with $\nabla \varphi: \mathbb{R}^{d} \to \mathbb{R}^{d}$ the Brenier map from $e^{-f}$ to $e^{-h}$. Let $(X_t, t \geq 0)$ denote the primal Mirror Langevin diffusion released from stationarity, defined in \eqref{eq:MLD}. Set $\bar{\ell}_{\eps} = \mathrm{Law}(X_0,\nabla \varphi(X_{\eps})) \in \Pi(e^{-f},e^{-h})$, and let $\Pi_{\eps}$ denote the $\eps$-static Euclidean $\Schro$ bridge from $e^{-f}$ to $e^{-h}$. Then under assumptions it holds that
    \begin{align*}
        \lim\limits_{\eps \downarrow 0} H(\bar{\ell}_{\eps}|\Pi_{\eps}) = 0.
    \end{align*}
\end{informaltheorem}
Our approach to proving Theorem \ref{thm:mld-quad-sb-comparison} rests on a comparison of the heat kernel of the Hessian manifold \cite{shima2007geometry} to the Euclidean heat kernel. In this comparison, we use small time manifold heat kernel asymptotics. A technical impediment of our proof is that we cannot establish a quantitative rate of convergence in $\eps$. That being said, given a more detailed heat kernel asymptotic (perhaps one proven specifically for the Hessian manifold), it may be possible to quantify this convergence. We make the following conjecture. For a rationale, see Remark \ref{remark:justification-conjecture} in Section \ref{sec:diff-marg}.
\begin{conjecture}\label{conjecture:O-eps}
    Under assumptions, perhaps slightly stronger than those of Theorem \ref{thm:mld-quad-sb-comparison},
    \begin{align}\label{eq:conj-o-eps}
        H(\bar{\ell}_{\eps}|\Pi_{\eps}) = O(\eps).
    \end{align}
    And in general, this quantity cannot be $o(\eps)$. Using the tools in this paper, as a corollary
    \begin{align}\label{eq:bp-conv}
        \int_{\mathbb{R}^{d}} \norm{\Exp{\Pi^{\eps}}[Y|X=x]-\nabla \varphi(x)}^2 e^{-f(x)}dx = O(\eps^2).  
    \end{align}
\end{conjecture}

The approximation developed in Theorem \ref{thm:mld-quad-sb-comparison} is two orders of magnitude weaker than that of the same marginal case. Our belief is that this gap is due to the curvature of the Hessian manifold. In Section \ref{subsec:affine} in the case of affine Brenier map, i.e.\ when the Hessian metric is constant, we recover an $o(\eps)$ rate in the diffusion approximation for the different marginal case in Proposition \ref{prop:affine-brenier}.

Results such as \cite[Theorem 1.1]{chiarini2022gradient} and \cite[Corollary 1]{pooladian2022entropic} establish, in the case of different marginals, the convergence of the rescaled gradients of $\Schro$ potentials, $(\eps \nabla_g \log a^{\eps}, \eps > 0)$, or the barycentric projection, to the gradient of the Kantorovich potential or Brenier map between the marginals as $\eps \downarrow 0$. The conjecture in \eqref{eq:bp-conv} gives a rate of convergence for noncompactly supported densities. For compactly supported densities, a rate of $O(\eps^2)$ is established in \cite[Corollary 1]{pooladian2022entropic}.

As a concluding remark, this paper contains several minor results that extend or modify standard facts about diffusion processes and $\Schro$ bridges to the geometric settings we consider, particularly manifolds possessing a global chart (of which Hessian manifolds are our main example of interest). Many of these results are presented in Section \ref{sec:preliminaries}, and in lieu of complete proofs that are mostly redundant, we only remark on the necessary modifications to the pre-existing literature to encompass our setting. 

\subsection{Geometric settings of interest.}\label{subsec:geom}
To motivate the generalized setting considered in this paper, we present several applications in which either the manifold setting or non-Wiener reference measure is of interest. 

In statistical data modeling it is often a natural assumption that data come from some manifold that is not Euclidean. For instance, \cite{shi2023diffusion} computes $\Schro$ bridges on $\mathbb{S}^{2}$ for use in climate modeling. The Riemannian score function appears in applications such as \cite{debortoli2022riemannian,rygaard2025score}. In computational biology it is of interest to compute $\Schro$ bridges with respect to a general diffusion reference to perform trajectory inference \cite{chizat2022trajectory,lavenant-traj-inf,yks-25}. 

This paper is especially motivated by the Hessian manifold, i.e.\ when the metric $g$ is the Hessian of a smooth convex function. For an introduction to Hessian manifolds, consult \cite[Chapters 1,6]{amari-info-geo}. Our consideration of Hessian manifolds is motivated by approximating the different marginal Euclidean $\Schro$ bridge in Theorem \ref{thm:mld-quad-sb-comparison}. However, Hessian manifolds are interesting objects in their own right. Hessian manifolds are real-valued analogues of K\"{a}hler manifolds, a popular object in differential geometry \cite[Section 3]{kolesnikov-hessian-metric}. In convex geometry, Hessian manifolds found applications in \cite{eldan-klartag-11,klartag-moment-maps13} to establish functional inequalities and other estimates.   

To briefly introduce Hessian manifolds, fix a smooth, strictly convex $\varphi: \mathbb{R}^{d} \to \mathbb{R}$ and let $\varphi^*$ denote its convex conjugate. We consider two global charts on the Hessian manifold: its primal coordinates (denoted $x$) and dual coordinates (denoted $x^* := \nabla \varphi(x)$). The metric writes in primal coordinates as $\nabla^2 \varphi$ and in dual coordinates as $\nabla^2 \varphi^*$. In finite dimensional optimization, endowing $\mathbb{R}^{d}$ with a Hessian geometry gives rise to mirror gradient flows, which can improve the convergence rate to the minimizer of a function depending on the choice of $\varphi$. For sampling, \cite{zhang-mld-20a,ahn2021efficient} introduce and analyze an analogous stochastic process called the Mirror Langevin diffusion (MLD). Let $e^{-f},e^{-h} \in \cP_2(\mathbb{R}^{d})$ be such that $(\nabla \varphi)_{\#}e^{-f} = e^{-h}$, then the primal and dual MLD (respectively) are given by the following SDEs, where $Y_t  = \nabla \varphi(X_t)$ and $(B_t, t \geq 0)$ is Euclidean Brownian motion:
\begin{align}\label{eq:MLD}\tag{MLD}
    dX_t = -\frac{1}{2}\nabla h(Y_t)dt+\sqrt{\nabla^2 \varphi^*(Y_t)}dB_t \text{ and } dY_t = -\frac{1}{2}\nabla f(X_t)dt+\sqrt{\nabla^2 \varphi(X_t)}dB_t.
\end{align}
Note that the stationary measure for the primal SDE is $e^{-f}$, and for the dual SDE it is $e^{-h}$.
In Corollary \ref{cor:non-explosion-mld}, we present a (to our knowledge) novel nonexplosion criterion for \eqref{eq:MLD} modeled on a standard criterion for diffusions on $\mathbb{R}^{d}$ \cite[Theorem 2.2.19]{royer-lsi}. The key geometric assumption is stochastic completeness of the Hessian manifold, i.e.\ that the Brownian motion on the Hessian manifold does not explode. Stochastic completeness is guaranteed when $M$ has a global lower bound on its Ricci curvature (which is our setting), but more involved criteria are known \cite[Section 4.2]{hsu-stoch-analysis-manifold}.

The \eqref{eq:MLD} has an important interpretation stemming from gradient flows, ``curves of steepest descent'', on the space of probability measures (Wasserstein space) \cite{ambrosio2005gradient}. Let $\mu \in \cP(M)$ denote the measure on the Hessian manifold whose Lebesgue density in primal and dual coordinates is equal to $e^{-f}$ and $e^{-h}$, respectively. This is a special case of \cite[Theorem 1.1]{lisini2009nonlinear}, which proves that the time marginal distributions of \eqref{eq:MLD} evolve according to the gradient flow of $H(\cdot|\mu)$ on the Wasserstein space associated to the Hessian manifold.

From an optimal transport perspective, a natural choice of convex function with which to build a Hessian geometry is the Brenier potential \cite{THEbrenier} between two probability distributions on $\mathbb{R}^{d}$. Indeed, \cite{DKPS23} establish a connection between time-rescaled iterates of the Sinkhorn algorithm and a time-inhomogenous generalization of the MLD. This connection between the MLD and (Euclidean) $\Schro$ bridges is further explored in Section \ref{sec:diff-marg}: we quantify in relative entropy the strength of approximation $(X_0,Y_{\eps})$ from \eqref{eq:MLD} (initialized at its stationary distribution) provides to the Euclidean $\Schro$ bridge from $e^{-f}$ to $e^{-h}$. 

\subsection{Outline of Paper.}
In Section \ref{sec:preliminaries}, we introduce the necessary technical preliminaries for diffusion processes and $\Schro$ bridges on manifolds. A key point of departure from the Euclidean case considered in \cite{AHMP25} is that we also must concern ourselves with the small time heat kernel expansion on a manifold-- see Proposition \ref{prop:heat-kernel-asymp}. In Section \ref{sec:diff-approx} we quantify in symmetric relative entropy a diffusion approximation to the same marginal $\Schro$ bridge in Theorem \ref{thm:sym-rel-ent}. This approximation is leveraged in Section \ref{sec:generator-analysis} to prove an ``approximate'' generator result for the family of conditional measures arising from the same marginal $\Schro$ bridge with varying temperature (Theorem \ref{thm:generator-transfer}). Additionally, Theorem \ref{thm:score-function} establishes the convergence of the gradient of $\Schro$ potentials in a slightly more restricted setting. In Section \ref{sec:diff-marg}, Theorem \ref{thm:mld-quad-sb-comparison} quantifies in relative entropy a diffusion approximation to the different marginal Euclidean $\Schro$ bridge.

\section{Preliminaries}\label{sec:preliminaries}
We now summarize the fundamental geometric and probabilistic notions required to state and prove the results of this paper. Additionally, we state several auxiliary results that are used in the following sections. Many proofs are relegated to the Appendix.

\subsection{Geometric Setting}
Recall that $(M,g)$ denotes a smooth ($C^{\infty}$), connected, complete Riemannian manifold without boundary, equipped with the Levi-Civita connection denoted $\nabla_g$. In a fixed coordinate system, we write the components functions of the metric $g$ as $(g_{ij}(x))$. As the co-metric $g^{-1}$ appears in many calculations, we write its components with upper indices $(g^{ij}(x))$. Let $d: M \times M \to [0,+\infty)$ denote the distance on $M$ induced by the metric. For $x \in M$, $T_{x}M$ denotes the tangent space of $M$ at $x$, $TM$ denotes the tangent bundle, and $\ricci_{x}: T_{x}M \times T_{x}M \to \mathbb{R}$ denotes the Ricci curvature tensor at $x$. Let $(\Gamma_{ij}^{k}(x))$ denote the Christoffel symbols for the Levi-Civita connection. In particular, for a fixed coordinate system $(x^1,\dots,x^d)$ recall that
\begin{align*}
    \nabla_{\frac{\partial}{\partial x^i}} \frac{\partial}{\partial x^{j}} = \Gamma_{ij}^{k} \frac{\partial}{\partial x^k},
\end{align*}
where the above sum is written using Einstein summation notation (as we will use for the rest of this paper).

Let $(B_t^{M}, t \geq 0)$ denote Brownian motion on $M$, and let $U \in C^2(M)$ be a potential (on which we will eventually make further assumptions). Consider the following manifold-valued SDE
\begin{equation}\label{eq:reference-process-intro}
    dX_t = -\frac{1}{2}\nabla_g U(X_t) dt + dB_t^{M}. 
\end{equation}
Observe that the above process (under suitable assumptions) is reversible with stationary measure given by $e^{-U(x)}\vol(dx)$, where $\vol(dx)$ is the volume measure measure on $M$. Recall in coordinates that
\begin{align}\label{eq:vol-measure-coords}
    \vol(dx) = \sqrt{\det g(x)}dx.
\end{align}
We emphasize that $e^{-U(x)}\vol(dx)$ need not always be a probability measure. This is the case when $U = 0$ and $M$ is, for instance, $\mathbb{R}^{d}$ with $\vol = \mathrm{Leb}$.

We consider two diffusion processes, which we will distinguish as the \textbf{reference} and the \textbf{target}. The \textbf{reference process} is the solution to the SDE \eqref{eq:reference-process-intro} with drift equal to $-\frac{1}{2}\nabla_{g} U_{\diffref}$, for some $U_{\diffref} \in C^2(M)$. We write the stationary measure as $\diffref(dx) := \exp(-U_{\diffref})\vol(dx)$. We use the name ``reference process'' as its law provides the reference measure with respect to which the $\Schro$ bridge will be computed. In this paper, the applications developed in Sections \ref{sec:generator-analysis} and \ref{sec:diff-marg} take $U_{\diffref} \equiv 0$, that is, $\diffref = \vol$. However, it is of interest to consider more general reference process in both theoretical \cite{fathigozlan20,conforti21deriv,chiarini2022gradient} and applied \cite{lavenant-traj-inf,yks-25} settings. Thus, in Section \ref{sec:diff-approx} we take a general reversible diffusion as reference process. 

The second process we consider is the \textbf{target diffusion}, which \emph{must have an integrable stationary measure}. We then normalize this stationary measure and insist it is a \textbf{probability} measure. Given some $\mu \in \cP(M)$ with $\mu = \exp(-U_{\mu})\vol$, the target diffusion is the solution to \eqref{eq:reference-process-intro} with drift equal to $-\frac{1}{2}\nabla_{g} U_{\mu}$. 

The following definition is standard and was developed in \cite{lott-villani-09,strum-mms-06}.
\begin{definition}[Curvature Dimension Condition]
    For $\kappa \in \mathbb{R}$ and $N \in [d,+\infty]$, the triplet $(M,g,e^{-U}\vol)$ satisfies
$\text{CD}(\kappa,N)$ if the following inequality holds for all $x \in M$ 
\begin{align}\label{eq:cd-defn}\tag{CD($\kappa,N$)}
    \ricci + \mathrm{Hess}(U) \geq \kappa g + \frac{1}{(N-d)}\nabla_{g} U \otimes \nabla_{g} U. 
\end{align}
\end{definition}
As an example, suppose that $\ricci \geq \kappa g$, then the manifold Brownian motion on $M$ ($U = 0$) satisfies $\text{CD}(\kappa,d)$. On the other hand, the standard Gaussian on $\mathbb{R}^{d}$ satisfies $\text{CD}(1,+\infty)$. Similarly, \cite[Theorem 4.3]{kolesnikov-hessian-metric} gives a condition under which the Mirror Langevin diffusion satisfies $\text{CD}(C,+\infty)$ for some $C \in \mathbb{R}$. For a comprehensive reference on curvature dimension conditions and their interactions with Markov processes, we refer readers to the excellent monograph \cite{bgl-markov} (although beware of the different constant conventions used). 

We now state the assumptions on our geometric setting. 
\begin{assumption}\label{assumption:manifold}
    Let $(M,g)$ be a $d$-dimensional smooth, connected, complete Riemannian manifold without boundary with a global lower bound on its Ricci tensor, meaning that there is some $\kappa \in \mathbb{R}$ such that for all $x \in M$ and $v \in T_x M$, $\ricci_{x}(v,v) \geq \kappa g(v,v)$.
    
    In addition, we assume one of the following settings:
    \begin{itemize}
        \item[(H1)] $(M,g)$ is compact. 
        \item[(H2)] $(M,g)$ possesses a global coordinate chart, and in such a fixed chart there exists $\alpha,\beta > 0$ such that $\alpha \Id \leq g(x) \leq \beta \Id$ for all $x \in M$ and all first derivatives of $g$ are bounded.   
        \end{itemize}
\end{assumption}
Regarding (H2), we will fix the global chart satisfying the specified behavior of the metric and identify $M$ with $\mathbb{R}^{d}$. This is the setting considered in \cite{lisini2009nonlinear}, although our assumptions are more stringent on $g$.  
The main example we have in mind for (H2) is the Hessian manifold.

Next, we delineate the collection of assumptions we make on the reference measure $\diffref = \exp(-U_{\diffref})\vol \in \cM_{+}(M)$ and the target measure $\mu = \exp(-U_{\mu})\vol \in \cP(M)$.
\begin{assumption}\label{assumption:diffusions}
    Let $U_{\diffref}, U_{\mu} \in C^3(M)$ be such that
    \begin{itemize}
        \item $(M,g,\diffref)$ satisfies $\mathrm{CD}(\kappa_{\diffref},N)$ for $\kappa_ \diffref \in \mathbb{R}$ and $N \in [d,+\infty)$ \textbf{OR} $\mathrm{CD}(\kappa_{\diffref},+\infty)$ holds with $\diffref(M) = 1$. 
        \item $(M,g,\mu)$ satisfies $\mathrm{CD}(\kappa_{\mu},+\infty)$ for some $\kappa_{\mu} \in \mathbb{R}$, $\mu \in \cP_2(M)$, and $H(\mu|\diffref)$ and $I(\mu|\diffref)$ are finite. 
    \end{itemize}

    Under (H2), additionally assume that $\inf\limits_{M} U_{\diffref} > -\infty$. Let $\cU_{\mu},\cU_{\diffref}$ be as defined in \eqref{eq:reciprocal-characteristic}, then assume that $\inf \limits_{M} \cU_{\mu}, \inf\limits_{M} \cU_{\diffref} > -\infty$. If $\diffref \notin \cP(M)$, then assume that $\nabla_g U_{\diffref}$ is Lipshitz. 
\end{assumption}
We emphasize that $\diffref = \vol$ satisfies the above assumptions. 

\subsection{Manifold-Valued Diffusion Processes}
For a detailed reference on manifold-valued stochastic processes, we refer readers to the excellent text \cite{hsu-stoch-analysis-manifold}. A process $(X_t, t \geq 0)$ on a manifold is called an $\cL$-diffusion process, for some smooth second order elliptic operator $\cL$, if for all $f \in C^{\infty}(M)$ the process 
\begin{align}
    M^{f}(X)_t := f(X_t) - f(X_0) - \int_0^{t} \cL f(X_s)ds, \quad t \geq 0 
\end{align}
is a local martingale with respect to the natural filtration.  

\subsubsection{Manifold Brownian Motion and Heat Kernel Expansion}
A process $(B_t, t \geq 0)$ on $M$, given some initial distribution $\mu \in \cP(M)$, is called a \textbf{(manifold) Brownian motion} if it is a $\frac{1}{2}\Delta_g$-diffusion process \cite[Proposition 3.2.1]{hsu-stoch-analysis-manifold}. The law of this process on the path space $C([0,\infty),M)$ is called the \textbf{Wiener measure} starting from $\mu$. Here, $\Delta_g$ is the Laplace-Beltrami operator on $M$, expressed in local coordinates for $u \in C_c^{\infty}(M)$ as
\begin{align}\label{eq:laplace-beltrami-coords}
    \Delta_g u &= \frac{1}{\sqrt{\det g}}\frac{\partial}{\partial x^i} \left(\left(\sqrt{\det g}\right) g^{ij} \frac{\partial}{\partial x^j} u\right) = g^{ij}\frac{\partial^2}{\partial x^i\partial x^j}u-g^{ij}\Gamma_{ij}^{k}\frac{\partial}{\partial x^k}u.
\end{align}

When the Ricci curvature of $M$ has a global constant lower bound, the manifold Brownian motion exists and has a.s.\ infinite explosion time \cite[Theorem 3.5.3]{hsu-stoch-analysis-manifold}. Thus, under Assumption \ref{assumption:manifold}, the manifold Brownian motion does not explode.

In a local coordinate system, the manifold Brownian motion admits the following expression \cite[equation (3.3.11)]{hsu-stoch-analysis-manifold}
\begin{align}\label{eq:local-mani-bm}
    dB_t = \sqrt{g^{-1}(B_t)}dW_t - \frac{1}{2}g^{ij}(B_t)\Gamma_{ij}^{k}(B_t)dt,
\end{align}
where $(W_t, t \geq 0)$ is a standard $d$-dimensional Euclidean Brownian motion. The manifold Brownian motion is a reversible process with stationary measure equal to $\vol$. 

Let $(r_{t}(\cdot,\cdot), t \geq 0)$ denote fundamental solution to the heat equation $\partial_t - \frac{1}{2}\Delta_g = 0$, then $(r_{t}(\cdot,\cdot), t \geq 0)$ is the transition density of Brownian motion on $M$ \cite[Theorem 4.1.4, Proposition 4.1.6]{hsu-stoch-analysis-manifold}. We emphasize here that each $r_{t}(x,\cdot)$ is the Radon-Nikodym derivative of the law of $B^{M}_{\eps}|B^{M}_0 = x$ with respect to the volume measure. That is, in the case of a.s.\ infinite explosion time, it holds for all $t > 0$ that
\begin{align*}
    \int_M r_{t}(x,y)\vol(dy) = 1. 
\end{align*}
A key technical tool of which we make frequent use is the following small time asymptotic expansion of $(r_t(\cdot,\cdot), t > 0)$, a subject area in its own right with a rich history. Historically, this expansion is attributed to the efforts of \cite{varadhan-diff-ldp67,molchanov-exp75,benarous-expan-88,azencott84} and many others. For clarity, we present Proposition \ref{prop:heat-kernel-asymp} with the notation developed in \cite[Theorem 1.4]{neel2025uniformlocalizedasymptoticssubriemannian}, which works in significantly greater generality than necessary for our purposes (and on a different time scale). We collate the results and some of its consequences that we will make use of in the following proposition.

For each $x \in M$, let $C_x \subset M$ denote the cut locus of $x$ \cite[page 308]{riemannian-manifolds-lee}. It holds that $z \in C_x$ if and only if $x \in C_z$. Define the following ``exceptional'' set of points on the product manifold $M \times M$ as $\cC := \{(x,z) \in M \times M: x \in C_z\}$.

\begin{proposition}\label{prop:heat-kernel-asymp}
    On a complete Riemannian manifold $(M,g)$, the following expansion holds for some choice of smooth functions $c_0, R(\eps,\cdot,\cdot): (M \times M) \setminus \cC \to \mathbb{R}$, 
    \begin{align}\label{eq:heat-kernel-expan}
        r_{\eps}(x,z) &= \frac{1}{(2\pi \eps)^{d/2}}\exp\left(-\frac{1}{2\eps}d_g^2(x,z)\right)\left(c_0(x,z)+\eps R(\eps,x,z)\right).
    \end{align}
    Additionally, it holds that $\left(\nabla_{x} \log c_{0}(x,z)\right)|_{z=x}=\lim\limits_{z \to x} \nabla_{x} \log c_{0}(x,z) = 0 \in T_{x}M$. 

    We also have a uniform bound in the following sense. Let $K \subset (M \times M) \setminus \cC$ be a compact set. Then, for $(x,z) \in K$ it holds that $c_{0}(x,z) > 0$, and for any $\eps > 0$ 
    \[
     \sup\limits_{t \in (0,2\eps)}\sup \limits_{(x,z) \in K} \abs{R(t,x,z)}, \sup\limits_{t \in (0,2\eps)}\sup \limits_{(x,z) \in K} \norm{\nabla_x R(t,x,z)}_{g} < +\infty. 
    \]   
\end{proposition}
Recall that the Riemannian log based at a point $x$, denoted $\log_x$, is the inverse of $\exp_x$ (defined on the subset of $M$ where it exists). Then, $c_{0}(x,z) = [J(\exp_{x})(\log_{x}(z))]^{-1/2}$, i.e.\ the determinant of the Jacobian of $\exp_{x}$ evaluated at $\log_{x}(z)$. Taking the log and gradient in $x$ of \eqref{eq:heat-kernel-expan} gives
\begin{align}\label{eq:grad-log-heat-ker-exp}
    \nabla_{x} \log r_{\eps}(x,z) = \frac{1}{\eps}\log_{x} z + \nabla_{x} \log c_{0}(x,z) + \nabla_{x} \log\left(1+\eps\frac{R(\eps,x,z)}{c_{0}(x,z)}\right).
\end{align}

\subsubsection{Manifold Diffusions}
In this subsection, we discuss existence, uniqueness, and non-explosion of the manifold-valued diffusions of the form \eqref{eq:reference-process-intro}. We also perform some essential computations that we frequently use in the sequel. Many proofs are relegated to the Appendix for readability. Recall from the previous subsection that under Assumption \ref{assumption:manifold}, the manifold Brownian motion exists and has a.s.\ infinite explosion time. 

First, consider Assumption \ref{assumption:manifold} (H1), that is, $(M,g)$ is a compact manifold. In this case, as the potentials $U_{\diffref}, U_{\mu}$ in Assumption \ref{assumption:diffusions} are in $C^3(M)$, the drifts of the target and reference diffusions are Lipshitz. As such, existence, uniqueness, and nonexplosions of these diffusions are immediate.

Next, we consider Assumption \ref{assumption:manifold} (H2). In this case, the most straightforward way to proceed is to identify $M$ with $\mathbb{R}^{d}$ and use standard SDE results. In this global chart, a manifold diffusion of the form \eqref{eq:reference-process-intro} writes as
\begin{align}\label{eq:manifold-diff-global-chart}
    dX_t = \left( -\frac{1}{2}g^{kj}(X_t) \frac{\partial}{\partial x^{j}}U(X_t) - \frac{1}{2}g^{ij}(X_t)\Gamma_{ij}^{k}(X_t)\right)dt + \sqrt{g^{-1}(X_t)}dB_t,
\end{align}
where $(B_t, t \geq 0)$ is a standard $d$-dimensional Euclidean BM. This can be seen by writing the generator for \eqref{eq:reference-process-intro} in coordinates and recalling that $\Delta_g$ writes as \eqref{eq:laplace-beltrami-coords}. Under Assumption \ref{assumption:manifold} (H2), the diffusion matrix is bounded and Lipshitz, and $x \mapsto \frac{1}{2}g^{ij}(x)\Gamma_{ij}^{k}(x)$ is bounded. When $\diffref \notin \cP(M)$, under Assumption \ref{assumption:diffusions}, $\nabla_g U_{\diffref}$ is Lipshitz. As the manifold distance is equivalent to the Euclidean distance, it holds that the drift of \eqref{eq:manifold-diff-global-chart} has at most linear growth. By \cite[Theorem 5.2.9]{karatshreve91}, existence, uniqueness, and nonexplosion of the reference process holds.

In the absence of a Lipshitz drift, we present the following non-explosion criterion under (H2) that is modeled after the criterion of \cite[Theorem 2.2.19]{royer-lsi}. Thus, Assumption \ref{assumption:diffusions} guarantees existence, uniqueness and non-explosion of the target diffusion as well as the reference diffusion when it has integrable reversible measure.  
\begin{proposition}\label{prop:ex-uniq-nonexp-mld}
    Let $(M,g)$ be a Riemannian manifold satisfying Assumption \ref{assumption:manifold} (H2), and let $U: M \to \mathbb{R}$ satisfy the following condition, where $x_0 \in M$ is arbitrary
\begin{equation}\label{assumption:potent}\tag{U}
    U \in C^2(M), U(x) \to +\infty \text{ as $d(x,x_0) \to +\infty$, and } \inf\limits_{x \in M}\left(\frac{1}{4}\norm{\nabla_g U}_{g}^2 - \frac{1}{4}\Delta_g U\right) > -\infty. 
\end{equation}
    Consider the following manifold-valued SDE
    \begin{align}\label{eq:hess-mani-sde-mld}
        dX_t = - \frac{1}{2}\nabla_g U(x) dt + dB_t^{M}, \quad t \geq 0,
    \end{align}
    where $(B_{t}^{M}, t \geq 0)$ is the manifold Brownian motion. Then, a unique weak solution to \eqref{eq:hess-mani-sde-mld} exists and has almost sure infinite explosion time.
\end{proposition}

We translate this result into the Hessian manifold introduced in Section \ref{subsec:geom} to apply \eqref{assumption:potent} to obtain non-explosion criterion for \eqref{eq:MLD}, using only the non-explosion of the Brownian motion on the Hessian manifold and the criterion \eqref{assumption:potent}. 
\begin{corollary}[Non-explosion of \eqref{eq:MLD}]\label{cor:non-explosion-mld}
    Let $(M,g)$ be a smooth Hessian manifold such that $(B_t^{M}, t \geq 0)$ exists, is unique, and has a.s.\ infinite explosion time. Then the Mirror Langevin diffusion \eqref{eq:MLD} may be written as a manifold diffusion as
    \begin{align}\label{eq:mld-manifold-diff}
        dX_t = \frac{1}{2}\nabla_g \log\left(\frac{d\mu}{d\vol}\right)(X_t)dt + dB_t^{M}, 
    \end{align}
    where $\mu \in \cP(M)$ is the manifold measure whose Lebesgue density in primal and dual coordinates is given by $e^{-f}$ and $e^{-h}$, respectively. When the potential $-\log\left(\frac{d\mu}{d\vol}\right)$ satisfies \eqref{assumption:potent}, then \eqref{eq:mld-manifold-diff} and thus \eqref{eq:MLD} has a unique solution with a.s.\ infinite explosion time. 
\end{corollary}

\begin{proof}
    The key computation is to establish that \eqref{eq:mld-manifold-diff} writes in coordinates as \eqref{eq:MLD}, then the claim follows from Proposition \ref{prop:ex-uniq-nonexp-mld}. Recall the shorthand $x^* = \nabla \varphi(x)$. First, $d\mu/d\vol$ writes in primal coordinates as $e^{-f(x)}/\sqrt{\det \nabla^2 \varphi(x)}$ and in dual coordinates as $e^{-h(x^*)}/\sqrt{\det \nabla^2 \varphi^*(x^*)}$. By change of variables, it holds that
    \begin{align}
        h(x^*) = f(x) + \log \det \nabla^2 \varphi(x). 
    \end{align}
    Computing $\partial/\partial x^i$ on each side gives
    \begin{align}\label{eq:grad-log-cov}
        \partial^2_{ij}\varphi(x)\frac{\partial}{\partial (x^*)^j}h(x^*) = \frac{\partial}{\partial x^i} f(x) + \partial^2_{jk}\varphi^*(x^*)\partial^{3}_{jki}\varphi(x). 
    \end{align}
    In primal coordinates, the Christoffel symbols of the Levi-Civita connection write as
    \begin{align}\label{eq:cris-sym-primal}
        \Gamma_{ij}^{k}(x) = \frac{1}{2}\partial^2_{k\ell}\varphi^*(x^*)\partial^{3}_{\ell ij}\varphi(x). 
    \end{align}
    By \eqref{eq:manifold-diff-global-chart}, in primal coordinates the manifold diffusion in \eqref{eq:mld-manifold-diff} has diffusion matrix equal to $\sqrt{\nabla^2 \varphi^*(x^*)}$ and drift equal to 
    \begin{align*}
        &-\frac{1}{2}\partial^2_{k\ell}\varphi^*(x^*)\left(\frac{\partial}{\partial x^\ell}f(x)+\frac{1}{2}\frac{\partial}{\partial x^\ell}\log \det \nabla^2 \varphi(x) \right)- \frac{1}{4}\partial^2_{ij}\varphi^*(x^*)\partial^2_{k\ell}\varphi^*(x^*)\partial^3_{\ell ij}\varphi(x)\\
        &= - \frac{1}{2}\partial^2_{k\ell}\varphi^*(x^*)\left(\frac{\partial}{\partial x^\ell}f(x)+ \partial^2_{ij}\varphi^*(x^*)\partial^3_{\ell ij}\varphi(x)\right) = -\frac{1}{2}\frac{\partial}{\partial (x^*)^{\ell}}h(x^*),
    \end{align*}
    where the first equality follows from the gradient log determinant identity in \eqref{eq:grad-log-cov}, and the last line follows from multiplying both sides of \eqref{eq:grad-log-cov} by $\nabla^2 \varphi^*(x^*) = (\nabla^2 \varphi(x))^{-1}$. Hence, we recover \eqref{eq:MLD} in primal coordinates. The argument is analogous for dual coordinates. 
\end{proof}

The following formula for the Radon-Nikodyn derivative in $C([0,T],M)$ of a diffusion with infinite-explosion time with respect to the Wiener measure is essential and well-known. We give its justification for the setting outlined in Assumption \ref{assumption:manifold} in the Appendix. We emphasize the role of the reciprocal characteristic of the potential $U$, defined as
\begin{align}\label{eq:reciprocal-characteristic}
    \cU(x) &:= \frac{1}{8}\norm{\nabla_{g} U}_{g}^2-\frac{1}{4}\Delta_{g}U.
\end{align}
Observe that when $\inf\limits_{x \in M} \cU(x) > -\infty$, the second condition of \eqref{assumption:potent} holds. 
\begin{proposition}\label{prop:rn-deriv-diffusion}
Let $(M,g)$ be a Riemannian manifold satisfying Assumption \ref{assumption:manifold}. Consider the manifold valued SDE  
\begin{align}\label{eq:mani-diff-sde}
    dX_t = -\frac{1}{2}\nabla_{g} U(X_t)dt + dB^M_t, \quad t \geq 0
\end{align}
and assume that it has infinite explosion time. Let $\cU$ be as defined in \eqref{eq:reciprocal-characteristic}.
Fix $T > 0$ and let $P_x, W_x \in \cP(C([0,T],M))$ denote the laws of the diffusion measure and Wiener measure started from some point $x \in M$, respectively. Then
\begin{align}\label{eq:rn-path-measure}
    \frac{dP_x}{dW_x}(\omega) &= \sqrt{\frac{e^{-U(\omega_T)}}{e^{-U(x)}}}\exp\left(-\int_0^{T} \cU(\omega_t)dt\right).
\end{align}
\end{proposition}

The following calculation is essential. Recall the (relative) Fisher information defined in \eqref{eq:intro-fi-defn}. 
\begin{proposition}\label{prop:cal-u-ip-identities}
    Let $P^i$, $i = 1,2$ be the law on $C([0,T],M)$ of the following SDEs, with $U_i \in C^2(M)$ guaranteeing existence, uniqueness, and nonexplosion
    \begin{align}
        dX_t^i = -\frac{1}{2}\nabla_g U_i(X_t^i) dt + dB_t^M, \quad t \geq 0. 
    \end{align}
    Similarly, let $\cU_i$ denote the reciprocal characteristic \eqref{eq:reciprocal-characteristic} corresponding to $U_i$. Let $\mu_2 \in \cP(M)$ and $\mu_{1} \in \cM_{+}(M)$ denote the stationary measures of $(X_t^i, t\geq 0)$, then when $I_g(\mu_2|\mu_1) < +\infty$ and $(\cU_2-\cU_1) \in L^{1}(\mu_2)$ it holds that
    \begin{align*}
        \int_M (\cU_2 - \cU_1)(x) d\mu_2(x) &= -\frac{1}{8}I_g(\mu_2|\mu_1).
    \end{align*}
    Importantly, setting $U_1 = 0$ gives $\mu_1 = \vol$ and thus $\Exp{\mu_2}[\cU_2] = -\frac{1}{8}I_g(\mu_2|\vol)$.
\end{proposition}
\begin{proof}
This calculation is a simple consequence of integration by parts identity for diffusions \cite[Section 3.4.3, (F7)]{bgl-markov}. First, let $F: M \to \mathbb{R} $ be such that $F \in C^2(M)$. Then by the Chain Rule for Laplace-Beltrami 
\begin{align*}
    (\Delta_g e^{-F})(x) = e^{-F}(\norm{\nabla_g F}_{g}^2-\Delta_g F).
\end{align*}
Let $\cL_i$ denote the generator of $P^i$, that is, for $F \in C^2(M)$ it holds 
\begin{align*}
    \cL_i F = -\frac{1}{2}\langle \nabla_g U_i, \nabla_g F \rangle_g + \frac{1}{2}\Delta_g F. 
\end{align*}
As $U_1, U_2 \in C^2(M)$, the following identities hold in a classical sense. First, observe from the Laplace-Beltrami identity that
\begin{align*}
    \frac{1}{2}\Delta_g e^{\frac{1}{2}(U_1-U_2)} = e^{\frac{1}{2}(U_1-U_2)}\left(\frac{1}{4}\Delta_g (U_1-U_2) + \frac{1}{8}\norm{\nabla_g U_1-\nabla_g U_2}_{g}^2\right).
\end{align*}
Thus, it holds that
\begin{align*}
    \cL_1 e^{\frac{1}{2}(U_1-U_2)}(x) &= e^{\frac{1}{2}(U_1-U_2)}\left(-\frac{1}{4}\langle \nabla_g U_1,\nabla_g U_1-\nabla_g U_2 \rangle_{g} + \frac{1}{4}\Delta_g(U_1-U_2)+\frac{1}{8}\norm{\nabla_g U_1 - \nabla_g U_2}_{g}^2\right) \\
    &= e^{\frac{1}{2}(U_1-U_2)}\left(\cU_2 - \cU_1\right),
\end{align*}
where the second line follows from expanding the inner product and norm squared terms. 

Let $(\chi_{k}, k \geq 1) \subset C_c^{\infty}(M)$ be a sequence of nonnegative functions increasing pointwise to $1$ such that $\norm{\nabla \chi_k}_{\infty} \leq \frac{1}{k}$ for all $k \geq 1$ (such a family exists because $M$ is complete \cite[Proposition C.4.1]{bgl-markov}). Observe that when $M$ is compact we can simply set $\chi_k \equiv 1$. Set $F_k = \chi_k e^{\frac{1}{2}(U_1-U_2)}$ and $H = e^{\frac{1}{2}(U_1-U_2)}$, as each $f_k$ is compactly supported it follows from \cite[Section 3.4.3, (F7)]{bgl-markov} that
\begin{align}\label{eq:ip-with-cutoff}
    -\frac{1}{2}\int_{M} \langle \nabla_g F_k, \nabla_g H \rangle_{g} d\mu_1 &= \int_{M} F_k \cL_1 H d\mu_1.  
\end{align}
Let's send $k\to +\infty$ and analyze the left and right hand sides separately. First, for the left hand side observe that 
\begin{align*}
    \langle \nabla_g F_k, \nabla_g H \rangle &= \chi_{k} \norm{\nabla_g H}_g^2 + \langle \nabla_g \chi_{k}, \nabla_g \log H \rangle_g H^2 \\
    &= \left(\chi_{k} \frac{1}{4}\norm{\nabla_g U_1 - \nabla_g U_2}_g^2 + \frac{1}{2} \langle \nabla_g \chi_k,\nabla_g U_1 - \nabla_g U_2\rangle_g \right)e^{U_1-U_2}.  
\end{align*}
It thus holds that
\begin{align*}
    \int_{M} \langle \nabla_g F_k, \nabla_g H \rangle_{g} d\mu_1 &= \frac{1}{4}\int_{M} \chi_{k}\norm{\nabla_g (U_1 - U_2)}_g^2e^{U_1-U_2}e^{-U_1}d\vol \\
    &+ \frac{1}{2}\int_{M}  \langle \nabla_g \chi_k,\nabla_g (U_1 - U_2)\rangle_{g} e^{U_1-U_2}e^{-U_1}d\vol \\
    &= \frac{1}{4}\int_{M} \chi_{k} \norm{\nabla_g U_1 - \nabla_g U_2}_{g}^2d\mu_{2} + \frac{1}{2} \int_{M}  \langle \nabla_g \chi_k,\nabla_g U_1 - \nabla_g U_2\rangle_g d\mu_2. 
\end{align*}
Now, send $k \to +\infty$. By the Monotone Convergence Theorem on the left term and the Dominated Convergence Theorem on the right (using that $I(\mu_2|\mu_1) < +\infty$ and that the uniform bound on gradients of the $\chi_k$ are vanishing), it holds that
\begin{align*}
    \lim\limits_{k \to \infty} \frac{1}{2}\int_{M}\langle \nabla_g F_k, \nabla_g H \rangle_{g} d\mu_1 &= \frac{1}{8}\int_{M} \norm{\nabla_g U_1 - \nabla_g U_2}_{g}^2d\mu_{2} = \frac{1}{8}I(\mu_2|\mu_1). 
\end{align*}
Consider now the right hand side of \eqref{eq:ip-with-cutoff}. From the previous computation with the $\cL_1$, observe that
\begin{align*}
    F_{k}\cL_1 H &= \chi_k e^{\frac{1}{2}(U_1-U_2)}\cdot e^{\frac{1}{2}(U_1-U_2)}\left(\cU_2 - \cU_1\right) = \chi_k (\cU_2-\cU_1) e^{U_1 - U_2},
\end{align*}
and thus
\begin{align*}
    \int_M F_k \cL_1 H d\mu_1 &= \int_M \chi_k (\cU_2-\cU_1)e^{U_1-U_2}e^{-U_1}d\vol = \int_{M} \chi_k (\cU_2-\cU_1)d\mu_2. 
\end{align*}
Sending $k \to +\infty$, the desired identity follows from the Dominated Convergence Theorem. 
\end{proof}

\subsection{Optimal Transport, Entropic Regularization, and the Role of Curvature}
Let $(M,g)$ satisfy Assumption \ref{assumption:manifold}. For a general treatise of optimal transport and Otto calculus on $\cP_2(M)$, we refer readers to \cite{gigli-ot-manifold} and \cite[Part II]{villani2009optimal}. 

\subsubsection{Dynamic Optimal Transport}
Following \cite[Theorem 1.28]{gigli-ot-manifold}, a curve of measures $(\mu_t, t \in [0,T]) \subset \cP_2(M)$ is called \textbf{absolutely continuous} if and only if there exists a family of vector fields $(v_t, t \in [0,T])$ with $v_t \in L^2(\mu_t)$ and $\norm{v_t}_{L^2(\mu_t)} \in L^1(0,T)$ such that the continuity equation is satisfied in the weak sense
\begin{align}\label{eq:continuity-equation}
    \frac{d}{dt} \mu_t + \Div (\mu_t v_t) = 0.
\end{align}
To be precise, this means that for all $f \in C_c^{\infty}(M)$ it holds
\begin{align}
    \frac{d}{dt}\int_M f d\mu_t &= \int_M \langle \nabla_g f,v_t\rangle_g d\mu_t. 
\end{align}

An important family of curves we will consider are called \textbf{McCann interpolations} \cite{mccann-interp97}, which provide a natural linear interpolation between two measures based on the optimal transport map between them. For instance, if $\mu \in \cP(M)$ then the McCann interpolation from $\mu$ to itself is given by the constant curve $(\mu_{t} \equiv \mu, v_t \equiv 0)$.

\subsubsection{$\Schro$ Bridge}
Let $(X_t, t \geq 0)$ denote the reference process, that is, the stationary solution of the manifold-valued SDE on $M$ 
\begin{align}
    dX_t = -\frac{1}{2}\nabla_g U_{\diffref}(X_t) dt + dB_t^{M}, \quad X_0 \sim \diffref. 
\end{align}
For each $\eps > 0$, let $R^{\eps}$ denote the law on $C([0,1],M)$ of the process $(X_{\eps t}, t \in [0,1])$. In particular, note that $(\omega_t)_{\#}R^{\eps} = \diffref$ for all $t \in [0,1]$. Define $R_{0\eps}:= (\omega_0,\omega_1)_{\#}R^{\eps}$, and set
\begin{align}
    r_{\eps}(x,z) := \frac{dR_{0\eps}}{d(\diffref\otimes \diffref)}(x,z).
\end{align}
Fix $\mu,\nu \in \cP_2(M)$ with $H(\mu|\diffref), H(\nu|\diffref)$ finite. The $\eps$-dynamic $\Schro$ bridge from $\mu$ to $\nu$ is defined as
\begin{align}\label{eq:dynam-schro-bridge-defn}
    P^{\eps} &:= \arg\min \left\{H(P|R^{\eps}): P \in \cP(C([0,1],M)): (\omega_0)_{\#}P=\mu, (\omega_1)_{\#}P = \nu\right\}.
\end{align}
Similarly, the $\eps$-static $\Schro$ bridge from $\mu$ to $\nu$ is defined as
\begin{align}\label{eq:static-schro-bridge-defn}
    \pi^{\eps} &:= \argmin\limits_{\pi \in \Pi(\mu,\nu)} H(\pi|R_{0\eps}).
\end{align}
The quantity $H(\pi^{\eps}|R_{0\eps})$ is called the \textbf{entropic cost}. 
Let $(R^{\eps}_{xy}, x,y \in M)$ denote the laws of the bridges of $R^{\eps}$ on $C([0,1],M)$. Then the static and dynamic $\Schro$ bridges are related in the following manner
\begin{align}\label{eq:stat-dynam-bridge-mix}
    P^{\eps} &= \int_{M \times M} R^{\eps}_{xy}R_{0\eps}(dxdy). 
\end{align}
Assumptions \ref{assumption:manifold} and \ref{assumption:diffusions} on $(M,g)$ and the reference measure $\diffref$ are contained in the assumptions of \cite{chiarini2022gradient}. Hence, by \cite[Proposition 2.2]{chiarini2022gradient}, the $\Schro$ bridge defined in \eqref{eq:dynam-schro-bridge-defn} and \eqref{eq:static-schro-bridge-defn} exists, is unique, and possesses an $(f,g)$-decomposition. That is, there exists $f,g: M \to [0,+\infty)$ measurable such that 
\begin{align}
    \frac{dP^{\eps}}{dR^{\eps}}(\omega) = f(\omega_0)g(\omega_1) \text{ and } \frac{d\pi^{\eps}}{dR_{0\eps}}(x,z) = f(x)g(z).  
\end{align}
Consider the curve in $P_2(M)$ dictated by the marginal flow of the dynamic $\Schro$ bridge, $((\omega_t)_{\#}P^{\eps}, t \in [0,1]) \subset \cP_2(M)$. This curve is in fact an absolutely continuous curve in $\cP_2(M)$ called the $\eps$-\textbf{entropic interpolation} from $\mu$ to $\nu$. This curve is a solution to the following entropic analogue of the celebrated Benamou-Brenier dynamic formulation of optimal transport \cite{benamou2000computational}. Fix $\eps > 0$ and define the following action functional on absolutely continuous curves $(\mu_t, v_t, t \in [0,1])$
\begin{align}\label{eq:action-functional}
    A_{\eps}((\mu_t,v_t)) = \frac{1}{2}\int_0^{1} \norm{v_t}_{L^{2}(\mu_t)}^2dt + \frac{\eps^2}{8}\int_0^{1} I_g(\mu_t|\diffref) dt. 
\end{align}
\begin{proposition}[Entropic Benamou-Brenier]\label{prop:entropic-bb}
    Let $(M,g)$ satisfy Assumption \ref{assumption:manifold} and $\diffref$ satisfy Assumption \ref{assumption:diffusions}. Let $\mu, \nu \in \cP_2(M)$ with $H(\mu|\diffref)$, $H(\nu|\diffref)$ finite (so that $P^{\eps}$ exist). Then the $\eps$-entropic interpolation from $\mu$ to $\nu$ minimizes $A_{\eps}$ defined in \eqref{eq:action-functional} over a class of absolutely continuous curves from $\mu$ to $\nu$. Let $(\mu_t^{\eps},v_t^{\eps})$ denote the entropic interpolation, then the entropic cost can be written as
    \begin{align}\label{eq:ent-cost-ent-interp}
        \eps H(\pi^{\eps}|R_{0\eps}) &= \frac{\eps}{2}(\Ent(\mu|\diffref)+\Ent(\nu|\diffref))+\frac{1}{2}\int_0^{1} \norm{v_t^{\eps}}_{L^{2}(\mu_t^{\eps})}^2dt + \frac{\eps^2}{8}\int_0^{1} I_g(\mu_t^{\eps}|\diffref)dt. 
    \end{align}
\end{proposition}
We provide more detail in the Appendix, but broadly the justification for Proposition \ref{prop:entropic-bb} is the following. Under Assumption \ref{assumption:manifold} (H1), this is essentially established in \cite[Theorem 4.1]{gigli-bb-ent-rcd}. In Euclidean space, this is \cite[Corollary 5.8]{gentil2017analogy}. To our knowledge, this has not been previously established in the full generality of Assumption \ref{assumption:manifold}, but this can be achieved with benign modifications to the proofs in \cite[Section 5]{gentil2017analogy}. We outline the necessary modifications in the Appendix. 

We conclude with the following proposition summarizing important remarks about the Fisher information along the entropic interpolation when $\eps$ varies.
\begin{proposition}\label{prop:the-one-about-fi}
    Let $(\mu_{t}^{\eps}, t \in [0,1])$ denote the $\eps$-entropic interpolation from $\mu$ to $\mu$. Recall the setting in Assumptions \ref{assumption:manifold} and \ref{assumption:diffusions}. Under (H1) it holds that
    \begin{align}\label{eq:cont-int-fisher}
        \lim\limits_{\eps \downarrow 0} \int_{0}^{1} I_g(\mu_{t}^{\eps}|\diffref) dt = I_g(\mu|\diffref), 
    \end{align}
    and if in addition $U_{\diffref} \in C^{\infty}(M)$ then
    \begin{align}\label{eq:int-ent-inter-min}
        I_g(\mu_{1/2}^{\eps}|\diffref) &= \inf\limits_{t \in [0,1]} I_g(\mu_t^{\eps}|\diffref).
    \end{align}
    Under (H2), \eqref{eq:cont-int-fisher} holds under the additional assumption that either (i) $\diffref(M) = 1$ or (ii) $U_{\diffref}: M \to \mathbb{R}$ is bounded.
\end{proposition}

\begin{remark}
    On $(\mathbb{R}^{d},\Id,\diffref)$ with $\cU_{\diffref}$ strictly convex, \cite[Theorem 1.5]{conforti-second-order-sb19} details further conditions that establish the convexity of $t \mapsto I(\mu_{t}^{\eps}|\diffref)$ and gives a lower bound for its second derivative in $t$, providing a basis to quantify the rate of convergence in \eqref{eq:cont-int-fisher}. 
\end{remark}

\begin{proof}
The continuity at $\eps = 0$ of the integrated Fisher information follows from a two-step argument. First, from Proposition \ref{prop:entropic-bb} and suboptimality of the constant curve $(\mu_t \equiv \mu, v_t \equiv 0, t \in [0,1]) \subset \cP_2(M)$ it holds 
\begin{align}\label{eq:int-fi-limsup}
    \frac{1}{2}\int_{0}^{1}\norm{v_t^{\eps}}_{L^{2}(\mu_t^{\eps})}^2dt + \frac{\eps^2}{8}\int_{0}^{1} I_g(\mu_{t}^{\eps}|\diffref)dt \leq \frac{\eps^2}{8} I_g(\mu|\diffref) \Rightarrow \limsup\limits_{\eps \downarrow 0} \int_0^{1} I_g(\mu_{t}^{\eps}|\diffref)dt \leq I_g(\mu|\diffref). 
\end{align}
To supply the matching lower bound, it suffices to establish lower semicontinuity of the Fisher information $I(\cdot|\diffref)$ with respect to weak convergence in $\cP(M)$ and then apply Fatou's lemma. Under (H1), the argument presented in the proof of \cite[Theorem 1.6]{conforti21deriv} establishes the desired lsc. When $\diffref(M) = 1$ (under either (H1) or (H2)) Fisher information is lsc with respect to weak convergence on level sets of $H(\cdot|\diffref)$ in $\cP(M)$ from \cite[Corollary 13]{gigli-mms-heat-flow-09}. Under (H2) with $\diffref \notin \cP(M)$, from \cite[Proposition 9.7]{ags-calc-heat-flow} lsc holds along $\Was{2}$ convergent sequences in $\cP_2(M)$ contained in a level set of $H(\cdot|\diffref)$. 

Regarding weak convergence, let $(R_x^{\eps}, \eps > 0)$ denote the laws on $C([0,1],M)$ of the reference diffusion started from $x \in M$, and recall the large deviation rate function for this family given by Proposition \ref{prop:ldp-rate-diff}. As $M$ is complete, by the Hopf-Rinow Theorem $\frac{1}{2}d^2(x,\cdot)$ has compact level sets for all $x \in M$ and thus is coercive, so by \cite[Proposition 3.7(3)]{leo-sb-to-kp12} it holds that the dynamic $\Schro$ bridge from $\mu$ to itself converges weakly to the dynamic OT from $\mu$ to itself with cost function $\frac{1}{2}d^2$. By the continuous mapping theorem, $\mu_t^{\eps}$ converges weakly to $\mu$ as $\eps \downarrow 0$. 

Next, we show in the case of (H2) with $U_{\diffref}$ bounded and $\cU_{\diffref}$ lower bounded that for each $t \in (0,1)$ it holds that $(\mu_{t}^{\eps}, \eps > 0)$ converges in $\Was{2}$ to $\mu$ \cite[Proposition 9.7, Section 7]{ags-calc-heat-flow}. Fix some $o \in M$. By the weak convergence shown above, from the Portmanteau lemma
\begin{align*}
    \int_{M} d^2(x,o) \mu(dx) \leq \liminf\limits_{\eps \downarrow 0} \int_M d^2(x,o) \mu_t^{\eps}(dx). 
\end{align*}
To establish the matching upper bound, fix $\delta > 0$. There exists $C_\delta > 0$ such that
\begin{align*}
    \int_M d^2(x,o) \mu_t^{\eps}(dx) &= \int_{M \times M} \Exp{R_{u,v}^{\eps}}[d^2(\omega_t,o)] \pi^{\eps}(dudv) \\
    &\leq (1+\delta)\int_M d^2(x,o)\mu(dx) + C_{\delta}\int_{M \times M} \Exp{R_{u,v}^{\eps}}[d^2(\omega_t,u)] \pi^{\eps}(dudv).
\end{align*}
We now show that the rightmost integral vanishes as $\eps \downarrow 0$. Let $(W_{u,v}^{\eps}, \eps > 0,u,v \in M) \subset \cP(C([0,1],M))$ denote the laws of the $\eps$-Brownian bridges on $M$. Proposition \ref{prop:moment-bdd-brownian-bridge} establishes that
\begin{align}\label{eq:brownian-bridge-second-moment}
    \Exp{W_{u,v}^{\eps}}[d^2(\omega_t,u)] &\leq C\left(d^2(u,v)+\eps\right).
\end{align}
Now, by Proposition \ref{prop:rn-deriv-diffusion} and the assumption on $U_{\diffref}$ and $\cU_{\diffref}$, there is some $K > 0$ such that for all $\eps < 1$, $R_{u,v}^{\eps} \leq K W_{u,v}^{\eps}$. With \eqref{eq:brownian-bridge-second-moment} then,
\begin{align*}
    \int_{M \times M} \Exp{R_{u,v}^{\eps}}[d^2(\omega_t,u)] \pi^{\eps}(dudv) \leq CK \int_{M \times M} \left(d^2(u,v)+\eps\right)\pi^{\eps}(dudv).
\end{align*}
As $\pi^{\eps}$ converges in $\Was{2}$ to $(\Id,\Id)_{\#}\mu$ (as $\mu \in \cP_2(M)$), this integral vanishes as $\eps \downarrow 0$, giving
\begin{align*}
    \limsup\limits_{\eps \downarrow 0} \int_M d^2(x,o) \mu_t^{\eps}(dx) \leq (1+\delta) \int_{M} d^2(x,o)\mu(dx). 
\end{align*}
As $\delta > 0$ is arbitrary, this provides the matching upper bound and establishes the $\Was{2}$ convergence of $(\mu_t^{\eps}, \eps > 0)$ to $\mu$ as $\eps \downarrow 0$.

To conclude the lsc of Fisher information argument, we must show that $(\mu_t^{\eps}, \eps > 0)$ is in a level set of $H(\cdot|\diffref)$. By \cite[Theorem 2.4]{leonard2014some}, it holds that
\begin{align*}
    H(P^{\eps}|R^{\eps}) = H((\omega_t)_{\#}P^{\eps}|(\omega_t)_{\#}R^{\eps}) + \Exp{\mu_{t}^{\eps}}\left[H(P^{\eps}(\cdot|\omega_t)|R^{\eps}(\cdot|\omega_t))\right].
\end{align*}
As $P^{\eps}(\cdot|\omega_t), R^{\eps}(\cdot|\omega_t) \in \cP(C([0,1],M))$ (note that this holds even when $R^{\eps}$ is not a finite measure, see \cite[Remark 2.5(b)]{leonard2014some}), the second quantity on the right hand side is nonnegative. Thus, it holds that
\begin{align*}
    H(\mu_t^{\eps}|\diffref) = H((\omega_t)_{\#}P^{\eps}|(\omega_t)_{\#}R^{\eps}) \leq H(P^{\eps}|R^{\eps}) = H(\pi^{\eps}|R_{0\eps}). 
\end{align*}
Define $\ell_{\eps} := \mathrm{Law}(X_0,X_{\eps})$, where $(X_t, t \geq 0)$ is the stationary solution to the manifold diffusion \eqref{eq:reference-process-intro} with drift equal to $-\frac{1}{2}\nabla_g U_{\mu}$. Observe that $\ell_{\eps} \in \Pi(\mu,\mu)$. By the optimality of $\pi^{\eps}$ among all couplings of $\mu$ with itself, it holds from the future computation in \eqref{eq:static-ell-eps-mu-reps} that
\begin{align*}
    H(\pi^{\eps}|R_{0\eps}) \leq H(\ell_{\eps}|R_{0\eps}) = H(\mu|\diffref) + H(\ell_{\eps}|\mu R_{0\eps}) \leq H(\mu|\diffref) + \frac{\eps}{8}I_g(\mu|\diffref).
\end{align*}
Thus, for each fixed $t \in [0,1]$ it holds that $(H(\mu_t^{\eps}|\diffref), \eps \in (0,\eps_0))$ is bounded above by a finite constant for all $\eps_0 > 0$. 

Under (H1) with $U_{\diffref} \in C^{\infty}(M)$, \eqref{eq:int-ent-inter-min} follows from \cite[Lemma 3.2]{glrt-hwi-20} and the remark after \cite[Theorem 1]{AHMP25} on integrated Fisher information.
\end{proof}

\section{Diffusion Approximation to Same Marginal $\Schro$ Bridge}\label{sec:diff-approx}
In this Section we prove Theorem \ref{thm:sym-rel-ent}. Fix the following notation. Fix $\mu \in \cP(M)$ and $\diffref \in \cM_{+}(M)$ that are stationary measures of the following manifold-valued diffusion processes
\begin{align}\label{eq:sec3-target-diff}
    dX_t &= \frac{1}{2}\nabla_g \log\left( \frac{d\mu}{d\vol}\right)(X_t)dt + dB_t^M, \quad X_0 \sim \mu,
\end{align}
\begin{align}\label{eq:sec3-ref-diff}
    dY_t &= \frac{1}{2}\nabla_g \log \left(\frac{d\diffref}{d\vol}\right)(Y_t)dt+dB_t^M,\quad Y_0 \sim \diffref.
\end{align}

Fix some $\eps > 0$, let $Q,R,W \in \cM_{+}(C([0,\eps],M))$ denote the laws of $(X_t, t \in [0,\eps])$, $(Y_t, t \in [0,\eps])$, and $(B_t^M, t \in [0,\eps])$. We will use $Q_x, R_x, W_x \in \cP(C([0,\eps],M))$ to refer to the laws of the above processes when started from a point $x \in M$. Lastly, let $\ell_{\eps} = \text{Law}(X_0,X_{\eps})$ and $R_{0\eps} = \text{Law}(Y_0,Y_{\eps})$. Observe that $\ell_{\eps} \in \cP(M\times M)$ and is a coupling of $\mu$ with itself. The marginals of $R_{0\eps}$ are both equal to $\diffref$. Lastly, we define for convenience,
\begin{align}
    U_\mu :=  -\log\left(\frac{d\mu}{d\vol}\right), \text{ } U_\diffref := -\log\left( \frac{d\diffref}{d\vol}\right)
\end{align}
as well as
\begin{align}
    \cU_{\mu} := \frac{1}{8}\norm{\nabla_g U_{\mu}}_g^2 - \frac{1}{4}\Delta_g U_{\mu}, \text{ } \cU_{\diffref} := \frac{1}{8}\norm{\nabla_g U_{\diffref}}_g^2 - \frac{1}{4}\Delta_g U_{\diffref}.
\end{align}

\begin{theorem}\label{thm:sym-rel-ent}
Let $(M,g)$ be a Riemannian manifold that satisfies Assumption \ref{assumption:manifold}. Fix $\mu \in \cP_2(M)$ and $\diffref \in \cM_{+}(M)$ satisfying Assumption \ref{assumption:diffusions}. Let $\pi^{\eps}$ denote the $\eps$-static $\Schro$ bridge from $\mu$ to itself with reference process $R_{0\eps} =\text{Law}(Y_0,Y_{\eps})$, as defined in \eqref{eq:sec3-ref-diff}. Similarly, let $\ell_{\eps} = \text{Law}(X_0,X_{\eps})$ as defined in \eqref{eq:sec3-target-diff}.

Let $\eps > 0$ and suppose that, with $(\mu_t^{\eps}, t \in [0,1])$ denoting the entropic interpolation, 
\begin{equation}\label{eq:cu-integrability}\tag{$\cU$-int}
    C_{\eps} := \int_0^1 \Exp{\mu_t^{\eps}}\norm{\nabla_g (\cU_{\mu}-\cU_{\diffref})}^2_g dt < +\infty,
\end{equation}
then it holds that
    \begin{align}\label{eq:sym-rel-ent}
        H(\pi^{\eps}|\ell_{\eps})+H(\ell_{\eps}|\pi^{\eps}) &\leq \frac{1}{2}\eps^2 C_{\eps}^{1/2}\left(I_{g}(\mu|\diffref)-\int_{0}^{1} I_{g}(\mu_{t}^{\eps}|\diffref)dt\right)^{1/2}.
    \end{align}
In particular, when $\limsup\limits_{\eps \downarrow 0} C_{\eps} < +\infty$ and $\lim\limits_{\eps \downarrow 0} \int_0^{1} I(\mu_t^{\eps}|\diffref)dt = I(\mu|\diffref)$,
\begin{align}\label{eq:o-eps-2-limit}
    \lim\limits_{\eps \downarrow 0} \frac{1}{\eps^2}\left(H(\pi^{\eps}|\ell_{\eps})+H(\ell_{\eps}|\pi^{\eps})\right) = 0.
\end{align}
\end{theorem}

\begin{remark}
We present the following sufficient conditions for \eqref{eq:cu-integrability} to hold.
    \begin{itemize}
        \item Under Assumption \ref{assumption:manifold} (H1) and Assumption \ref{assumption:diffusions}, the gradients of $\cU_{\mu},\cU_{\diffref}$ are bounded. Thus, \eqref{eq:cu-integrability} automatically holds.
        \item Under Assumption \ref{assumption:manifold} (H2), when $\diffref = \vol$ (and thus $U_{\diffref} = \cU_{\diffref} = 0$), a sufficient condition for \eqref{eq:cu-integrability} is that $\mu$ has subexponential tails and $\norm{\nabla_g \cU_{\mu}}_{g}^2$ has polynomial growth. On Euclidean space, this is proven in \cite[Assumption 1, Lemma 2]{AHMP25}. For a proof in the (H2) setting, see Proposition \ref{prop:h2-int-argument} in the Appendix. 
    \end{itemize}
    Lastly, recall that Proposition \ref{prop:the-one-about-fi} presents sufficient conditions guaranteeing the convergence of $\int_0^{1} I(\mu_t^{\eps}|\diffref)dt$ to $I(\mu|\diffref)$. 
\end{remark}

\begin{proof}
As the major technical modifications to the manifold setting have been worked out in Section \ref{sec:preliminaries}, the broad strokes of this proof are simply mutadis mutandis of that of \cite[Theorem 1]{AHMP25}. For $x \in M$, observe by Proposition \ref{prop:rn-deriv-diffusion} that
    \begin{align}\label{eq:path-rn}
        \frac{dQ_x}{dR_x} &= \frac{dQ_{x}/dW_{x}}{dR_{x}/dW_{x}} 
        = \sqrt{\frac{e^{-U_{\mu}+U_{\diffref}}(\omega_{\eps})}{e^{-U_{\mu}+U_{\diffref}}(x)}}
        \exp\left(-\int_0^{\eps} (\cU_\mu-\cU_{\diffref})(\omega_t)dt\right)
    \end{align}
    Hence, we compute that 
    \begin{align}
        H(Q_x|R_x) &= \Exp{Q_x}\left[\frac{1}{2}\log \left(\frac{d\mu}{d\diffref}\right)(\omega_\eps)-\frac{1}{2}\log \left(\frac{d\mu}{d\diffref}\right)(x)- \int_0^{\eps} (\cU_{\mu}-\cU_{\diffref})(\omega_t)dt\right].
    \end{align}
    Take expectation over the initial value with $x \sim \mu$. From the stationarity of $Q$ and the identity in Proposition \ref{prop:cal-u-ip-identities},
    \begin{align*}
        \Exp{\mu}[H(Q_x|R_x)] &= \Exp{Q}\left[\frac{1}{2}\log \left(\frac{d\mu}{d\diffref}\right)(\omega_\eps)-\frac{1}{2}\log \left(\frac{d\mu}{d\diffref}\right)(\omega_0)- \int_0^{\eps} (\cU_{\mu}-\cU_{\diffref})(\omega_t)dt\right] \\
        &= -\int_0^{\eps} \Exp{Q}\left[(\cU_{\mu}-\cU_{\diffref})(\omega_t)\right] dt = \frac{\eps}{8}I_{g}(\mu|\diffref). 
    \end{align*}
    Define the quantity $c(x,z,\eps)$
    \begin{align}
        c(x,z,\eps) &:= -\log\left(\Exp{R_{x}}\left[\exp\left(-\int_0^{\eps}(\cU_{\mu}-\cU_{\diffref})(\omega_t)dt\right)\right]\middle| \omega_{\eps} = z\right)
    \end{align}
    and its corresponding remainder quantity
    \begin{align}
        R(x,z,\eps) &:= -\log\left(\Exp{R_{x}}\left[\exp\left(-\int_0^{\eps}(\cU_{\mu}-\cU_{\diffref})(\omega_t)-(\cU_{\mu}-\cU_{\diffref})(\omega_0) dt\right)\right]\middle| \omega_{\eps} = z\right).
    \end{align}
    From \eqref{eq:path-rn}, it holds that
    \begin{align}
        \frac{\ell_{\eps}(z|x)}{r_{\eps}(z|x)} &= \sqrt{\frac{e^{-U_{\mu}+U_{\diffref}}(z)}{e^{-U_{\mu}+U_{\diffref}}(x)}}\exp(-c(x,z,\eps)),
    \end{align}
    where we use $\ell_{\eps}(z|x)$, $r_{\eps}(z|x)$ to denote conditional densities starting from $x$ of $\ell_{\eps}$ and $R_{0\eps}$, respectively. Let $\mu r_{\eps} \in \cP(M \times M)$ be such that $\mu r_{\eps}(dxdz) = \mu(dx)r_{\eps}(dz|x)$, then
    \begin{align*}
        \frac{d\ell_{\eps}}{d\mu r_{\eps}}(x,z) = \frac{\ell_{\eps}}{r_{\eps}}(z|x)= \sqrt{\frac{e^{-U_{\mu}+U_{\diffref}}(z)}{e^{-U_{\mu}+U_{\diffref}}(x)}}\exp(-c(x,z,\eps)).
    \end{align*}
    Let $R_{\mu}$ denote the law of the reference diffusion process $(Y_t, t \geq 0)$ on $C([0,\eps],M)$ with initial distribution $Y_0 \sim \mu$.
    As $\ell_{\eps} \in \Pi(\mu,\mu)$ and $H(\mu|\diffref) = \Exp{\mu}[-U_{\mu}+U_{\diffref}]$ is finite, by the factorization of relative entropy and information processing inequality,
    \begin{align}\label{eq:static-ell-eps-mu-reps}
         \Exp{\ell_{\eps}}\left[-c(x,z,\eps)\right] = H(\ell_{\eps}|\mu r_{\eps}) \leq H(Q|R_{\mu}) = \Exp{\mu}[H(Q_x|R_x)] = \frac{\eps}{8}I_{g}(\mu|\diffref).
    \end{align}
    Altogether then, observe from the definition of $c(x,z,\eps)$, $R(x,z,\eps)$, and Proposition \ref{prop:cal-u-ip-identities} 
    \begin{align}\label{eq:remainder-bound}
        (\Exp{\ell_{\eps}}-\Exp{\pi^{\eps}})[-c(x,z,\eps)] &\leq \frac{\eps}{8}I_g(\mu|\diffref)+\eps \Exp{\mu}[\cU_{\mu}-\cU_{\diffref}] + \Exp{\pi^{\eps}}[R(x,z,\eps)] = \Exp{\pi^{\eps}}[R(x,z,\eps)].
    \end{align}
    Recall that with respect to $R_{0\eps}$ the Radon-Nikodym derivative of $\pi^{\eps}$ has the product structure given in \eqref{eq:intro-potent-defn}, where $a^{\eps}: M \to (0,+\infty)$ with $\log a^{\eps} \in L^{1}(\mu)$. It then follows that
    \begin{align*}
        \frac{d\ell_{\eps}}{d\pi^{\eps}}(x,z) &= \frac{d\ell_{\eps}/dR_{0\eps}}{d\pi^{\eps}/dR_{0\eps}}(x,z) = \frac{\sqrt{e^{-U_{\mu}+U_{\diffref}}(z)e^{-U_{\mu}+U_{\diffref}}(x)}\exp(-c(x,z,\eps))}{a^{\eps}(x)a^{\eps}(z)}.
    \end{align*}
    As $\pi^{\eps},\ell_{\eps} \in \Pi(\mu,\mu)$, all the terms in $\log \left(d\ell_{\eps}/d\pi^{\eps}\right)$ depending only on $x$ or $z$ alone cancel in the symmetric relative entropy calculation, giving with \eqref{eq:remainder-bound} that
    \begin{align}\label{eq:symm-rem-bdd}
        H(\ell_{\eps}|\pi^{\eps})+H(\pi^{\eps}|\ell_{\eps}) &= \left(\Exp{\ell_{\eps}}-\Exp{\pi^{\eps}}\right)\left[-c(x,z,\eps)\right] \leq \Exp{\pi^{\eps}}[R(x,z,\eps)].
    \end{align}
    Now, let $P^{\eps} \in \cP(C([0,1],M))$ denote the $\eps$-dynamic $\Schro$ bridge from $\mu$ to $\mu$. Let $(\mu_t^{\eps},v_t^{\eps}, t \in [0,1])$ denote the absolutely continuous curve in $\cP_2(M)$ given by the entropic interpolation. From the entropic Benamou-Brenier formula given in Proposition \ref{prop:entropic-bb}, comparison with the constant curve $(\mu_t \equiv \mu, v_t \equiv 0, t \in [0,1])$ gives
    \begin{align}\label{eq:ke-bdd-for-fi}
        \frac{1}{\eps^2}\int_{0}^{1} \norm{v_t^{\eps}}^2_{L^2(\mu_t^{\eps})} dt \leq \frac{1}{4}\left(I_{g}(\mu|\diffref)-\int_0^{1} I_{g}(\mu_t^{\eps}|\diffref)dt \right).
    \end{align}
    Let $\cL_{R}$ denote the generator associated to $R$, and let $R^{\eps}$ denote the stationary path measure on $C([0,1],M)$ associated to the rescaled $\eps \cL_{R}$. Observe that both $R^{\eps}$ and $R$ have the same stationary measure $\diffref$. With this rescaling, the remainder term writes as 
    \begin{align}
        R(x,z,\eps) &= -\log\left(\Exp{R_{x}^{\eps}}\left[\exp\left(-\eps \int_0^{1}(\cU_{\mu}-\cU_{\diffref})(\omega_t)-(\cU_{\mu}-\cU_{\diffref})(\omega_0) dt\right)\middle| \omega_{1} = z\right]\right).
    \end{align}
    As the bridges of $P^{\eps}$ are equal to the bridges of $R^{\eps}$, by Jensen's inequality it holds that
    \begin{align}\label{eq:rem-to-cont-eqn}
        \Exp{\pi^{\eps}}\left[R(x,z,\eps)\right] &\leq \eps \Exp{P^{\eps}}\left[\int_0^{1} (\cU_{\mu}-\cU_{\diffref})(\omega_t)-(\cU_{\mu}-\cU_{\diffref})(\omega_0) dt\right] = \eps \int_0^{1} (\Exp{\mu_t^{\eps}}[\cU_{\mu}-\cU_{\diffref}]-\Exp{\mu}[\cU_\mu-\cU_{\diffref}])dt. 
    \end{align}
    Let $(\chi_k,k \geq 1) \subset C_{c}^{\infty}(M)$ be a sequence of nonnegative functions increasing to $1$ pointwise such that $\norm{\nabla \chi_k}_{\infty}\leq k^{-1}$ (such a family exists as $M$ is complete). Set $\psi := \cU_\mu - \cU_{\diffref} \in C^{1}(M)$, then $\chi_k \psi \in C^{1}_{c}(M)$. By the definition of a weak solution to the continuity equation it holds for all $k \geq 1$ and $t \in [0,1]$ that
    \begin{align*}
        \Exp{\mu_t^{\eps}}[\chi_k \psi] - \Exp{\mu}[\chi_k \psi] &= \int_{0}^{t} \frac{d}{ds}\Exp{\mu_s^{\eps}}[\chi_k \psi] ds = \int_0^{t}  \Exp{\mu_s^{\eps}} \left[\langle \nabla (\chi_k \psi), v_t^{\eps}\rangle_{g} \right]ds.
    \end{align*}
    By \eqref{eq:cu-integrability}, send $k \to +\infty$ and apply Dominated Convergence to conclude for Leb-a.e.\ $t \in [0,1]$ 
    \begin{align*}
        \Exp{\mu_t^{\eps}}[\psi]-\Exp{\mu}[\psi] &= \int_0^{t} \Exp{\mu_{s}^{\eps}}[\langle \nabla_g \psi, v_s^{\eps} \rangle_g]ds. 
    \end{align*} 
    Thus, by Cauchy-Schwarz and \eqref{eq:ke-bdd-for-fi}, \eqref{eq:rem-to-cont-eqn} becomes
    \begin{align*}
        \Exp{\pi^{\eps}}\left[R(x,z,\eps)\right] &\leq \eps^2 \left(\int_0^{1} \norm{\nabla_g (\cU_{\mu}-\cU_{\diffref})}_{L^{2}(\mu_t^{\eps})}^2dt\right)^{1/2}\left(\frac{1}{\eps^2}\int_0^1 \norm{v_{t}^{\eps}}_{L^2(\mu_t^{\eps})}^2 dt\right)^{1/2} \\
        &\leq \frac{\eps^2}{2}\left(\int_0^{1} \norm{\nabla_g (\cU_{\mu}-\cU_{\diffref})}_{L^{2}(\mu_t^{\eps})}^2dt\right)^{1/2} \left(I_{g}(\mu|\diffref)-\int_0^{1} I_{g}(\mu_t^{\eps}|\diffref)dt \right)^{1/2}. 
    \end{align*}
\end{proof}

\subsection{Some consequences}
From Theorem \ref{thm:sym-rel-ent}, we extend the entropic cost expansion developed in \cite[Theorem 1.6]{conforti21deriv} in the case of identical marginals. 
\begin{corollary}\label{cor:ct-cost-exp}
    Let $\mu R_{0\eps} \in \cP(M \times M)$ denote the joint distribution $(Y_0,Y_{\eps})$ in \eqref{eq:sec3-ref-diff} with $Y_0 \sim \mu$. In the setting of Theorem \ref{thm:sym-rel-ent}, the entropic cost in the same marginal problem has the following expansion about $\eps = 0$
    \begin{align}
        H(\pi^{\eps}|\mu R_{0\eps}) &= \frac{\eps}{8}I_{g}(\mu|\diffref)+o(\eps^2). 
    \end{align}
\end{corollary}
\begin{proof}
    As in the Euclidean case presented in \cite[Corollary 1]{AHMP25} this is a simply consequence of the Pythagorean Theorem for relative entropy \cite{csiszar75idiv}. That is,
\begin{align*}
    H(\pi^{\eps}|\mu R_{0\eps}) &\leq H(\ell_{\eps}|\mu R_{0\eps})-H(\pi^{\eps}|\ell_{\eps}) = \frac{\eps}{8}I_g(\mu|\diffref)+o(\eps^2). 
\end{align*}
\end{proof}

Another benefit of the diffusion approximation established in Theorem \ref{thm:sym-rel-ent} is that we can quantify the rate of decay of measures of sets under $\pi^{\eps}$ that are separated from the diagonal of $M \times M$ as $\eps \downarrow 0$. The following bound will be used in several places in the sequel. 
\begin{proposition}\label{prop:almost-ldp-sb}
Resume the setting and notation of Theorem \ref{thm:sym-rel-ent}. Let $c > 0$ and define
\begin{align*}
    N = \left\{(x,z) \in M \times M: d(x,z)\leq c\right\}.
\end{align*}
Under (H1), for all $\eps > 0$ small enough it holds that
\begin{align}\label{eq:pi-eps-compact}
    \pi^{\eps}((M \times M) \setminus N) \leq H(\pi^{\eps}|\ell_{\eps})+2\exp\left(-\frac{c}{2\eps}\right).
\end{align}
Under (H2), suppose that $\Exp{\mu}\left[\norm{\nabla_g U_{\mu}}_g^p\right] < +\infty$ for some $p \geq 2$. There is $K >0$ such that
\begin{align}\label{eq:pi-eps-noncompact}
    \pi^{\eps}((M \times M) \setminus N) &\leq H(\pi^{\eps}|\ell_{\eps})+Kc^{-p}\eps^{p/2}.
\end{align}
\end{proposition}
\begin{proof}
    By Fenchel's inequality $xy \leq x\log x - x + \exp(y)$ for all $x > 0$, $y \in \mathbb{R}$, giving
    \begin{align*}
        \pi^{\eps}((M \times M) \setminus N) 
        &\leq \Exp{\ell_{\eps}}\left[\frac{d\pi^{\eps}}{d\ell_{\eps}}\log\left(\frac{d\pi^{\eps}}{d\ell_{\eps}}\right)-\frac{d\pi^{\eps}}{d\ell_{\eps}}+\exp\left(\mathbbm{1}((M \times M)\setminus N)\right)\right] \\
        &\leq H(\pi^{\eps}|\ell_{\eps}) + \left(1+\frac{e}{2}\right)\ell_{\eps}\left((M \times M)\setminus N\right).
    \end{align*}
    The last line follows the numerical inequality $e^x \leq 1 + x + \frac{e}{2}x^2$ for all $x \in [0,1]$. Thus it remains to quantify the decay of $\ell_{\eps}((M \times M)\setminus N)$. Under (H1), \eqref{eq:pi-eps-compact} follows from the Large Deviation Principle for $(\ell_{\eps}, \eps > 0)$ established in Proposition \ref{prop:ldp-rate-diff}. Under (H2) and the additional integrability hypothesis on $\nabla_g U_{\mu}$, \eqref{eq:pi-eps-noncompact} follows from Markov's inequality and Proposition \ref{prop:diff-fourth-moment}. That is, there is some $K > 0$ such that
    \begin{align}\label{eq:diff-off-diag-concen}
        \ell_{\eps}\left((M \times M)\setminus N\right) \leq c^{-p}\int_{M \times M} d^p(x,z)\ell_{\eps}(dxdz) \leq Kc^{-p}\eps^{p/2}. 
    \end{align}
\end{proof}

\section{Generator and $\Schro$ Potential Analysis}\label{sec:generator-analysis}
We now leverage the diffusion approximation for the same-marginal $\Schro$ bridge to develop further small temperature properties of same-marginal $\Schro$ bridges. 

\subsection{Generator Transfer under Curvature Condition}
We call the following Theorem a ``generator transfer'' Theorem. In this result, we show that a sufficiently strong diffusion approximation (in relative entropy) endows a family of integral operators with no a priori Markov structure a property resembling that of a generator. 
This is a geometric generalization to the Euclidean result in \cite[Theorem 3]{AHMP25}. As we will apply this result in various settings, we state the following Theorem more abstractly. 
\begin{theorem}\label{thm:generator-transfer}
Let $(M,g)$ satisfy Assumption \ref{assumption:manifold}, and let $\mu \in \cP_2(M)$ satisfy $\mathrm{CD}(\kappa,\infty)$ for some $\kappa \in \mathbb{R}$. Let $(X_t, t \geq 0)$ be the diffusion on $M$ given by \eqref{eq:sec3-target-diff}. Assume that this SDE has a unique weak solution that does not explode, and set $\beta^{\eps} = \mathrm{Law}(X_0,X_{\eps}) \in \Pi(\mu,\mu)$. Let $(\alpha^{\eps},\eps > 0) \subset \Pi(\mu,\mu)$ be such that $H(\alpha^{\eps}|\beta^{\eps}) = o(\eps)$. Let $\cL$ denote the generator of the diffusion in \eqref{eq:sec3-target-diff}, and let $D(\cL)$ denote its domain \cite[Section 3.4 (D7)]{bgl-markov}. It then holds for any $\xi \in \mathrm{Lip}(M) \cap D(\cL)$ in $L^2(\mu)$ that
    \begin{align}
        \lim\limits_{\eps \downarrow 0} \frac{1}{\eps}\left(\Exp{\alpha^{\eps}}[\xi(Y)|X=x]-\Exp{\beta^{\eps}}[\xi(Y)|X=x]\right) = 0.
    \end{align}
    In particular, this implies that in $L^2(\mu)$
    \begin{align}
        \lim\limits_{\eps \downarrow 0} \frac{1}{\eps}\left(\Exp{\alpha^{\eps}}[\xi(Y)|X=x] - \xi(x)\right) = \cL\xi(x). 
    \end{align}
\end{theorem}

\begin{proof}
Let $(p_t(x,\cdot),t > 0, x \in M)$ denote the transition densities of the diffusion process on $M$ with stationary measure $\mu$ satisfying $\mathrm{CD}(\kappa,\infty)$. By \cite[Proposition 5.5.2]{bgl-markov}, the triplet $(M,g,p_t(x,y)\mu(dy))$ satisfies $\mathrm{LSI}\left(\frac{2}{\kappa}(1-e^{-\kappa t})\right)$ (see \cite[Definition 5.1.1]{bgl-markov}) for all $x \in M$ and $t > 0$. Note the different scaling due to differing conventions in \cite{bgl-markov}. When $\kappa = 0$, interpret the LSI constant as being equal to $2t$. By \cite[Proposition 1, Theorem 5.2]{gigli-ledoux-lsi-to-tala} (originally from \cite{otto-villani-00}), each such triplet satisfies a Talagrand inequality, $T_{2}\left(\frac{2}{\kappa}(1-e^{-\kappa t})\right)$. That is, for all $\eta \in \cP(M)$,
\begin{align}\label{eq:cd-tala-transdens}
    \Was{2}^{2}(\eta,p_{t}(x,y)\mu(dy)) \leq \frac{1}{\kappa}(1-e^{-\kappa t}) H(\eta|p_{t}(x,y)\mu(dy)).
\end{align}

We proceed with a coupling argument identical to that of \cite[Theorem 3]{AHMP25}. Fix $\eps > 0$. Construct via a gluing argument $(X,Y,Z)$ a coupling of $(\mu,\mu,\mu)$ such that for each $x \in M$, $(Y|X=x,Z|X=x)$ is the $\Was{2}$ optimal coupling between $p_{\eps}(x,\cdot)\mu(\cdot)$ and $\alpha^{\eps}_{x}(\cdot)$, where $\alpha^{\eps}_x$ is the conditional distribution under $\alpha^{\eps}$ of the second coordinate given that the first is equal to $x$. Now, take $\xi \in \text{Lip}(M)$, then it holds by Jensen's inequality, $\Was{2}$ optimality of the construction, and Talagrand's inequality in \eqref{eq:cd-tala-transdens}
\begin{align*}
    \frac{1}{\eps}\abs{\Exp{\alpha^{\eps}}[\xi(Y)|X=x]-\Exp{\beta^{\eps}}[\xi(Z)|X=x]} &\leq \frac{1}{\eps} \norm{\xi}_{Lip}\Exp{}\left[d(Y,Z)|X=x\right] \leq \frac{1}{\eps}\norm{\xi}_{Lip} \Was{2}(\alpha^{\eps}_{x},p_{\eps}(x,\cdot)\mu(\cdot)) \\
    &\leq \frac{1}{\eps}\norm{\xi}_{Lip} \cdot \sqrt{\frac{1-e^{-\kappa \eps}}{\kappa} H(\alpha^{\eps}_{x}|p_{\eps}(x,\cdot)\mu(\cdot))}.
\end{align*}
Square both sides and take expectation with respect to $\mu$. As $\alpha^{\eps}$ and $\beta^{\eps}$ are couplings of the same measures, it holds from the factorization of relative entropy that
\begin{align*}
    \int \frac{1}{\eps^2}\abs{\Exp{\alpha^{\eps}}[\xi(Y)|X=x]-\Exp{\beta^{\eps}}[\xi(Z)|X=x]}^2 \mu(x)dx &\leq \frac{1}{\eps^2}\norm{\xi}_{Lip}^2 \cdot \frac{1-e^{-\kappa \eps}}{\kappa} H(\alpha^{\eps}|\beta^{\eps}) \\
    &= \norm{\xi}_{Lip}^2 \left(\frac{1-e^{-\kappa \eps}}{\kappa \eps}\right) \cdot \frac{1}{\eps}H(\alpha^{\eps}|\beta^{\eps}).
\end{align*}
As $\lim\limits_{\eps \downarrow 0} (1-e^{-\kappa \eps})/(\kappa \eps) = 1$, this bound vanishes as $H(\alpha^{\eps}|\beta^{\eps}) = o(\eps)$.
\end{proof}

Observe that the above Theorem \ref{thm:generator-transfer} only requires a diffusion approximation of order $o(\eps)$. In Theorem \ref{thm:sym-rel-ent}, we develop a diffusion approximation of order $o(\eps^2)$ to the static $\Schro$ bridge. In the following proposition, we show how this faster rate of convergence can enlarge the function class under which the generator transfer occurs.
\begin{proposition}[Extension to Theorem \ref{thm:generator-transfer}]\label{prop:generator-extn}
    In the setting of Theorem \ref{thm:generator-transfer}, suppose now that $H(\alpha^{\eps}|\beta^{\eps}) = o(\eps^2)$. Let $(\xi_{\eps}, \eps > 0) \subset \mathrm{Lip}(M) \cap D(\cL)$ and $\xi \in D(\cL)$ be such that $\norm{\xi_{\eps}-\xi}_{L^2(\mu)} = o(\eps)$. If $\sup\limits_{\eps > 0} \sqrt{\eps}\norm{\xi_{\eps}}_{Lip} < +\infty$, then it holds that
    \begin{align}
        \lim\limits_{\eps \downarrow 0} \frac{1}{\eps}\left(\Exp{\alpha^{\eps}}[\xi(Y)|X=x] - \xi(x)\right) = \cL\xi(x). 
    \end{align}
\end{proposition}
\begin{proof}
Let $\alpha^{\eps}_x$ and $\beta^{\eps}_{x}$ denote the conditional expectation operators under $\alpha^{\eps}$ and $\beta^{\eps}$ given $X=x$. Following the argument of Theorem \ref{thm:generator-transfer}, for all $\eps > 0$
\begin{align}
    \frac{1}{\eps}\norm{\alpha^{\eps}_x \xi - \beta^{\eps}_{x}\xi}_{L^{2}(\mu)} \leq \sqrt{\eps}\norm{\xi_{\eps}}_{Lip}\left(\frac{1-e^{-\kappa \eps}}{\kappa \eps}\right)^{1/2} \cdot \left(\frac{1}{\eps^2}H(\alpha^{\eps}|\beta^{\eps})\right)^{1/2},
\end{align}
so the left hand side vanishes as $\eps \downarrow 0$ as $H(\alpha^{\eps}|\beta^{\eps}) = o(\eps^2)$. Next, recall that $\alpha^{\eps}$ and $\beta^{\eps}$ are both couplings of $\mu$ with itself. By Jensen's inequality,
\begin{align}
    \norm{\Exp{\alpha^{\eps}}[(\xi_{\eps}-\xi)(Y)|X=x]}_{L^2(\mu)}, \norm{\Exp{\beta^{\eps}}[(\xi_{\eps}-\xi)(Y)|X=x]}_{L^2(\mu)} \leq \norm{\xi_{\eps}-\xi}_{L^2(\mu)}, 
\end{align}
so the terms on the left hand side are $o(\eps)$. Altogether then, observe that
\begin{align*}
    \frac{1}{\eps}\left(\alpha_x^{\eps} \xi_{\eps} - \xi\right) &= \frac{1}{\eps}(\alpha_x^{\eps}\xi-\alpha^{\eps}_{x}\xi_{\eps})+ \frac{1}{\eps}\left(\alpha^{\eps}_x \xi_{\eps}-\beta^{\eps}_x \xi_{\eps}\right) + \frac{1}{\eps}(\beta^{\eps}_x \xi_{\eps} - \beta^{\eps}_{x}\xi)+\frac{1}{\eps}\left(\beta^{\eps}_x \xi - \xi\right). 
\end{align*}
The first three terms on the right hand side vanish in $L^2(\mu)$ as $\eps \downarrow 0$, and the rightmost term converges in $L^2(\mu)$ to $\cL \xi$ as $\eps \downarrow 0$. 
\end{proof}

While this extension is immaterial when $M$ is compact, it does extend the function class on which the generator transfer holds when $M$ is noncompact. In particular, we obtain the following extension in the $\mathbb{R}^{d}$ setting of \cite[Theorem 3]{AHMP25}.
\begin{corollary}[Extension on $\mathbb{R}^{d}$]
    Let $\mu = e^{-U} \in \cP_2(\mathbb{R}^{d})$ satisfy Assumption \ref{assumption:diffusions}. Additionally, assume that $\norm{\nabla \cU_{\mu}}^2$ has polynomial growth and $\mu$ has subexponential tails (see Proposition \ref{prop:h2-int-argument}). Let $\pi^{\eps}$ denote the $\eps$-static $\Schro$ bridge from $\mu$ to itself computed with respect to Euclidean Wiener reference measure. Then, for all for all polynomial $\xi \in C^{\infty}(\mathbb{R}^{d})$ it holds in $L^2(\mu)$ that
    \begin{align*}
        \lim\limits_{\eps \downarrow 0} \frac{1}{\eps}\left(\Exp{\pi^{\eps}}[\xi(Y)|X=x] - \xi(x)\right) &= \cL\xi(x)
    \end{align*}
\end{corollary}

\begin{proof}
    By Proposition \ref{prop:h2-int-argument}, \eqref{eq:o-eps-2-limit} holds. As linear functions are Lipshitz and thus included in Theorem \ref{thm:generator-transfer}, let $\xi$ be a polynomial of degree $\geq 2$. As $\mu$ is subexponential (so has moments of all order) and $I(\mu|\text{Leb}) = \Exp{\mu}\norm{\nabla U}^2 < +\infty$, it holds that $\cL \xi \in L^2(\mu)$ and thus $\xi \in D(\cL)$. 

    Let $\left(\chi_R, R > 1\right) \subset C_c^{\infty}(\mathbb{R}^{d})$ be a family of smooth cutoff functions such that  $\chi_R(x) \equiv 1$ on $B(0,R)$, $\chi_R \equiv 0$ on $\mathbb{R}^{d}\setminus B(0,R+1)$, and $\sup_{R > 1} \norm{\nabla\chi_R}_{\infty} < +\infty$.
    % \GM{Note to self: this exists, standard mollifcation arg}
    Let $p \geq 2$ denote the degree of $\xi$ and define $\eta(\eps) = \eps^{-\frac{1}{p^2}}$ for $\eps < 1$. Set $\xi_{\eps} := \chi_{\eta(\eps)}\xi \in C_c^{\infty}(\mathbb{R}^{d})$, there exists some $C > 0$ (depending on $p$ and the coefficients of $\xi$) such that for any $x \in \mathbb{R}^{d}$
    \begin{align*}
        \norm{\nabla (\chi_{\eta(\eps)}\xi)(x)} &\leq \sup\limits_{ B(0,\eta(\eps)+1)}\left(\norm{(\nabla \chi_{\eps}) \xi}  + \norm{\nabla \xi} \right)\leq C\left(\eta(\eps)^{p}+\eta(\eps)^{p-1}\right) \leq 2C\eps^{-\frac{1}{p}}.
    \end{align*}
    The Lipshitz constant of $\xi_{\eps}$ is less than or equal to $2C\eps^{-\frac{1}{p}}$, so as $p \geq 2$ the Lipshitz condition of Proposition \ref{prop:generator-extn} on the $(\xi_{\eps}, \eps > 0)$ is satisfied. To apply Proposition \ref{prop:generator-extn}, it then remains to show that $\norm{\xi_{\eps}-\xi}_{L^2(\mu)} = o(\eps)$. By Holder's inequality,
    \begin{align*}
        \norm{\xi_{\eps}-\xi}_{L^2(\mu)}^2 &\leq (\Exp{\mu}[\xi^2])^{1/2}\mu(\mathbb{R}^{d}\setminus B(0,\eta(\eps)))^{1/2}.
    \end{align*}
    As $\xi^2$ is a polynomial, the desired convergence rate follows from the subexponential tail bound. That is, there is $C, K > 0$ such that $\mu(\mathbb{R}^{d}\setminus B(0,\eta(\eps))) \leq C\exp\left(-K\eta(\eps)\right)$, which decays faster as $\eps \downarrow 0$ than any polynomial power of $\eps$. 
\end{proof}

\subsection{Next Order Potential Convergence}
From \cite[Theorem 1.1]{chiarini2022gradient}, $(\eps \nabla_g \log a^{\eps}, \eps >0)$ vanishes as $\eps \downarrow 0$. Using the diffusion approximation, we are able to identify the next order in $\eps$ and uncover the role of the manifold geometry in the limit.
\begin{theorem}\label{thm:score-function}
Let $(M,g)$ satisfy Assumption \ref{assumption:manifold} and $\mu \in \cP_2(M)$ satisfy Assumption \ref{assumption:diffusions}. In the case of (H2), \textbf{assume in addition} that $\Exp{\mu}[\norm{\nabla_g U_{\mu}}_g^4] < +\infty$ and the injectivity radius admits a positive global lower bound, that is, $\inf\limits_{x \in M} \mathrm{inj}(x) > 0$.

Let $\pi^{\eps}$ denote the $\eps$-static $\Schro$ bridge from $\mu \in \cP_2(M)$ to itself computed with the manifold Brownian motion as reference measure. Let $a^{\eps}$ denote the $\Schro$ potential as in \eqref{eq:intro-potent-defn}. It holds in $L^{2}(TM,\mu)$ that
\begin{align}\label{eq:mainfold-potential}
    \lim\limits_{\eps \downarrow 0}\nabla_g \log a^{\eps}(x) = -\lim\limits_{\eps \downarrow 0}\left(\frac{1}{\eps}\int_{M} \log_{x}(z) \pi^{\eps}(dz|x)\right) = \frac{1}{2}\nabla_g \log \left(\frac{d\mu}{d\vol}\right)(x).
\end{align}
In the case of (H2), fix a global chart so that the following conditional expectation makes sense. Then on this chart it holds in $L^2(\mu)$, under no assumption on the cutlocus, that
    \begin{align}\label{eq:manifold-bp}
        \lim\limits_{\eps \downarrow 0} \frac{1}{\eps}\left(\Exp{\pi_{\eps}}[Z|X=x] - x\right) \to \frac{1}{2}g^{ij}(x)\frac{\partial}{\partial x^{j}}\left(\log \frac{d\mu}{d\vol}\right)(x)-\frac{1}{2}g^{ij}(x)\Gamma_{ij}^{k}(x). 
    \end{align}
In the case of a Hessian manifold (i.e.\ $g = \nabla^2 \varphi$), this further reduces to
\begin{align}\label{eq:hess-manifold-bp}
    \lim\limits_{\eps \downarrow 0} \frac{1}{\eps}\left(\Exp{\pi_{\eps}}[Z|X=x] - x\right) \to \frac{1}{2}g^{ij}(x)\frac{\partial}{\partial x^{j}}\left(\log \frac{d\mu}{d\vol} - \frac{1}{2}\log g \right)(x). 
\end{align}
\end{theorem}

\begin{remark}
    Fix $x \in M$ and let $C_{x} \subset M$ denote the cutlocus of $x$. For $z \in M \setminus C_x$, $\log_x z \in T_{x}M$ always exists 
    and $\vol(C_x) = 0$ \cite[Theorem 10.34]{riemannian-manifolds-lee}.
    Thus, the correct way to interpret the first term on the right hand side of \eqref{eq:mainfold-potential} is the following: consider the $\vol$-a.e.\ defined function $z \mapsto \log_{x}(z)$. Since $\pi^{\eps}(\cdot|x) << \vol$, the quantity $\int_{M} \log_{x}(z) \pi^{\eps}(dz|x)$ is then a well-defined element of $T_{x}M$. 
\end{remark}

\begin{remark}
    From \cite[Section 3]{brewin-expansions-09}, the Riemannian log has the following local expansion 
    \begin{align*}
        \log_x(z) = (z-x) + \frac{1}{2}\Gamma_{ij}^{k}(x)(z-x)^i(z-x)^k + O(\norm{z-x}^3). 
    \end{align*}
    Thus, the extra term arising in \eqref{eq:manifold-bp} compared to \eqref{eq:mainfold-potential} is due to non-zero curvature of $M$.  
\end{remark}

We quickly recount how the results of Section \ref{sec:diff-approx} apply to the setting in Theorem \ref{thm:score-function}. Let $\ell_{\eps}$ denote the diffusion approximation to $\pi^{\eps}$ developed in Theorem \ref{thm:sym-rel-ent}. Then Theorem \ref{thm:sym-rel-ent} and Proposition \ref{prop:the-one-about-fi} (as $U_{\diffref} = 0$) establish \eqref{eq:o-eps-2-limit}. Hence, for $c > 0$ and $N = \{(x,z) \in M \times M: d(x,z) > c\}$, Proposition \ref{prop:almost-ldp-sb} gives that, under both (H1) and (H2), there is some constant $K > 0$, depending on $c$, such that
    \begin{align}\label{eq:vol-estimate}
        \pi_{\eps}((M \times M)\setminus N) \leq K\eps^2. 
    \end{align}
We now proceed to the proof of Theorem \ref{thm:score-function}.
\begin{proof}
    The $\Schro$ bridge is a probability measure on $M \times M$ that we write as
    \begin{align}
        \frac{d\pi^{\eps}}{d(\vol \otimes \vol)}(x,z) &= a^{\eps}(x)a^{\eps}(z)p_{\eps}(x,z),
    \end{align}
    where $p_{\eps}(x,z)$ is the transition density of the reversible manifold Brownian motion process with respect to $\vol$, that is, $dW_{0\eps}/d(\vol \otimes \vol)$ where $W_{0\eps} := \mathrm{Law}(B_0^{M},B_\eps^{M})$ with $B_0^{M} \sim \vol$. We note that this is a change from the notation of previous sections to emphasize that we only considering Wiener reference for this result. 
    
    \textbf{Outline of argument.} The corrector estimates in \cite[Proposition 2.4]{chiarini2022gradient} establish 
    \begin{align}\label{eq:corrector-estimate}
        \norm{\nabla_g \log a^{\eps}}^2_{L^2(\mu)} &\leq \frac{2}{E_{\kappa_{\diffref}}(\eps)}\left(H(\pi^{\eps}|W_{0\eps})-H(\mu|\vol)\right),
    \end{align}
    where $E_{\kappa_{\vol}}(\eps) = \int_0^{\eps} \exp(\kappa_{\vol}t)dt$ with $\kappa_{\vol} \in \mathbb{R}$ as defined in Assumption \ref{assumption:diffusions}. Observe the different scaling factors due to differing conventions. From Corollary \ref{cor:ct-cost-exp}, it then holds that
    \begin{align}\label{eq:upper-bdd-l2-norm}
        \limsup\limits_{\eps \downarrow 0} \norm{\nabla_g \log a^{\eps}}^2_{L^2(\mu)} &\leq 2 \cdot \frac{1}{8}I(\mu|\vol) = \frac{1}{4}\norm{\nabla_g \log \left(\frac{d\mu}{d\vol}\right)}^2_{L^2(\mu)}
    \end{align}
    Hence, the collection $(\nabla \log a^{\eps}, \eps > 0)$ is $L^2(\mu)$-bounded and thus weakly $L^2(\mu)$-compact. The major claim in this proof is the following
    \begin{align}\label{eq:major-claim}
        \text{Any subsequence of $\nabla_g \log a^{\eps}$ has a further subsequence converging $\mu$-a.s.\ to $\frac{1}{2}\nabla_g \log\left(\frac{d\mu}{d\vol}\right)$}.
    \end{align}
    Take for granted that \eqref{eq:major-claim} holds. Fix some subsequence $(\nabla_g \log a^{\eps_{n}}, n \geq 1)$ such that $\eps_n \to 0$ as $n \to +\infty$. From \eqref{eq:major-claim}, extract a subsequence such that the desired $\mu$-a.s.\ pointwise convergence happens. Denote this subsequence $(\nabla_g \log a^{\eps_{n_{k}}}, k \geq 1)$, then thanks to the bound from \eqref{eq:corrector-estimate}, it holds that the sequence converges weakly in $L^2(\mu)$ to $\frac{1}{2} \nabla_g \log \left(\frac{d\mu}{d\vol}\right)$ \cite[Chapter 6, Exercise 20]{folland-measure-theory}. Moreover, by Fatou's Lemma
    \begin{align*}
        \frac{1}{4}\norm{\nabla_g \log\left(\frac{d\mu}{d\vol}\right)}^2_{L^2(\mu)} &\leq \liminf\limits_{k \to +\infty} \norm{\nabla_g \log a^{\eps_{n_{k}}}}^2_{L^2(\mu)},
    \end{align*}
    which is a matching lower bound to \eqref{eq:upper-bdd-l2-norm}. Altogether then, it holds that $(\nabla_g \log a^{\eps_{n_{k}}}, k \geq 1)$ converges weakly and has convergence of norms to $\frac{1}{2}\nabla_g \log \left(\frac{d\mu}{d\vol}\right)$. Hence, the convergence holds strongly in $L^2(\mu)$ \cite[Exercise 5.19]{brezis-fa11}. We have shown that any subsequence of $(\nabla_g \log a^{\eps}, \eps > 0)$ admits a further subsequence converging strongly in $L^2(\mu)$ to $\frac{1}{2}\nabla_g \log \left(\frac{d\mu}{d\vol}\right)$, so $L^2(\mu)$ convergence holds along the entire sequence as $\eps \downarrow 0$. 

    \textbf{Proof of $\mu$-a.s.\ Subsequential Limit.} The rest of the proof will now focus on establishing \eqref{eq:major-claim}. We break this proof into several steps.

    \textbf{Step 1:} It holds for all $\eps > 0$ that
    \begin{align}\label{eq:a-eps-v1}
        \nabla_g \log a^{\eps}(x) &= -\int_{M} \nabla_{x} \log p_{\eps}(x,z)\pi^{\eps}(dz|x) + \nabla_g \log\left(\frac{d\mu}{d\vol}\right)(x),
    \end{align}
    where $\nabla_x$ denotes the Riemannian gradient computed in the $x$ variable. To see this, fix $x \in M$. The conditional density writes as
    \begin{align}\label{eq:cond-dens-pi-eps}
        \frac{d\pi^{\eps}(\cdot|x)}{d\vol}(z) &= a^{\eps}(x)a^{\eps}(z)p_{\eps}(x,z)\left(\frac{d\mu}{d\vol}(x)\right)^{-1}.
    \end{align}
    Take the log and then gradient with respect to $x$. Rearranging the expression gives
    \begin{align}
        \nabla_g \log a^{\eps}(x) &= - \nabla_{x} \log p_{\eps}(x,z)+\nabla_g \log \left(\frac{d\mu}{d\vol}\right)(x) + \nabla_{x} \log \left(\frac{d\pi^{\eps}(\cdot|x)}{d\vol} (z)\right). 
    \end{align}
    Observe that all the above quantities in the above identity are elements in the tangent space $T_{x}M$. It holds that
    \begin{align}\label{eq:zero-exp-score-fnc}
        \int_{M} \nabla_{x} \log \left(\frac{d\pi^{\eps}(\cdot|x)}{d\vol} (z)\right) \pi^{\eps}(dz|x) = 0 \in T_xM. 
    \end{align}
    To see this, we argue as follows. First, observe that we are integrating a function valued on $T_x M$, which is a finite dimensional normed vector space. Hence, such an operation is well-defined. Let $F: M \times M \to [0,+\infty)$ be the function $F(x,z) = \frac{d\pi^{\eps}(\cdot|x)}{d\vol}(z)$. Fix any local system of coordinates about $x \in M$, say $(u^{1},\dots,u^{d})$. In this coordinate system, the integrand in \eqref{eq:zero-exp-score-fnc} writes as
    \begin{align*}
        \left.\nabla_{u} \log F(u,z)\right|_{u=x} &= \left(g^{ij}(x) \left.\partial_{u^{j}}\right|_{u=x}\log F(u,z)\right)\left.\frac{\partial}{\partial u^i}\right|_{x} = \left(g^{ij}(x)\frac{\left.\partial_{u^{j}}\right|_{u=x} F(u,z)}{F(x,z)}\right)\left.\frac{\partial}{\partial u^i}\right|_x \in T_{x}M. 
    \end{align*}
    We now integrate each coordinate individually. Fix $i \in \{1,\dots,d\}$, then the $i$th coordinate of \eqref{eq:zero-exp-score-fnc} becomes
    \begin{align*}
        \int_{M} g^{ij}(x)\frac{\left.\partial_{u^{j}}\right|_{u=x}F(u,z)}{F(x,z)}\pi^{\eps}(dz|x) &= g^{ij}(x) \left.\partial_{u^{j}}\right|_{u=x}\left(\int_{M} F(u,z) \vol(dz)\right) = g^{ij}(x) \partial_{u^{j}}1 = 0. 
    \end{align*}
    This establishes \eqref{eq:zero-exp-score-fnc}, so \eqref{eq:a-eps-v1} holds. Observe from \eqref{eq:cond-dens-pi-eps} that the smoothness necessary to interchange integral and derivative is present.

    \textbf{Step 2:} Let $c > 0$ denote the global injectivity radius of $M$. When $M$ is compact this is guaranteed to be positive \cite[Proposition 10.37]{riemannian-manifolds-lee}, otherwise we make this assumption when $M$ is non-compact.
    We claim that along any subsequence there is a further subsequence such that for $\mu$-a.s.\ $x \in M$
    \begin{align}\label{eq:claim-throw-away-set}
        \lim\limits_{\eps \downarrow 0} \int_{M} \norm{\nabla_x \log p_{\eps}(x,z) - \mathbbm{1}(B(x,c/2))(z) \cdot \frac{1}{\eps}\log_{x}(z)}_g \pi^{\eps}(dz|x) = 0.
    \end{align}
    Define the set
    \begin{align}
        N = \left\{(x,z) \in M \times M: d(x,y) \leq \frac{c}{2}\right\}.
    \end{align}
    Observe that $N \subset M \times M$ is closed and contained within the cutlocus. We claim that
    \begin{align}\label{eq:concentration-bound}
        \lim\limits_{\eps \downarrow 0} \int_{(M \times M) \setminus N} \norm{\nabla_{x} \log p_{\eps}(x,z)}_g \pi^{\eps}(dxdz)= 0. 
    \end{align}
    The key tool is the following bound on gradient of the log of the heat kernel, which furnishes $C > 0$ holding for all $x,z \in M$ and $\eps > 0$ small enough:
    \begin{align}\label{eq:log-heat-kernel-point-pdd}
        \norm{\nabla_x \log p_{\eps}(x,z)}_g \leq C\left(\frac{1}{\sqrt{\eps}}+\frac{d(x,z)}{\eps}\right).
    \end{align}
    Under (H1), this is for instance established in \cite[Theorem 2.1]{chen-log-heat-kernel-est-23}. Under (H2), the existence of such a $C > 0$ for all $x,z \in M$ and $\eps >0$ small enough follows from \cite[Theorem 1]{engoulatov-grad-log-heat-06} by the equivalence of the Riemannian and Euclidean distances.
    
    Applying this bound to \eqref{eq:concentration-bound} gives for all $\eps > 0$ small enough
    \begin{align*}
        \int_{(M \times M) \setminus N} \norm{\nabla_{x} \log p_{\eps}(x,z)}_g \pi^{\eps}(dxdz) &\leq C \left(\frac{1}{\sqrt{\eps}}\pi^{\eps}((M \times M) \setminus N)+\frac{1}{\eps}\int_{(M \times M) \setminus N} d(x,z)\pi^{\eps}(dxdz)\right).
    \end{align*}
    The first term vanishes as $\eps \downarrow 0$ by \eqref{eq:vol-estimate}. For the right-most term, observe from H\"{o}lder that
    \begin{align*}
        \int_{(M \times M) \setminus N} d(x,z)\pi^{\eps}(dxdz) &\leq \left(\int_{M} d^2(x,z)\pi^{\eps}(dxdz)\right)^{1/2}\sqrt{\pi^{\eps}((M \times M) \setminus N)}
    \end{align*}
    By \eqref{eq:vol-estimate} and the $\Was{2}$ convergence of $\pi^{\eps}$ to $(\Id,\Id)_{\#}\mu$, the RHS is $o(\eps)$.  
    This establishes \eqref{eq:concentration-bound}, which can be written as
    \begin{align}\label{eq:for-ptwise-1}
        \lim\limits_{\eps \downarrow 0} \int_{M} \left[\int_{M \setminus B(x,c/2)} \norm{\nabla_{x} \log p_{\eps}(x,z)}_g \pi^{\eps}(dz|x) \right] \mu(dx) = 0.
    \end{align}
    Next, we claim that there exists a further subsequence as $\eps \downarrow 0$ such that for $\mu$-a.e.\ $x \in M$  
    \begin{align}\label{eq:leading-order-limit}
        \lim\limits_{\eps \downarrow 0} \int_{B(x,c/2)} \norm{\nabla \log p_{\eps}(x,z) - \frac{1}{\eps}\log_x(z)}_g \pi^{\eps}(dz|x) = 0.
    \end{align}
    To see this, recall the expansion from Proposition \ref{prop:heat-kernel-asymp}, which holds for all $z \in B(x,c/2)$
    \begin{align}
        \nabla_{x} \log p_{\eps}(x,z) = \frac{1}{\eps}\log_{x} z + \nabla_{x} \log c_{0}(x,z) + \nabla_{x} \log\left(1+\eps\frac{R(\eps,x,z)}{c_{0}(x,z)}\right).
    \end{align}
    Moreover, recall that $c_0(x,\cdot)$ is smooth, strictly positive on $B(x,c/2)$ and its log gradient evaluated on the diagonal is $0$ (Proposition \ref{prop:heat-kernel-asymp}). Similarly, recall that $R$ is smooth and that there are uniform upper bounds on $R$ and its gradient for all $z \in B(x,c/2)$ and small enough $\eps$. In particular, let $\ell_{\eps}$ denote the joint distribution of the diffusion approximation to $\pi^{\eps}$ developed in Theorem \ref{thm:sym-rel-ent}. As its conditional distributions $\ell_{\eps}(\cdot|x)$ converge weakly to $\delta_x$ as $\eps \downarrow 0$ for all $x \in M$, it holds from the above discussion that
    \begin{align}\label{eq:heat-leading-order-for-diff}
        &\lim\limits_{\eps \downarrow 0}\int_{B(x,c/2)} \norm{\nabla_x \log p_{\eps}(x,z) - \frac{1}{\eps}\log_x(z)}_g \ell_{\eps}(dz|x)\\
        &= \lim\limits_{\eps \downarrow 0}\int_{B(x,c/2)} \norm{\nabla_{x} \log c_{0}(x,z) + \nabla_{x} \log\left(1+\eps\frac{R(\eps,x,z)}{c_{0}(x,z)}\right)}_g \ell_{\eps}(dz|x) = 0. 
    \end{align}
    We now pass this fact to the conditional distributions of $\pi^{\eps}$. By the variational characterization of total variation and Pinsker's inequality, for each $x \in M$ there are constants $C_1, C_2 > 0$ such that
    \begin{align*}
        \abs{\int_{B(x,c/2)} \norm{\nabla_x \log p_{\eps}(x,z) - \frac{1}{\eps}\log_x(z)}_g (\pi^{\eps}(dz|x)-\ell_{\eps}(dz|x))} &\leq C_1 d_{TV}(\pi^{\eps}(\cdot|x),\ell_{\eps}(\cdot|x)) \\
        &\leq C_2 \sqrt{H(\pi^{\eps}(\cdot|x)|\ell_{\eps}(\cdot|x))}
    \end{align*}
    As $\Exp{\mu}[H(\pi^{\eps}(\cdot|x)|\ell_{\eps}(\cdot|x))] = H(\pi^{\eps}|\ell_{\eps})$, by Theorem \ref{thm:sym-rel-ent} there is a subsequence as $\eps \downarrow 0$ on which $H(\pi^{\eps}(\cdot|x)|\ell_{\eps}(\cdot|x)) \to 0$ as $\eps \downarrow 0$ for $\mu$-a.s.\ $x \in M$. This fact along with \eqref{eq:heat-leading-order-for-diff} establishes the existence of a subsequence along which for $\mu$-a.s. $x \in M$ \eqref{eq:leading-order-limit} holds. It then follows from \eqref{eq:for-ptwise-1} and \eqref{eq:leading-order-limit} that \eqref{eq:claim-throw-away-set} holds. 

    \textbf{Step 3:} Let $(X_t, t \geq 0)$ denote the diffusion approximation to $\pi^{\eps}$ developed in Theorem \ref{thm:sym-rel-ent}, then we claim for all $x \in M$ that
    \begin{align}\label{eq:ptwise-limit}
        \lim\limits_{\eps \downarrow 0} \frac{1}{\eps}\Exp{x}\left[\log_x X_{\eps}\right] = \frac{1}{2}\nabla_g \log\left(\frac{d\mu}{d\vol}\right)(x).  
    \end{align}
    Let $T = \inf\{t > 0: d(X_t,x) > c/2\}$ and fix a normal coordinate system about $x$, on which  $\log_{x}z = z-x$. In this coordinate system, the SDE for the diffusion can be written as
    \begin{align*}
        X_{\eps \wedge T} - x = \int_{0}^{\eps \wedge T} \left(\frac{1}{2}\nabla_g \log\left(\frac{d\mu}{d\vol}\right)(X_t) - \frac{1}{2}g^{ij}(X_t)\Gamma_{ij}^k(X_t)\right)dt + \int_0^{\eps \wedge T}g^{-1/2}(X_t)dW_t,  
    \end{align*}
    where $(W_t, t \geq 0)$ is $d$-dimensional Euclidean Brownian motion. As $B\left(x,\frac{c}{2}\right) \subset M$ is compact, the stochastic integral term is a martingale, giving that
    \begin{align}
        \frac{1}{\eps}\Exp{x}[\log_x X_{\eps \wedge T}] &= \frac{1}{\eps}\int_0^{\eps \wedge T} \Exp{x}\left[\frac{1}{2}\nabla_g \log\left(\frac{d\mu}{d\vol}\right)(X_t) - \frac{1}{2}g^{ij}(X_t)\Gamma_{ij}^k(X_t)\right] dt. 
    \end{align}
    Path continuity and the fact that $\Gamma_{ij}^{k}(x) = 0$ in normal coordinates based at $x$ establishes the pointwise limit $\lim\limits_{\eps \downarrow 0} \frac{1}{\eps}\Exp{x}[\log_x X_{\eps \wedge T}] = \frac{1}{2}\nabla_g \log \left(\frac{d\mu}{d\vol}\right)(x)$. 
    
    We now show that it suffices to consider the stopped process. For $\eps > 0$ and $x \in M$,
    \begin{align*}
        \norm{\log_x X_{\eps \wedge T}-\log_x X_{\eps}}_g &= 0 \cdot \mathbbm{1}(T > \eps) + \norm{\log_{x} X_{\eps \wedge T}-\log_x X_{\eps}}_g \mathbbm{1}(T \leq \eps) \\
        &\leq \mathbbm{1}(T\leq \eps) \left(d(x,X_{\eps \wedge T}) + d(x,X_{\eps})\right).
    \end{align*}
    As $M$ has lower bounded Ricci curvature, \cite[Proposition 3.7]{krajj-ldp-mani-19} furnishes $C_1, C_2 > 0$ such that for all $\eps > 0$ small enough 
    \begin{align*}
        \Exp{W_{x}^{\eps}}[\mathbbm{1}(T \leq \eps)] &\leq C_1 \exp\left(-C_2/\eps\right).
    \end{align*}
    Let $Q^{\eps}_x \in \cP(C([0,1],M))$ denote the law of the diffusion with generator $\eps\left(-\frac{1}{2}\nabla_g U_{\mu} \cdot \nabla_g + \frac{1}{2}\Delta_g\right)$ started from $x \in M$. From Proposition \ref{prop:rn-deriv-diffusion} and the assumption that $\cU_{\mu}$ is lower bounded, there is some $C(x) > 0$ such that $\frac{dQ^{\eps}_x}{dW^{\eps}_x} \leq C(x)$ so this bound transfers to the diffusion, giving that
    \begin{align*}
        \Exp{Q_{x}^{\eps}}[\norm{\log_x X_{\eps \wedge T}-\log_x X_{\eps}}_g] &\leq C(x) \Exp{W_{x}^{\eps}}[\mathbbm{1}(T\leq \eps)(d(x,X_{\eps \wedge T})+d(x,X_{\eps}))] \\
        &\leq C(x) \sqrt{\Exp{W_{x}^{\eps}}[\mathbbm{1}(T \leq \eps)]} \sqrt{\Exp{W_{x}^{\eps}}[(c+d(x,X_{\eps}))^2]}.
    \end{align*}
    As $M$ has lower bounded Ricci curvature, $\limsup\limits_{\eps \downarrow 0}\Exp{W_x^{\eps}}[d^2(x,X_{\eps})] < +\infty$ by \cite[Equation (3.6.1)]{hsu-stoch-analysis-manifold}. Thus, this upper bound is $o(\eps)$, 
    giving \eqref{eq:ptwise-limit}. 
    
    \textbf{Step 4:}
    Finally then, it remains to show that along a subsequence it holds $\mu$-a.s.\ that
     \begin{align}\label{eq:sb-cutoff-limit}
          \lim\limits_{\eps \downarrow 0} \int_{B\left(x,\frac{c}{2}\right)} \frac{1}{\eps} \log_x(z)\pi^{\eps}(dz|x) &= \frac{1}{2}\nabla_g \log \left(\frac{d\mu}{d\vol}\right).
     \end{align}
    To do this, we introduce a family of cutoff functions $(\chi_x, x \in M)$ such that each $\chi_x : M \to [0,1]$ is smooth with $\chi_x \equiv 1$ on $B(x,c/2)$ and $\chi_x \equiv 0$ on $M \setminus B(x,2c/3)$. For the $\Schro$ bridge, observe that
    \begin{align*}
        \frac{1}{\eps}\int_M \int_{M} \norm{\left(\mathbbm{1}(d(z,x) \leq c/2)-\chi_x(z)\right) \log_x(z)}_g\pi^{\eps}(dz|x) \mu(dx) &\leq \frac{1}{\eps} \cdot \frac{2c}{3} \pi^{\eps}\left((x,z): d(x,z) \in \left[\frac{c}{2},\frac{2c}{3}\right]\right) \\
        &\leq \frac{c}{\eps}\pi^{\eps}((M \times M) \setminus N),  
    \end{align*}
    which vanishes by \eqref{eq:vol-estimate}. Then, assuming these limits exist, along a subsequence it holds for $\mu$-a.s.\ $x\in M$ that 
    \begin{align}
        \lim\limits_{\eps \downarrow 0} \int_{B\left(x,\frac{c}{2}\right)} \frac{1}{\eps} \log_x(z)\pi^{\eps}(dz|x) &=  \lim\limits_{\eps \downarrow 0} \int_{M} \frac{1}{\eps} \chi_x(z)\log_x(z)\pi^{\eps}(dz|x).
    \end{align}
    Again, let $\ell_{\eps}$ denote the coupling constructed in Theorem \ref{thm:sym-rel-ent} and let $(\ell^{\eps}_x, x \in M)$ denote the conditional distributions of the second coordinate given the first is $x$. Apply the first part of Theorem \ref{thm:generator-transfer} to the function $z \mapsto \chi_x(z)\log_x(z)$ to get that
    \begin{align}\label{eq:limit-transfer-ptwise}
        \frac{1}{\eps}\norm{(\ell^{\eps}_x-\pi^{\eps}_x)[\chi_x(z)\log_x(z)]}_g \leq \frac{1}{\eps}\norm{\chi_x\log_x}_\mathrm{Lip} \sqrt{\frac{1-e^{-\kappa \eps}}{\kappa}H(\pi^{\eps}_x|\ell^{\eps}_x)}.
    \end{align}
    For the diffusion analyzed in Step 3, observe that
    \begin{align*}
        \frac{1}{\eps}\int_{M} \norm{\Exp{x}[(1-\chi_x(X_{\eps}))\log_{x}(X_{\eps})]}_g \mu(dx) &\leq \frac{1}{\eps}\Exp{}\left[\mathbbm{1}(d(X_0,X_{\eps})>c/2)\norm{\log_{X_0}(X_\eps)}_g\right] \\
        &\leq \frac{1}{\eps} \sqrt{\Exp{}[d^2(X_0,X_{\eps})]}\sqrt{\ell_{\eps}((M \times M) \setminus N)}.
    \end{align*}
    This bound vanishes by Propositions \ref{prop:diff-fourth-moment} and \ref{prop:almost-ldp-sb}. Hence, from Step 3 it holds $\mu$-a.s.\ along a subsequence that
    \begin{align}\label{eq:cutoff-log-diff-limit}
        \lim\limits_{\eps \downarrow 0} \frac{1}{\eps}\int_{M} \chi_x(z)\log_x(z) \ell_{\eps}(dz|x) &= \lim\limits_{\eps \downarrow 0} \frac{1}{\eps} \Exp{}[\log_x X_{\eps}] = \frac{1}{2}\nabla_g \log\left(\frac{d\mu}{d\vol}\right)(x).
    \end{align}

    By Theorem \ref{thm:sym-rel-ent}, \eqref{eq:limit-transfer-ptwise}, and \eqref{eq:cutoff-log-diff-limit}, along a subsequence it holds for $\mu$-a.s.\ $x \in M$ that
     \begin{align*}
         \lim\limits_{\eps \downarrow 0}\frac{1}{\eps}\int_{M} \chi_{x}(z)\log_{x}(z) \pi^{\eps}(dz|x) &= \lim\limits_{\eps \downarrow 0} \frac{1}{\eps}\ell_{x}^{\eps}[\chi_x(z)\log_x(z)] = \frac{1}{2}\nabla_g \log \left(\frac{d\mu}{d\vol}\right). 
     \end{align*}
     Along with Step 1 (and perhaps taking a further sequence), this establishes \eqref{eq:major-claim}. 

     Finally then it remains to establish \eqref{eq:manifold-bp}. This is simply a consequence of applying Theorem \ref{thm:generator-transfer} in global coordinates to the smooth functions $\xi_{i}(x) := x_i$ for all $i \in \{1,\dots,d\}$. Observe that each $\xi_i$ is Lipshitz with respect to $d$ as the Hessian distance is equivalent to the Euclidean distance. In the global coordinate system,
     \begin{align*}
         \cL \xi_i(x) &= \frac{1}{2}\langle \nabla_g \log \left(\frac{d\mu}{d\vol}\right), \nabla_g \xi_{i}\rangle_{g}(x) + \frac{1}{2}\Delta_{g}\xi_{i}(x) \\
         &= \frac{1}{2} \frac{\partial}{\partial x^{u}}\log \left(\frac{d\mu}{d\vol}\right)(x)g^{uv}(x)\frac{\partial}{\partial x^{v}}\xi_i(x) + \frac{1}{2}g^{uv}(x)\frac{\partial^2}{\partial x^u \partial x^{v}}\xi_{i}(x) - \frac{1}{2}g^{uv}\Gamma_{uv}^{\ell}\frac{\partial}{\partial x^{\ell}}\xi_{i}(x) \\
         &= \frac{1}{2}g^{iu}\frac{\partial}{\partial x^{u}}\log \left(\frac{d\mu}{d\vol}\right)(x)-\frac{1}{2}g^{uv}\Gamma_{uv}^{i}.
     \end{align*}
     Now, suppose in addition that  $g = \nabla^2 \varphi(x)$ for a smooth strictly convex function $\varphi$. Then \eqref{eq:hess-manifold-bp} is established from the following identity. The Christoffel symbols for Hessian manifolds are given in \eqref{eq:cris-sym-primal}, from which it follows that
     \begin{align*}
         \frac{1}{2} \nabla_g \log \det g   &= \frac{1}{2}g^{ij}g^{uv}\frac{\partial}{\partial x^j} g_{uv} = g^{uv}\Gamma_{uv}^{j}.
     \end{align*}
\end{proof}

\section{Towards Different Marginals}\label{sec:diff-marg}
We now specialize to the case of the Hessian manifold. For specificity, we develop notation that is used only in this section. 

Recall the notation for Hessian manifolds developed in Section \ref{subsec:geom}. In particular, for a convex function $\varphi: \mathbb{R}^{d} \to \mathbb{R}$, $\varphi^*$ denotes its convex conjugate. We denote the primal coordinates by $x \in \mathbb{R}^{d}$ and the dual coordinates by $x^* = \nabla \varphi(x) \in \mathbb{R}^{d}$. For a measure $\mu \in \cP_2(M)$ satisfying Assumption \ref{assumption:hessian-manifold} below, denote its Lebesgue density in primal and dual coordiantes by $e^{-f}$ and $e^{-h}$, respectively. Observe that $e^{-h} = (\nabla \varphi)_{\#}e^{-f}$, so that $\nabla \varphi$ is the Brenier map from $e^{-f}$ to $e^{-h}$. By the change of variables formula, the following crucial identity holds
\begin{align}\label{eq:change-of-variables}
    h(x^*) = f(x) + \log \det \nabla^2 \varphi(x). 
\end{align}
To summarize the notational differences, in this section we write $\vol^{\varphi}(dx) = \sqrt{\det \nabla^2 \varphi(x)}dx$ to refer to the volume measure of the Hessian manifold in primal coordinates. To refer to Lebesgue measure on $\mathbb{R}^{d}$, we write $\Leb$. With regards to information theoretic quantities, $\Ent(\eta) = H(\eta|\Leb)$ for all $\eta \in \cP(\mathbb{R}^{d})$. To refer to the gradient and vector norm on the Hessian manifold, we write $\nabla_g$ and $\norm{\cdot}_g$, respectively. For their Euclidean counterparts we write $\nabla$ and $\norm{\cdot}$, i.e.\ with no subscript. In this way, we distinguish between Fisher information on the Hessian manifold versus that on Euclidean space by writing, respectively,
\begin{align*}
    I_{\varphi}(\alpha|\beta) &= \Exp{\alpha}\norm{\nabla_g \log \left(\frac{d\alpha}{d\beta}\right)}^2_{g} \text{ and } I(\alpha|\beta) = \Exp{\alpha}\norm{\nabla \log \left(\frac{d\alpha}{d\beta}\right)}^2.
\end{align*} 
From this moment on, we identify the Hessian manifold $M$ with $\mathbb{R}^{d}$, making it clear when we are performing computations on the manifold or on $\mathbb{R}^{d}$. Let $(B_t^{\varphi}, t \geq 0)$ denote the $\mathbb{R}^{d}$-valued stochastic process obtained by writing the Hessian Brownian motion in \textbf{primal coordinates} running in stationarity. Set $R_{0\eps}^{\varphi} = \text{Law}(B_0^{\varphi},B_\eps^{\varphi})\in \cM_{+}(\mathbb{R}^{d}\times \mathbb{R}^{d})$ and in this primal coordinate system define $r_{\eps}: \mathbb{R}^{d} \times \mathbb{R}^{d} \to (0,+\infty)$ by
\begin{align}\label{eq:r-eps-primal-primal}
    r_{\eps}(x,z) = \frac{dR_{0\eps}^{\varphi}}{d(\vol^{\varphi} \otimes \vol^{\varphi})}(x,z) \sqrt{\det \nabla^2 \varphi(z)}.
\end{align}
Observe carefully that this is a different convention than used in previous sections. The reason for this is that we treat $r_{\eps}$ as an transition density over $\mathbb{R}^{d}$. In particular, observe that now $\int_{\mathbb{R}^{d}} r_{\eps}(x,z)dz = 1$ for all $x \in \mathbb{R}^{d}$. Of course, $dz$ here means $\Leb(dz)$. 

We make the following assumptions on the Hessian manifold. Recall from Corollary \ref{cor:non-explosion-mld} that in primal and dual coordinates, respectively, $U_{\mu} = -\log(d\mu/d\vol)$ writes as
\begin{align*}
    U_{\mu} = f(x) + \frac{1}{2}\log \det \nabla^2 \varphi(x) = h(x^*) + \frac{1}{2}\log \det \nabla^2 \varphi^*(x^*). 
\end{align*}
\begin{assumption}[Hessian Manifold]\label{assumption:hessian-manifold}
    We assume that $\varphi: \mathbb{R}^{d} \to \mathbb{R}$ is a smooth, strictly convex function and $\mu \in \cP_2(M)$ are such that
    \begin{itemize}
        \item[(1)] There exists $\alpha, \beta > 0$ such that $\alpha \Id \leq \nabla^2 \varphi(x) \leq \beta \Id$ for all $x \in \mathbb{R}^{d}$.
        \item[(2)] All third and fourth derivatives of $\varphi$ are bounded over $\mathbb{R}^{d}$.
        % \GM{from \cite[page 7]{kolesnikov-hessian-metric} $\ricci$ is actually only of function of second and third derivatives, so just need bounded third derivative? But bdd fourth derivative gives global error estimates for $d^2$ and $D\left[\cdot|\cdot\right]$}
        \item[(3)] $H(\mu|\vol^{\varphi})$ is finite, $\Exp{\mu}\left[\norm{\nabla_g U_{\mu}}_g^4\right] < +\infty$, $\inf \limits_{M}\cU_{\mu} > -\infty$, and $(M,g,\mu)$ satisfies $\mathrm{CD}(\kappa_{\mu},+\infty)$ for some $\kappa_{\mu} \in \mathbb{R}$. 
        \item[(4)] The limit \eqref{eq:o-eps-2-limit} holds with Wiener reference measure ($\diffref = \vol^{\varphi}$ so that $U_{\diffref} = \cU_{\diffref} = 0$). By Proposition \ref{prop:h2-int-argument}, a sufficient condition is that $\mu$ has subexponential tails and $\norm{\nabla_g \cU_{\mu}}^2_{g}$ has polynomial growth. 
    \end{itemize}
\end{assumption}
Observe that this setting satisfies Assumption \ref{assumption:manifold} (H2). In particular, from the Ricci tensor computation in \cite[Equation (3.9)]{kolesnikov-hessian-metric}, the global bounds on the second and third derivatives of $\varphi$ guarantee that the Ricci tensor is lower bounded by some constant. That is, $(M,g,\vol)$ satisfies $\mathrm{CD}(\kappa_{\vol},d)$ for some $\kappa_{\vol} \in \mathbb{R}$. We remark that \cite[Theorem 4.3]{kolesnikov-hessian-metric} gives a sufficient to conclude that the curvature dimension condition on $(M,g,\mu)$ holds. Additionally,  $I_{\varphi}(\mu|\vol^{\varphi}) = \Exp{\mu}[\norm{\nabla_g U_{\mu}}_g^2]$ is finite by Assumption \ref{assumption:hessian-manifold} (3). 

Towards the objective outlined in the Introduction above the informal statement of Theorem \ref{thm:mld-quad-sb-comparison}, pushforward the second coordinate of $r_{\eps}$ defined in \eqref{eq:r-eps-primal-primal} by $\nabla \varphi$ to obtain $\bar{r}_{\eps}: \mathbb{R}^{d} \times \mathbb{R}^{d} \to (0,+\infty)$ defined as 
\begin{align}\label{eq:r-eps-primal-dual}
    \bar{r}_{\eps}(x,y) := r_{\eps}(x,\nabla \varphi^*(y))\det \nabla^2 \varphi^*(y) = \frac{dR_{0\eps}^{\varphi}}{d(\vol^{\varphi} \otimes \vol^{\varphi})}(x,\nabla \varphi^*(y)) \sqrt{\det \nabla^2 \varphi^*(y)}. 
\end{align}
This is the reference measure on $\mathbb{R}^{d} \times \mathbb{R}^{d}$ we will consider. In other words, we are computing entropic optimal transport with the following cost function on $\mathbb{R}^{d} \times \mathbb{R}^{d}$
\begin{align}\label{eq:pushforward-cost}
    -\eps\log \bar{r}_{\eps}(x,y) = \frac{1}{2}d^2_{\varphi}(x,\nabla \varphi^*(y)) + \frac{\eps}{2} \log \det \nabla^2 \varphi^*(y) + \text{(Lower order terms)}.
\end{align}
As Proposition \ref{prop:hess-dist-symm-breg} details, the leading order term in the above cost function is best interpreted as an information geometric quantity. It has the following expansion, where $\cdot$ denotes the standard Euclidean inner product (and will for this section),
\begin{align}\label{eq:expansion-riemann-dist-symm-div}
    \frac{1}{2}d_{\varphi}^2(x,\nabla \varphi^*(y)) &= \frac{1}{2}(y-x^*)\cdot (\nabla \varphi^*(y)-x)+O\left(\norm{x-\nabla \varphi^*(y)}^4\right).
\end{align}
In other words, by taking the pushforward in second coordinate we shift the underlying geometry on the Hessian manifold. While the cost function from $r_{\eps}$ reflects the Hessian manifold equipped with its Levi-Civita connection (in primal coordinates), the cost function derived from $\bar{r}_{\eps}$ reflects the Hessian manifold equipped with its dually flat affine connection \cite{shima2007geometry}. That is, recall the Bregman divergence corresponding to $\varphi$, defined by
\begin{align}\label{eq:bregman-div-varphi}
    D_{\varphi}[z|x] = \varphi(z)-\varphi(x) - \nabla \varphi(x) \cdot (z-x) = \varphi(z)-\varphi(x) - x^* \cdot (z-x).
\end{align}
Observe that the symmetrized Bregman divergence is then
\begin{align}\label{eq:symm-bregman-div-varphi}
    D_{\varphi}[z|x]+D_{\varphi}[x|z] = (z^*-x^*)\cdot (z-x),
\end{align}
which matches the leading order term on the RHS of \eqref{eq:pushforward-cost} with $y = z^*$.

The cost function in \eqref{eq:pushforward-cost} is continuous and satisfies the necessary integrability properties outlined in \cite[Theorem 4.2]{nutz-notes} to guarantee the existence of the $\Schro$ bridge from $e^{-f}$ to $e^{-h}$ with respect to this cost and well as its product structure. We now show that this $\Schro$ bridge is a simple transformation of a same-marginal $\Schro$ bridge computed on the Hessian manifold.  

Let $\pi^{\eps}$ denote the same-marginal $\Schro$ bridge from $\mu$ to itself with respect to the Hessian BM in primal coordinates. Then in primal coordinates $\pi^{\eps} \in \Pi(e^{-f},e^{-f})$ and has Lebesgue density given by $\pi^{\eps}(x,z) = a^{\eps}(x)a^{\eps}(z)r_{\eps}(x,z)\vol^{\varphi}(x)$. Now, pushforward the second coordinate by $\nabla \varphi$ to obtain
\begin{align}\label{eq:sb-dual-primal}
    \bar{\pi}_{\eps} := (\Id,\nabla \varphi)_{\#}\pi^{\eps} \in \Pi(e^{-f},e^{-h}).
\end{align}
Now, observe that 
\begin{align}
    \bar{\pi}_{\eps}(x,y) &= \pi^{\eps}(x,\nabla\varphi^*(y))\det \nabla^2 \varphi^*(y) 
    = a^{\eps}(x)a^{\eps}(\nabla \varphi^*(y))\bar{r}_{\eps}(x,y)\vol^{\varphi}(x). 
\end{align}
As it possesses the necessary product structure, $\bar{\pi}_{\eps}$ is the $\eps$-static $\Schro$ bridge from $e^{-f}$ to $e^{-h}$ with respect to $\bar{r}_{\eps}$. As in \eqref{eq:MLD}, let $(X_t, t \geq 0)$ denote the $\mathbb{R}^{d}$-valued process that is the Mirror Langevin diffusion corresponding to $\mu$ in primal coordinates. Set $\ell_{\eps} := \text{Law}(X_0,X_{\eps}) \in \Pi(e^{-f},e^{-f})$, and pushforward the second coordinate by $\nabla \varphi$ to obtain 
\begin{align}\label{eq:mld-dual-primal}
    \bar{\ell}_{\eps} := (\Id,\nabla \varphi)_{\#}\ell_{\eps} = \mathrm{Law}(X_0,X_\eps^*) \in \Pi(e^{-f},e^{-h}).
\end{align}

Recall that Theorem \ref{thm:sym-rel-ent} bounds $H(\ell_{\eps}|\pi^{\eps})+H(\pi^{\eps}|\ell_{\eps})$. As $(\Id,\nabla \varphi)$ is invertible, by \cite[Lemma 9.4.5]{ambrosio2005gradient} it holds that $H(\bar{\ell}_{\eps}|\bar{\pi}_{\eps}) = H(\ell_{\eps}|\pi^{\eps})$ and $H(\bar{\pi}_{\eps}|\bar{\ell}_{\eps}) = H(\pi^{\eps}|\ell_{\eps})$.
Hence, the exact same bound for their sums as in Theorem \ref{thm:sym-rel-ent} holds. We summarize this below.
\begin{align}\label{prop:symm-ent-pushforward}
    H(\bar{\ell}_{\eps}|\bar{\pi}_{\eps})+H(\bar{\pi}_{\eps}|\bar{\ell}_{\eps}) &\leq  \frac{1}{2}\eps^2 \left(\int_0^{1} \norm{\nabla_g \cU_{\mu}}_{L^{2}(\mu_t^{\eps})}^{2}\right)^{1/2}\left(I_{\varphi}(\mu|\vol^{\varphi})-\int_{0}^{1} I_{\varphi}(\mu_{t}^{\eps}|\vol^{\varphi})dt\right)^{1/2}.
    \end{align}
Note that the left hand side contains quantities obtained by mixing the dual and primal coordinate systems, whereas the upper bound is independent of any coordinate system (i.e.\ it is a geometrically invariant property of the manifold and choice of $\mu$).  

We now use \eqref{prop:symm-ent-pushforward} to quantify the extent to which $\bar{\ell}_{\eps} \in \Pi(e^{-f},e^{-h})$ approximates the (one-half) \textbf{quadratic cost} Euclidean $\Schro$ bridge from $e^{-f}$ to $e^{-h}$. That is, let $q_{\eps}(x,y) :=  (2\pi\eps)^{-d/2}\exp\left(-\norm{x-y}^2/2\eps\right)$ and let $\Pi_{\eps} \in \Pi(e^{-f},e^{-h})$ be such that
\begin{align}\label{eq:euclidean-sb}
    \Pi_{\eps} := \argmin\limits_{\pi \in \Pi(e^{-f},e^{-h})} H(\pi|q_{\eps}).
\end{align}

\begin{theorem}\label{thm:mld-quad-sb-comparison}
    Under Assumption \ref{assumption:hessian-manifold}, let $\bar{\ell}_{\eps}$ be as defined in \eqref{eq:mld-dual-primal}. Then 
    \begin{align}
        \eps H(\bar{\ell}_{\eps}|q_{\eps}) = \frac{1}{2}\Was{2}^2(e^{-f},e^{-h})+\frac{\eps}{2}(\Ent(e^{-f})+ \Ent(e^{-h}))+o(\eps). 
    \end{align}
    Let $\Pi_{\eps}$ be as in \eqref{eq:euclidean-sb}, then
    \begin{align}\label{eq:mld-sb-approx}
        \lim\limits_{\eps \downarrow 0}H(\bar{\ell}_{\eps}|\Pi_{\eps}) &= 0. 
    \end{align}
\end{theorem}   

\begin{proof}
    Observe that
    \begin{align}\label{eq:quad-cost-split-mld}
        \eps H(\bar{\ell}_{\eps}|q_{\eps}) &= \eps H(\bar{\ell}_{\eps}|\bar{r}_{\eps}) + \eps \Exp{\bar{\ell}_{\eps}}\left[\log\left(\frac{r_{\eps}(x,\nabla \varphi^*(y))\det \nabla^2 \varphi^*(y)}{q_{\eps}(x,y)}\right)\right].
    \end{align}
    As computed previously in \eqref{eq:static-ell-eps-mu-reps},
    \begin{align*}
        H(\bar{\ell}_{\eps}|\bar{r}_{\eps}) &= H(\ell_{\eps}|r_{\eps}) = \Exp{e^{-f}}[-f]+H(\ell_{\eps}|e^{-f(x)}r_{\eps}(x,z)) \leq \Ent(e^{-f})+\frac{\eps}{8}I_{\varphi}(\mu|\vol^{\varphi}).
    \end{align*}
    It remains then to analyze the rightmost term in \eqref{eq:quad-cost-split-mld}. First, recall that
    \begin{align*}
        \eps \log r_{\eps}(x,z) &= \eps\log\left(\frac{dR_{0\eps}^{\varphi}}{d(\vol^{\varphi} \otimes \vol^{\varphi})}(x,z)\right)+\frac{\eps}{2}\log \det \nabla^2 \varphi(z).
    \end{align*}
    For shorthand, set $r_{\eps}^{\varphi}(x,z) := \frac{dR_{0\eps}^{\varphi}}{d(\vol^{\varphi} \otimes \vol^{\varphi})}(x,z)$. Let $(X_t, t \geq 0)$ denote the primal \eqref{eq:MLD}. Then from \eqref{eq:change-of-variables} it holds that
    \begin{align*}
        &\eps \Exp{\bar{\ell}_{\eps}}\left[\log\left(\frac{r_{\eps}(x,\nabla \varphi^*(y))\det \nabla^2 \varphi^*(y)}{q_{\eps}(x,y)}\right)\right]
        &= \eps\Exp{}\left[\log\left(\frac{r_{\eps}^{\varphi}(X_0,X_{\eps})}{q_{\eps}(X_0,X_{\eps}^*)}\right)\right]+\frac{\eps}{2}(\Ent(e^{-h})-\Ent(e^{-f})).
    \end{align*}
    The claim of the theorem then holds once we show
    \begin{align}\label{eq:heat-kernel-comp}
    \Exp{}\left[\log\left(\frac{r_{\eps}^{\varphi}(X_0,X_{\eps})}{q_{\eps}(X_0,X_{\eps}^*)}\right)\right] &= \frac{1}{2\eps}\Was{2}^2(e^{-f},e^{-h})+o(1).
    \end{align} 
    Now, observe that the left hand side of \eqref{eq:heat-kernel-comp} can be written as 
    \begin{align}
        \Exp{}\left[\log\left((2\pi\eps)^{d/2}r_{\eps}^{\varphi}(X_0,X_{\eps})\right)+\frac{1}{2\eps}d^2(x,z)\right]+\Exp{}\left[\frac{1}{2\eps}\norm{X_0-X_{\eps}^*}^2-\frac{1}{2\eps}d^2(X_0,X_{\eps})\right].
    \end{align}
    
First, we analyze the term inside the second expectation. Recall the Bregman divergence $D_{\varphi}[\cdot|\cdot]$ introduced in \eqref{eq:bregman-div-varphi}. From Proposition \ref{prop:hess-dist-symm-breg}, there is a $K> 0$ such that for all $x \in \mathbb{R}^{d}$ and $z \notin C_{x}$, 
\begin{align}
    \abs{d^2_{\varphi}(x,z) - \left(D_{\varphi}[x|z]+D_{\varphi}[z|x]\right)} \leq K\norm{x-z}^4
\end{align}
Recall that $\vol(C_x) = 0$ and thus $\mu(C_x) = 0$ for all $x \in \mathbb{R}^{d}$. Next, we compute
\begin{align*}
    \Exp{}\left[D_{\varphi}[X_\eps|X_0]+D_{\varphi}[X_0|X_{\eps}]-\norm{X_\eps^*-X_0}^2\right] &= \Exp{}\left[(X_{\eps}-X_0)\cdot (X_{\eps}^*-X_0^*)-\norm{X_\eps^*-X_0}^2\right] \\
    % &= \Exp{}\left[X_{\eps}^* \cdot X_\eps +X_0 \cdot X_0^* -\norm{X_{\eps}^*}^2-\norm{X_0}^2\right] \\
    &= \Exp{}\left[X_{\eps}^*\cdot (X_{\eps}-X_\eps^*)+X_0 \cdot (X_0^*-X_0)\right]\\
    &= \Exp{}\left[(X_0^*-X_0) \cdot (X_0 - X_0^*)\right] = -\Was{2}^2(e^{-f},e^{-h}).
\end{align*}
For the second line, observe that $\Exp{}\left[-X_{\eps}\cdot X_0^* - X_0 \cdot X_{\eps}^* + 2X_0 \cdot X_{\eps}^*\right] = 0$ by the reversibility of the MLD. The last line follows from stationarity. From Proposition \ref{prop:diff-fourth-moment} and the fact that $\Exp{\mu}[\norm{\nabla_g U_{\mu}}_g^{4}] < +\infty$, it holds that $\Exp{}\norm{X_\eps-X_0}^4 = O(\eps^2)$ and thus
\begin{align*}
    \Exp{}\left[-\frac{1}{2}d^2_{\varphi}(X_0,X_{\eps})+\frac{1}{2}\norm{X_{\eps}^*-X_0}^2\right] = \frac{1}{2}\Was{2}^2(e^{-f},e^{-h}) + O(\eps^2).
\end{align*}
Now, we will show that
\begin{align}\label{eq:heat-kernel-d2-bdd}
    \lim\limits_{\eps \downarrow 0}\Exp{}\left[\log\left((2\pi\eps)^{d/2}r_{\eps}^{\varphi}(X_0,X_{\eps})\right)+\frac{1}{2\eps}d^2(X_0,X_{\eps})\right] = 0.
\end{align}
Fix $\eta > 0$, and let $K \subset \mathbb{R}^{d}$ be a compact set such that $\mu(K) > 1 - \eta$. Let $c = \inf\limits_{x \in K} \mathrm{inj}(x) > 0$. Define the following three subsets of $\mathbb{R}^{d} \times \mathbb{R}^{d}$:
\begin{align*}
    U_1 &= \{(x,z) \in \mathbb{R}^{d} \times \mathbb{R}^{d}: d(x,z) > c/2\} \\
    U_2 &= \{(x,z) \in \mathbb{R}^{d} \times \mathbb{R}^{d}: x,z \in K, d(x,z) \leq c/2\} \\
    U_{3} &= (\mathbb{R}^{d}\times \mathbb{R}^{d}) \setminus (U_1 \cup U_2).
\end{align*}
On $U_2$, we will use the pointwise expansion of the heat kernel in Proposition \ref{prop:heat-kernel-asymp}, as $U_2$ is a compact set that does not intersect the cut locus. On $U_1$ and $U_3$, we will use the following global bound from \cite[Theorem 1.2]{jiang16heatkernel} 
    \begin{align}\label{eq:heat-kernel-majorization}
        \abs{\log\left((2\pi\eps)^{d/2}r_{\eps}^{\varphi}(x,z)\right)+\frac{1}{2\eps}d^2(x,z)} &\leq K\left(\frac{1}{\eps}d^2(x,z)+1\right)
    \end{align}
    where $K > 0$ holds for all $x,z \in M$ and $\eps > 0$ small enough. Observe that the existence of a $K$ for all $x,z \in M$ owes to the global positive upper and lower bounds on the Hessian in Assumption \ref{assumption:hessian-manifold}. We also use the following fourth moment bound from Proposition \ref{prop:diff-fourth-moment} thanks to the assumption that $\Exp{\mu}\left[\norm{\nabla_g U_{\mu} }^{4}_{g}\right] < +\infty$,
\begin{align}\label{eq:diff-fourth-moment-bdd}
    \Exp{}[d^{4}(X_0,X_{\eps})] \leq K\eps^2. 
\end{align}
On $U_1$, observe that
\begin{align*}
    \int_{U_1} \abs{\log\left((2\pi\eps)^{d/2}r_{\eps}^{\varphi}(x,z)\right)+\frac{1}{2\eps}d^2(x,z)} &\ell_{\eps}(dxdz) \leq \frac{K}{\eps} \int_{U_1} d^2(x,z) \ell_{\eps}(dxdz) + K\ell_{\eps}(U_1) \\
    &\leq \frac{K}{\eps}\left(\int_{\mathbb{R}^{d} \times \mathbb{R}^{d}}d^4(x,z)\ell_{\eps}(dxdz)\right)^{1/2}\sqrt{\ell_{\eps}(U_1)}+K\ell_{\eps}(U_1).
\end{align*}
By \eqref{eq:diff-fourth-moment-bdd} and \eqref{eq:diff-off-diag-concen}, the above quantity vanishes as $\eps \downarrow 0$. Similarly, on $U_3$ we obtain
\begin{align*}
    \int_{U_3} \abs{\log\left((2\pi\eps)^{d/2}r_{\eps}^{\varphi}(x,z)\right)+\frac{1}{2\eps}d^2(x,z)} \ell_{\eps}(dxdz) &\leq \frac{K}{\eps}\left(\int_{\mathbb{R}^{d} \times \mathbb{R}^{d}}d^4(x,z)\ell_{\eps}(dxdz)\right)^{1/2}\sqrt{\ell_{\eps}(U_3)}+K\ell_{\eps}(U_3) \\
    &\leq K'\sqrt{\eta}+K\eta.
\end{align*}
Recall the expansion from Proposition \ref{prop:heat-kernel-asymp}, which holds for all $(x,z) \in U_2$. In this case,
\begin{align*}
    \int_{U_2} \abs{\log\left((2\pi\eps)^{d/2}r_{\eps}^{\varphi}(x,z)\right)+\frac{1}{2\eps}d^2(x,z)} \ell_{\eps}(dxdz) &=  \int_{U_2} \abs{\log c_0(x,z) + \log\left(1+\eps \frac{R(\eps,x,z)}{c_0(x,z)}\right)} \ell_{\eps}(dxdz).
\end{align*}
As $c_0$ is smooth on $U_2$, $c_0(x,x) = 1$ for all $x \in \mathbb{R}^{d}$, and $R(\eps,x,z)$ remains bounded over all $x,z \in U_2$ as $\eps$ vanishes, it follows that the above quantity vanishes as $\eps \downarrow 0$. Altogether then,
\begin{align*}
    \limsup\limits_{\eps \downarrow 0} \int_{\mathbb{R}^{d} \times \mathbb{R}^d} \abs{\log\left((2\pi\eps)^{d/2}r_{\eps}^{\varphi}(x,z)\right)+\frac{1}{2\eps}d^2(x,z)} \ell_{\eps}(dxdz) \leq K'\sqrt{\eta}+K\eta.
\end{align*}
As $\eta > 0$ is arbitrary, this establishes \eqref{eq:heat-kernel-d2-bdd}.

Now, let $\Pi_{\eps} \in \Pi(e^{-f},e^{-h})$ denote the (half)-quadratic cost $\eps$-$\Schro$ bridge as defined in \eqref{eq:euclidean-sb}. Observe that the relative entropy between $\bar{\ell}_{\eps}$ and $\Pi_{\eps}$ can be decomposed in the following manner
    \begin{align*}
        H(\bar{\ell}_{\eps}|\Pi_{\eps}) &= \Exp{\bar{\ell}_{\eps}}\left[\log\left(\frac{d\bar{\ell}_{\eps}}{dq_{\eps}}\right)\right]-\Exp{\bar{\ell}_{\eps}}\left[\log\left(\frac{d\Pi_{\eps}}{dq_{\eps}}\right)\right] = \Exp{\bar{\ell}_{\eps}}\left[\log\left(\frac{d\bar{\ell}_{\eps}}{dq_{\eps}}\right)\right]-\Exp{\Pi_{\eps}}\left[\log\left(\frac{d\Pi_{\eps}}{dq_{\eps}}\right)\right] \\
        &= H(\bar{\ell}_{\eps}|q_{\eps}) - H(\Pi_{\eps}|q_{\eps}).
    \end{align*}
The second equality holds as $\log\left(d\Pi_{\eps}/dq_{\eps}\right)$ is the sum of a function of $x$ with a function of $y$. As $\Pi_{\eps}$ and $\bar{\ell}_{\eps}$ are both couplings of $e^{-f}$ and $e^{-h}$, we can swap the measure with respect to which we take expectation. Applying Proposition \ref{prop:sb-cost-exp-rd} to the previously computed expansion of $H(\bar{\ell}_{\eps}|q_{\eps})$ establishes \eqref{eq:mld-sb-approx}.
\end{proof}

\begin{remark}[Justification for \eqref{eq:conj-o-eps} in Conjecture]\label{remark:justification-conjecture}
    Observe that in the above proof of Theorem \ref{thm:mld-quad-sb-comparison}, the rates of convergence of the integrals over $U_1$ and $U_2$ are both $O(\eps)$. It is the inability to quantify the decay rate of the integral over $U_3$ that results in the non-quantified vanishing rate in \eqref{eq:mld-sb-approx}. With a better understanding of the heat kernel expansion in Proposition \ref{prop:heat-kernel-asymp} in the Hessian manifold setting (for instance, obtaining a remainder that holds globally), we believe this technicality can be circumvented and establish the convergence rate in Conjecture \ref{conjecture:O-eps}. 
\end{remark}

\subsection{Affine Optimal Transport}\label{subsec:affine}
Observe that the diffusion approximation provided in Theorem \ref{thm:mld-quad-sb-comparison} is two orders of magnitude weaker than that of Theorem \ref{thm:sym-rel-ent}. Even the rate in Conjecture \ref{conjecture:O-eps} is too weak to apply the generator transfer result of Theorem \ref{thm:generator-transfer}. We conclude this section with calculations that appear to indicate the geometric obstruction at play: local changes in curvature presented by the Hessian manifold.

We now insist that $\nabla \varphi$ is an affine function, i.e.\ that the Hessian geometry is flat. In this setting, the following symmetric relative entropy bound holds.
\begin{proposition}\label{prop:affine-brenier}
Let $\mu = e^{-f} \in \cP_2(\mathbb{R}^{d})$ be such that $f \in C^{3}(\mathbb{R}^{d})$, $f(x) \to +\infty$ as $x \to +\infty$, $\Ent(e^{-f})$ and $I(e^{-f}|\mathrm{Leb})$ are finite, and $e^{-f}$ has subexponential tails. Additionally, assume that $\inf \cU_{\mu} > -\infty$ and, with $\norm{\cdot}$ denoting the Euclidean norm,
\begin{align}
    \abs{\cU_{\mu}(x)} \leq C(1+\norm{x}^{N}) \text{ for all $x \in \mathbb{R}^{d}$}
\end{align}
for some $C > 0$ and $N \geq 1$. Now, let $S \in \mathbb{R}^{d \times d}$ be positive definite, $m \in \mathbb{R}^{d}$, and set $\nabla \varphi(x) := Sx+m$. Define $e^{-h} = (\nabla \varphi)_{\#}e^{-f}$. Let $\Pi_{\eps},\bar{\ell}_{\eps} \in \Pi(e^{-f},e^{-h})$ be as in Theorem \ref{thm:mld-quad-sb-comparison}. Then it holds that
    \begin{align*}
        \lim\limits_{\eps \downarrow 0}\frac{1}{\eps}\left(H(\bar{\ell}_{\eps}|\Pi_{\eps})+H(\Pi_{\eps}|\bar{\ell}_{\eps})\right) = 0. 
    \end{align*}
\end{proposition}

\begin{proof}
This proof is follows the same structure as that of Theorem \ref{thm:sym-rel-ent}. Recall that there exists functions $a^{\eps},b^{\eps}: \mathbb{R}^{d} \to (0,+\infty)$ such that the Lebesgue density of $\Pi_{\eps}$ writes as
\begin{align}
    \Pi_{\eps}(x,y) &= a^{\eps}(x)b^{\eps}(y)q_{\eps}(x,y).
\end{align}
To write the Lebesgue density of $\bar{\ell}_{\eps}$ (note that Corollary \ref{cor:non-explosion-mld} guarantees non-explosion), define \[q_{\eps}^{S^{-1}}(x,z) := (2\pi \eps)^{-d/2}(\det S)^{1/2}\exp\left(-\frac{1}{2\eps}(z-x)^{T}S(z-x)\right)\] and the two functions
\begin{align}
    c(x,z,\eps) &= -\log\left(\Exp{W_x^{S^{-1}}}\left[\exp\left(-\int_0^{\eps} \cU_{\mu}(\omega_t)dt\right)|\omega_{\eps}=z\right]\right) \\
    R(x,z,\eps) &= -\log \left(\Exp{W_x^{S^{-1,\eps}}}\left[\exp\left(-\int_0^{1} \eps\left(\cU_{\mu}(\omega_t)-\cU_{\mu}(\omega_0)\right)dt\right)|\omega_{1}=z\right]\right),
\end{align}
where $W_x^{S^{-1},\eps}$ is the law of $dX_t = \sqrt{\eps}S^{-1/2}dB_t$ started from $x \in \mathbb{R}^{d}$ (when $\eps$ is suppressed, set $\eps = 1$). Applying a change of variables to the same Girsanov argument from Theorem \ref{thm:sym-rel-ent}, the Lebesgue density of $\bar{\ell}_{\eps}$ is then (where $z = \nabla \varphi^*(y)$)
\begin{align}\label{eq:bar-ell-leb-dens}
    \bar{\ell}_{\eps}(x,y) 
    &= \mu(x)q_{\eps}^{S^{-1}}(x,z)\exp(-c(x,z,\eps))\det \nabla^2 \varphi^*(y).
\end{align}
The log likelihood ratio between the two densities is then
\begin{align*}
    \log \left(\frac{\bar{\ell}_{\eps}(x,y)}{\Pi_{\eps}(x,y)}\right) &= \log q_{\eps}^{S^{-1}}(x,z)-\log q_{\eps}(x,y)-c(x,z,\eps)+(\text{marginal terms}).
\end{align*}
As $S$ is constant, the constant terms are included as marginal expressions as well. Recall that $z= \nabla \varphi^*(y) = S^{-1}(y-m)$. By expanding the terms in the above kernels, $(\log q_{\eps}^{S^{-1}}(x,z)-\log q_{\eps}(x,y))$ consists of terms that depend only on one of $x$ or $y$ alone and are integrable with respect to $e^{-f}$ and $e^{-h}$, respectively. As $\bar{\ell}_{\eps}$ and $\Pi_{\eps}$ have the same marginals, with $z = \nabla \varphi^*(y)$,
\begin{align*}
    H(\bar{\ell}_{\eps}|\Pi_{\eps})+H(\Pi_{\eps}|\bar{\ell}_{\eps}) &= \left(\Exp{\bar{\ell}_{\eps}}-\Exp{\Pi_{\eps}}\right)\left[\log\left(\frac{d\bar{\ell}_{\eps}}{d\Pi_{\eps}}\right)\right]=\left(\Exp{\bar{\ell}_{\eps}}-\Exp{\Pi_{\eps}}\right)\left[-c(x,z,\eps)\right].
\end{align*}
Exactly as argued in \eqref{eq:static-ell-eps-mu-reps}, it holds that $\Exp{\bar{\ell}_{\eps}}[-c(x,z,\eps)] \leq \frac{\eps}{8}I_{g}(\mu|\vol^{\varphi})$. Recall from Proposition \ref{prop:cal-u-ip-identities} that $\Exp{\mu}[\cU_{\mu}] = -\frac{1}{8}I_g(\mu|\vol^{\varphi})$. Thus, we similarly deduce $H(\bar{\ell}_{\eps}|\Pi_{\eps})+H(\Pi_{\eps}|\bar{\ell}_{\eps}) \leq \Exp{\Pi_{\eps}}[R(x,z,\eps)]$.
Finally, apply Jensen's inequality to obtain, with $Z = \nabla \varphi^*(Y)$,
\begin{align*}
    \Exp{\Pi_{\eps}}[R(x,z,\eps)] &\leq \eps \Exp{(X,Y)\sim \Pi_{\eps}}\left[\Exp{W_{X}^{S^{-1},\eps}}\left[\int_{0}^{1} \cU_{\mu}(\omega_0)-\cU_{\mu}(\omega_s)ds\middle|\omega_{1}=Z\right]\right] \\
    &= \eps \int_0^{1} \left(\Exp{(X,Y)\sim \Pi_{\eps}}[\cU_{\mu}(X)]-\Exp{(X,Y)\sim \Pi_{\eps}}\Exp{W_{X}^{S^{-1},\eps}}\left[ \cU_{\mu}(\omega_s)\middle |\omega_{1}=Z\right] \right)ds.
\end{align*}
Let $(X_{\eps},Y_{\eps}) \sim \Pi_{\eps}$, then observe that $\omega_s \sim (1-s)X_{\eps}+s\nabla \varphi^*(Y_{\eps})+\sqrt{\eps s(1-s)}S^{-1/2}Z$ for $Z \sim N(0,\Id)$. As $X_{\eps},\nabla\varphi^*(Y_{\eps}) \sim e^{-f}$, which has subexponential tails, $(X_{\eps},\nabla \varphi^*(Y_{\eps}))$ converges in $\Was{p}$ to $(X,X)$ for $X \sim e^{-f}$ as $\eps \downarrow 0$ for any $p \geq 1$. As $Z$ is also subexponential, it follows that the law of $\omega_{s}$ converges in $\Was{p}$ to $e^{-f}$ as $\eps \downarrow 0$ for any $p \geq 1$. As $\cU_{\mu}$ has polynomial growth, the integral in the above bound vanishes as $\eps \downarrow 0$, establishing the desired convergence rate.
\end{proof}

\textbf{Integrated Fisher Information Identities.}
We can better understand this disconnect with the following simple calculations. Again, fix $e^{-f}, e^{-h} \in \cP_2(\mathbb{R}^{d})$ and let $\nabla \varphi$ denote the Brenier map from $e^{-f}$ to $e^{-h}$. We are now no longer assuming that $\nabla \varphi$ is affine. Let $(\rho_t^0, t\in [0,1])$ denote the (standard) McCann interpolation from $e^{-f}$ to $e^{-h}$, and for $t \in [0,1]$ let $\nabla \varphi_{0 \to t} = (1-t)\Id + t\nabla \varphi$ and $f_t = -\log \rho_t^0$.
By the change of variables formula,
\begin{align}\label{eq:mccann-cov}
    \nabla f_{t}(\nabla \varphi_{0 \to t}(x)) &= \nabla^{-2}\varphi_{0 \to t}(x) \left(\nabla f(x) + \sum_{ij} \partial^{2}_{ij}\varphi_{t \to 0}(x_t) \partial^{3}_{ijk}\varphi_{0 \to t}(x)\right).
\end{align}
As $\nabla^2 \varphi(x)$ is positive definite, we can write $\nabla^{2} \varphi(x) = P(x) D(x) P^{-1}(x)$ for matrix valued functions $P$ and $D$, with $D(x)$ diagonal for all $x \in \mathbb{R}^{d}$. As the spatial variable has no role in the time integration, we obtain the following identity:
\begin{align*}
    \int_{0}^{1} \Exp{} \norm{\nabla^{-2}\varphi_{0 \to t}(X)\nabla f(X)}^{2}dt &= \Exp{\mu}\left[\nabla f(X)^{T}\nabla^2 \varphi^*(X^*)\nabla f(X)\right] = \Exp{\mu}\|\nabla_g f\|_{g}^2 = I_{\varphi}(\mu|\vol^{\varphi}).
\end{align*}
From \eqref{eq:mccann-cov}, when $\nabla \varphi$ is \textbf{affine} and thus $\partial^{3}_{ijk} \varphi_{0 \to t} \equiv 0$ for all $t \in [0,1]$, observe that
\begin{align}\label{eq:affine-fi-identity}
    \int_0^{1} I(\rho_t^0|\mathrm{Leb})dt &=  \int_{0}^{1} \Exp{} \norm{\nabla^{-2}\varphi_{0 \to t}(X)\nabla f(X)}^{2}dt = I_{\varphi}(\mu|\vol^{\varphi}).
\end{align}
To illustrate how \eqref{eq:affine-fi-identity} fails when we consider Hessian manifolds with non-constant metric, we simplify to the \textbf{one dimensional setting}. 
Recall that $\varphi'''_{0 \to t}(x) = t \varphi'''(x)$ and $\varphi_{0 \to t}''(x) = 1 +t(\varphi''(x)-1)$. Integrating \eqref{eq:mccann-cov} in $t$ and taking expectation gives
\begin{align*}
    \int_0^1 I(\rho_t^0|\mathrm{Leb})dt &= \Exp{\mu}\left[\frac{(f'(X))^2}{\varphi''(X)}+\frac{f'(X)\varphi'''(X)}{(\varphi''(X))^2}+\frac{1}{3}\frac{(\varphi'''(X))^2}{(\varphi''(X))^{3}}\right]
\end{align*}
Hence, the integrated Euclidean Fisher information and the Hessian Fisher information differ in their dependence on higher derivatives of the Brenier map, i.e.\ the nontrivial curvature of the Hessian manifold. With these Fisher information calculations in tow, we conclude our investigation of affine optimal transport with the following corollary. 
\begin{corollary}\label{cor:affine-brenier-map}
    In the setting and notations of Proposition \ref{prop:affine-brenier}, it holds that
    \begin{align}\label{eq:ct-expan-affine}
        H(\Pi_{\eps}|q_{\eps}) = \frac{1}{2\eps}\Was{2}^2(e^{-f},e^{-h})+\frac{1}{2}(\Ent(e^{-f})+\Ent(e^{-h}))+\frac{\eps}{8}\int_0^1 I(\rho_t^0|\mathrm{Leb})dt + o(\eps).
    \end{align}
    Assume in addition that $(\mathbb{R}^{d},S,\mu)$ satisfies $\textrm{CD}(\kappa_{\mu},+\infty)$ for some $\kappa_{\mu} \in \mathbb{R}$. Let $\cL$ denote the generator of the primal \eqref{eq:MLD}, and let $\xi \in D(\cL) \cap \mathrm{Lip}(\mathbb{R}^{d})$. In $L^2(e^{-f})$, 
    \begin{align}\label{eq:generator-ident-affine}
        \lim\limits_{\eps \downarrow 0} \frac{1}{\eps}\left(\Exp{\Pi_{\eps}}[(\xi \circ \nabla \varphi^*)(Y)|X=x]-\xi(x)\right) = \cL \xi(x).
    \end{align}
    In particular, it holds that
    \begin{align}\label{eq:score-function-affine}
        \lim\limits_{\eps \downarrow 0} \frac{1}{\eps}\left(\Exp{\Pi_{\eps}}[\nabla \varphi^*(Y)|X=x]-x\right) = -\frac{1}{2}\nabla f(x)
    \end{align}
\end{corollary}
Observe that \eqref{eq:ct-expan-affine} matches the expansion provided by \cite[Theorem 1.6]{conforti21deriv}.

\begin{proof}
As argued in the proof of Theorem \ref{thm:mld-quad-sb-comparison}, it holds that
\begin{align*}
    H(\Pi_{\eps}|q_{\eps}) &= H(\bar{\ell}_{\eps}|q_{\eps})-H(\bar{\ell}_{\eps}|\Pi_{\eps}).
\end{align*}
First, recall the rate for $H(\bar{\ell}_{\eps}|\Pi_{\eps})$ in Proposition \ref{prop:affine-brenier}.
The expansion in \eqref{eq:ct-expan-affine} then follows from plugging in \eqref{eq:bar-ell-leb-dens}, recalling that $\Exp{\bar{\ell}^{\eps}}[-c(x,\nabla\varphi^*(z),\eps)] \leq \frac{\eps}{8}I_{\varphi}(\mu|\vol^{\varphi})$, and then applying the identity \eqref{eq:affine-fi-identity}. As for \eqref{eq:generator-ident-affine}, this follows from Theorem \ref{thm:generator-transfer}. In the notation of Theorem \ref{thm:generator-transfer}, set $\alpha^{\eps} = (\Id,\nabla \varphi^*)_{\#}\Pi_{\eps}$ and $\beta^{\eps} = \text{Law}(X_0,X_{\eps})$. As $(\Id,\nabla \varphi^*)$ is invertible, $H(\alpha^{\eps}|\beta^{\eps}) = H(\Pi_{\eps}|\bar{\ell}_{\eps})$ and thus \eqref{eq:generator-ident-affine} holds for all Lipshitz $\xi \in \mathrm{Lip}(\mathbb{R}^{d}) \cap D(\cL)$. Picking $\xi$ to be the coordinate functions $\xi(x) = x_i$ establishes \eqref{eq:score-function-affine}.
\end{proof}

\bibliographystyle{alpha}
\bibliography{sample}

\section{Appendix}\label{sec:appendix}
\subsection{Deferred Proofs}
\begin{proof}[Proof of Proposition \ref{prop:heat-kernel-asymp}]
    The existence of $c_0(\cdot,\cdot), R(\eps,\cdot,\cdot)$ such that \eqref{eq:heat-kernel-expan} with the stipulated properties is established in \cite[Theorem 1.4]{neel2025uniformlocalizedasymptoticssubriemannian}. The identity in \eqref{eq:grad-log-heat-ker-exp} follows from the fact that $\nabla_x \frac{1}{2}d^2(x,z) = - \log_x z$ when $(x,z) \notin \cC$. The identity for $c_{0}(x,z)$ is historically attributed to \cite{molchanov-exp75}, and it is also called the Pauli-Van Vleck-Morette determinant. The stated limits at coincidence (i.e.\ limits as $z \to x$) of $-\frac{1}{2}\log c_0$ and its gradient in $x$ are well-known in the physics literature. For example, they are given in \cite[Equation (A.20)]{avramidi16} (note one must apply the identity in \cite[Equation (A.1)]{avramidi16} so that the gradient is computed with respect to $x$ and not $z$). 
    For more comprehensive information on the Van Vleck-Morette determinant, we refer readers to the survey \cite[Section 7]{poisson11-vanvleckdet} and text \cite[Section 3.6.3]{avramidi15-book}.
\end{proof}

\begin{proof}[Proof of Proposition \ref{prop:ex-uniq-nonexp-mld}]
Fix a global chart as stipulated in (H2), and let $\cL$ denote the generator of this strictly elliptic diffusion on $\mathbb{R}^{d}$. There is existence and uniqueness of a weak solution (given an initial distribution) to \eqref{eq:hess-mani-sde-mld} up to an explosion time $e(X): \Omega \to [0,+\infty]$. 

To see that the explosion time is a.s.\ infinite for any initial data, we simply repeat the proof of \cite[Theorem 2.2.19]{royer-lsi}. Fix $x_0 \in \mathbb{R}^{d}$ and consider \eqref{eq:hess-mani-sde-mld} with initial condition $X_0 = x_0$. Fix $R > 0$ and define the stopping time $T_{R} = \inf\{t: U(X_t) > R\}$. As $T_{R} < e(X)$, by Dynkin it holds that
\begin{align}\label{eq:exp-bdd-u-cu}
    \Exp{}[U(X_{t \wedge T_{R}})-U(x_0)] = -\Exp{}\left[\int_{0}^{t \wedge T_{R}} \cL U(X_s)ds\right].  
\end{align}
As $\cL U = -\frac{1}{2}\norm{\nabla_g U}_g^2 + \frac{1}{2}\Delta_g U$, $\cL U$ has a global upper bound by \eqref{assumption:potent}. The proof of \cite[Theorem 2.2.19]{royer-lsi} now follows without any changes to establish nonexplosion. 
\end{proof}

\begin{proof}[Proof of Proposition \ref{prop:rn-deriv-diffusion}]
Under Assumption \ref{assumption:manifold} (H1), \eqref{eq:rn-path-measure} is computed in \cite[Proposition 2.11]{lavenant-traj-inf}. Now, let $(M,g)$ satisfy (H2) and fix a global coordinate system $(x^1,\dots,x^d)$. We will now identify $M$ with this global chart $\mathbb{R}^{d}$ and simply perform all computations over Euclidean space. 
We have remarked already that under (H2) the manifold Brownian motion, i.e.\ the solution to \eqref{eq:local-mani-bm} (which now holds globally), exists, is unique, and has a.s.\ infinite explosion time. In this global chart, recall that \eqref{eq:mani-diff-sde} now writes as \eqref{eq:manifold-diff-global-chart}.

Fix some $x \in \mathbb{R}^{d}$. First, we assume that each $\partial_{x_i}U$ is bounded on $\mathbb{R}^{d}$. Since $\partial_{x_i} U$ is bounded, it follows from \cite[Chapter IX, Theorem 1.11]{revuz2004continuous} that there exists a unique solution to \eqref{eq:manifold-diff-global-chart} (in their notation, set $f = g^{-1/2}$, $g = -\frac{1}{2}g^{ij}\Gamma_{ij}^{k}$, and $h = -\frac{1}{2}\partial_{x_i}U$). 
Moreover, the Radon-Nikodym derivative has the following expression. In the language of \cite[Chapter IX, Theorem 1.10]{revuz2004continuous}, set $a(x) = g^{-1}(x)$, $b(x) = -\frac{1}{2}g^{ij}(x)\Gamma_{ij}^{k}(x)$, and $c(x) = -\frac{1}{2}\partial_{x_i}U(x)$.
Altogether then, from \cite[Chapter IX, Theorem 1.11]{revuz2004continuous} it holds that 
\begin{align}\label{eq:ry-ix-1-10}
    \frac{dP_x}{dW_x}(\omega) &= \exp\left(\int_0^{\eps} c(\omega_t)^{T}d\omega_t - \int_0^{\eps} c(\omega_t)^{T} b(\omega_t)dt-\frac{1}{2}\int_0^{\eps} c(\omega_t)^{T}a(\omega_t)c(\omega_t)dt\right).
\end{align}
To remove the stochastic integral appearing in \eqref{eq:ry-ix-1-10}, apply $\Ito$'s formula to obtain
\begin{align}\label{eq:girs-global}
    \frac{dP_x}{dW_x} &= \sqrt{\frac{e^{-U(\omega_{\eps})}}{e^{-U(x)}}}\exp\left(\int_0^{\eps} \frac{1}{4}\Tr\left(g^{ij}(\omega_t)\frac{\partial^2}{\partial x^j \partial x^k}U(\omega_t)\right)-c(\omega_t)^{T} b(\omega_t)-\frac{1}{2}c(\omega_t)^{T}a(\omega_t)c(\omega_t)dt\right).
\end{align}
Let's analyze the terms in the right hand side of \eqref{eq:girs-global}. First, observe that $-\frac{1}{2}c(x)^{T}a(x)c(x) = -\frac{1}{8}\norm{\nabla_g U}_{g}^2(x)$.
Similarly, the next two terms write as
\begin{align*}
    \frac{1}{4}\Tr\left(g^{ij}(x)\frac{\partial^2}{\partial x^j \partial x^k}U(x)\right) &= \frac{1}{4}g^{ij}(x)\frac{\partial^2}{\partial x^j \partial x^i}U(x), \text{ }-c(x)^{T}b(x) = \frac{1}{4} \frac{\partial}{\partial x^{k}}U(x)g^{ij}(x)\Gamma_{ij}^{k}(x). 
\end{align*}
From \eqref{eq:laplace-beltrami-coords},
the right hand side of \eqref{eq:girs-global} has the stated form when $\partial_{x_i}U$ is globally bounded.

Now, consider a general $U$ such that \eqref{eq:manifold-diff-global-chart} does not explode. To generalize the above formula, we parrot the argument from \cite[Lemma 2.2.21]{royer-lsi}. Let $(\chi_n, n \geq 1) \subset C_c^{\infty}(\mathbb{R}^d)$ be sequence of smooth cutoff functions on the balls of radius $n$ centered at the origin. Set $U_n := \chi_n U$, and let $P^n_x$ denote the law on $C([0,T],\mathbb{R}^{d})$ of the solution to \eqref{eq:mani-diff-sde} with drift equal to $-\frac{1}{2}\nabla_g U_n$ started from $x$. The previous paragraph gives an explicit computation for each $dP^n_x/dW_x$. Following the exact line of argumentation in the last paragraph of \cite[Lemma 2.2.21]{royer-lsi} establishes \eqref{eq:rn-path-measure} for this general $U$. 
\end{proof}

We now justify the Entropic Benamou-Brenier formulation of entropic cost in our setting. 
\begin{proof}[Proof of Proposition \ref{prop:entropic-bb}]
Under (H1) with $\mathrm{CD}(K,N)$ for $N < +\infty$, this is exactly the conclusion of \cite[Theorem 4.1]{gigli-bb-ent-rcd}. The class of absolutely continuous curves considered is $(\mu_t,v_t, t \in [0,1])$ for which there is some $C > 0$ such that $\mu_t \leq C\diffref$ for all $t \in [0,1]$ and $t \mapsto \int_{M} \norm{v_t}_g^2 d\mu_t$ is Borel and in $L^1(0,1)$.
When $N = +\infty$ and $\diffref(M) = 1$, \cite[Remark 2.5]{chiarini2022gradient} demonstrates that the exact same proof holds. In our setting, recall that Assumption \ref{assumption:diffusions} forces that we consider fully supported measures on $M$. Hence, the results from \cite{gigli-bb-ent-rcd} only apply when $M$ is compact as the setting in \cite{gigli-bb-ent-rcd} requires that $\mu$ and $\nu$ have both bounded support and densities with respect to $\diffref$.   

In the case that $M$ is not compact but rather given by a global chart of $\mathbb{R}^{d}$ satisfying Assumption \ref{assumption:manifold} (H2), by mirroring the arguments provided in \cite[Section 5]{gentil2017analogy} (particularly Theorem 5.1, Corollary 5.8 therein) we arrive at the same conclusion. The minimum is taken over all absolutely continuous curves on $\cP_2(M)$ over the time interval $[0,1]$ from $\mu$ to $\nu$. We now identify $M$ with $\mathbb{R}^{d}$. Recall the definition of the dynamic $\Schro$ bridge $P^{\eps} \in \cP(C([0,1],\mathbb{R}^{d}))$ given in \eqref{eq:dynam-schro-bridge-defn}. By assumption, $P^{\eps}$ exists and $H(P^{\eps}|R^{\eps})$ is finite. Additionally, $R^{\eps}$ is a Markov measure. By \cite[Proposition 2.10]{schroLeonard13}, we can restrict the feasible set in \eqref{eq:dynam-schro-bridge-defn} to
\begin{align}\label{eq:dynam-schro-bridge-defn-finite-entropy}
   \left\{P \in \cP(C([0,1],\mathbb{R}^{d})): (\omega_0,\omega_1)_{\#}P\in \Pi(\mu,\nu),P \text{ is Markov}, H(P|R^{\eps}) \text{ is finite}\right\}.
\end{align}
Fix some $P \in \cP(C([0,1],\mathbb{R}^{d}))$ in the above constraint set. As $H(P|R^{\eps})$ is finite, applying Girsanov's Theorem as developed in \cite[Theorem 2.1, Theorem 2.3]{leonard2014some} gives the existence of a drift $(Z_t, t \in [0,1])$ such $P$ is the law of the following SDE (written as a manifold SDE):
\begin{align}
    dX_t &= Z_t(X_t) dt + dB_t^{M,\eps}, X_0 \sim \mu. 
\end{align}
Here, $(B_t^{M,\eps},t \in [0,1])$ is the $\frac{\eps}{2}\Delta_g$-diffusion process. Let $(\mu_t, t \in [0,1]) \subset \cP(M)$ denote the marginal flow of $P$. Set  $\rho_t = \frac{d\mu_t}{d\vol}$ and $v_t = Z_t - \frac{\eps}{2}\nabla_g \log \rho_t$. The key fact is to recognize that $(\mu_t,v_t, t \in [0,1])$ satisfies the continuity equation on the manifold. By performing these calculations on the global chart (thus reducing everything to Euclidean calculations), the rest of the argument of \cite[Theorem 5.1, Corollary 5.8]{gentil2017analogy} proceeds identically. 
\end{proof}

\begin{proposition}[LDP rate function transfer]\label{prop:ldp-rate-diff}
    Let $(M,g)$ be a complete Riemannian manifold with lower bounded Ricci curvature. Fix $U \in C^2(M)$ and and $x \in M$. Let $(R_x^{\eps}, \eps > 0)$ denote the laws on $C([0,1],M)$ corresponding to diffusions with generators $\cL^\eps = \eps\left(-\frac{1}{2}\nabla_g U \cdot \nabla_g + \frac{1}{2}\Delta_{g}\right)$ started from $x$. If $\inf\limits_{M} \cU > -\infty$ and $\inf\limits_M U > -\infty$, then $(R_x^{\eps},\eps > 0)$ satisfies a Large Deviation Principle (LDP) on $C([0,1],M)$ with good rate function
    \begin{align}\label{eq:ldp-geod}
        I(\omega) &= \begin{cases}
            \frac{1}{2} \int_{0}^{1} \norm{\dot{\omega}_t}_{g}^2 dt & \text{ if $\omega_0 = x$ and $\omega \in H^1(M)$} \\
            +\infty & \text{otherwise}
        \end{cases}
    \end{align}
    Now, assume in addition that $M$ is compact. Let $\mu = \exp(-U)d\vol$ denote the stationary measure of the diffusion, and suppose that $\mu \in \cP(M)$. Let $(R^{\eps}, \eps > 0)$ denote the laws on $C([0,1],M)$ corresponding to the temperature $\eps$ diffusion with initial distribution $\mu$. Then $(R^{\eps}, \eps > 0)$ satisfies an LDP with good rate function
    \begin{align}\label{eq:ldp-stationary}
        I(\omega) &= \begin{cases}
            \frac{1}{2} \int_{0}^{1} \norm{\dot{\omega}_t}_{g}^2 dt & \text{ if $\omega \in H^1(M)$} \\
            +\infty & \text{otherwise}
        \end{cases}
    \end{align}
    Hence, setting $\ell_{\eps} := (\omega_0,\omega_1)_{\#}R^{\eps}$ it holds that $(\ell_{\eps}, \eps > 0)$ satisfies an LDP  on $\cP(M \times M)$ with good rate function $\frac{1}{2}d^2(\cdot,\cdot)$. 
\end{proposition}

\begin{proof}
    Fix $x \in M$. Let $W_x^{\eps} \in \cP(C([0,1],M))$ denote Wiener measure started from $x$ with temperature $\eps$, i.e.\ the law of the process corresponding to $\frac{\eps}{2}\Delta_g$. This is the case $U = 0$. It is known that the LDP for $(W_x^{\eps}, \eps > 0)$ holds with rate function \eqref{eq:ldp-geod}, provided $M$ is complete and has lower bounded Ricci curvature \cite[Theorem 5.1]{krajj-ldp-mani-19}. 

    We now transfer this result to the $(R^{\eps}_x, \eps > 0)$. To start, we show the collection is exponentially tight.
    % by establishing that the large deviations upper bound with $I$ holds. 
    Let $K \in \mathbb{R}$ be such that $K < \cU(x), U(x)$ for all $x \in M$. From Proposition \ref{prop:rn-deriv-diffusion}, 
    \begin{align*}
        \frac{dR_{x}^{\eps}}{dW_x^{\eps}}(\omega) &= \exp\left(-\frac{1}{2}U(\omega_{\eps})+\frac{1}{2}U(x)-\int_0^{\eps} \cU(\omega_t)dt\right) \leq \exp\left(-\frac{1}{2}K + \frac{1}{2}U(x)-K\eps\right).
    \end{align*}
    That is, there is some $C > 0$ such that $dR_x^{\eps}/dW_x^{\eps} \leq C$ for all small enough $\eps$. Let $\Gamma \subset C([0,1],M)$ be Borel, then it holds that
    \begin{align}\label{eq:ldp-upper-bound}
        \limsup\limits_{\eps \downarrow 0} \eps \log R_x^{\eps}[\Gamma] &\leq \limsup\limits_{\eps \downarrow 0} \eps \left(\log C + \log W_x^{\eps}[\Gamma]\right) \leq -\inf\limits_{\omega \in \overline{\Gamma}} I(\omega),
    \end{align}
    where the last inequality follows from the LDP for $(W_x^{\eps},\eps > 0)$. As $I$ is a good rate function \cite[Section 1.2]{dz-large-dev} and $C([0,1],M)$ is Polish, it follows from \cite[Exercise 4.1.10(c)]{dz-large-dev} that $(R_x^{\eps}, \eps > 0)$ is exponentially tight.

    Now, we just show that $(R_x^{\eps}, \eps > 0)$ satisfies a weak LDP with $I$ \cite[Section 1.2]{dz-large-dev}. Fix $\omega \in C([0,1],M)$ such that $I(\omega) < +\infty$. Let $\Gamma \subset C([0,1],M)$ be measurable such that $\omega \in \mathrm{int}(\Gamma)$. Define \[B = \left\{\eta \in C([0,1],M): \sup\limits_{t \in [0,1]} d(\eta_t,\omega_t) < 1\right\},\] then there is some $c > 0$ such that $\left.\frac{dR_x^\eps}{dW_x^\eps}\right|_{B} \geq c$ as all paths in $B$ stay within a compact subset of $M$. As $(W_{x}^{\eps}, \eps > 0)$ satisfies the weak LDP with $I$, 
    \begin{align*}
        \liminf\limits_{\eps \downarrow 0} \eps \log \Exp{R_{x}^{\eps}}[\Gamma] &\geq \liminf\limits_{\eps \downarrow 0} \eps \log \Exp{R_{x}^{\eps}}[\Gamma \cap B] \geq \liminf\limits_{\eps \downarrow 0} \eps \left(\log c + \log \Exp{W_{x}^{\eps}}[\Gamma \cap B]\right) \geq -I(\omega). 
    \end{align*}
    Lastly, \eqref{eq:ldp-upper-bound} implies the upper bound condition for weak LDP. 
    By \cite[Lemma 1.2.18]{dz-large-dev} the full LDP holds, establishing \eqref{eq:ldp-geod}.

    Now, assume in addition that $M$ is compact. A similar argument establishes \eqref{eq:ldp-stationary}. First, let $W_{\mu}^{\eps} \in \cP(C([0,1],M))$ denote temperature $\eps$ Wiener measure on $M$ with initial distribution $\mu$, then by \cite[Theorem 5.1]{krajj-ldp-mani-19} the rate function governing the LDP of $(W_{\mu}^{\eps}, \eps > 0)$ is \eqref{eq:ldp-stationary}. By Proposition \ref{prop:rn-deriv-diffusion}, there are constants $c,C > 0$ such that $cW_{\mu}^{\eps} \leq R^{\eps} \leq CW_{\mu}^{\eps}$ as $U,\cU$ are continuous and $M$ is compact. Thus, for any Borel $\Gamma \subset C([0,1],M)$ and $\eps > 0$ it holds that
    \begin{align*}
       \eps \log c + \eps \log W_{\mu}^{\eps}[\Gamma] \leq \eps \log R^{\eps}[\Gamma] \leq \eps \log C + \eps \log W_{\mu}^{\eps}[\Gamma], 
    \end{align*}
    from which the LDP for $(R^{\eps}, \eps > 0)$ follows from that of the $(W_\mu^{\eps}, \eps > 0)$. Finally then, the last claim regarding the LDP of the $(\ell_{\eps}, \eps > 0)$ follows from the contraction principle \cite[Theorem 4.2.1]{dz-large-dev} and the definition of geodesic distance. 
\end{proof}

\begin{proposition}[Diffusion Moment Bound]\label{prop:diff-fourth-moment}
    On $(M,g)$ consider $\mu = \exp(-U_{\mu})\vol \in \cP_2(M)$ in the setting dictated by Assumption \ref{assumption:manifold} (H2). Consider the manifold SDE \eqref{eq:intro-target-process} with drift $-\frac{1}{2}\nabla_g U_{\mu}$ and initial distribution $\mu$. Let $\ell_{\eps} = \mathrm{Law}(Y_0,Y_{\eps})$. If $\Exp{\mu}\left[\norm{\nabla_g U_{\mu}}^{p}_{g}\right] < +\infty$ for some $p \geq 2$, then $\limsup\limits_{\eps \downarrow 0}\eps^{-p/2}\Exp{\ell_{\eps}}[d^{p}(x,z)] < +\infty$. 
\end{proposition}

\begin{proof}
    Fix a global chart under which Assumption \ref{assumption:manifold} (H2) holds and identify $M$ with $\mathbb{R}^{d}$. Then the manifold distance, denoted $d$, is equivalent to the Euclidean distance, which we denote $\norm{\cdot}$. Recall that in this global chart, the SDE \eqref{eq:intro-target-process} can be written as \eqref{eq:manifold-diff-global-chart}. For convenience, set $\beta(x) = -\frac{1}{2}g^{ik}(x) \frac{\partial}{\partial x^k}U(x) - \frac{1}{2}g^{ij}(x)\Gamma_{ij}^{k}(x)$, and observe that
    \begin{align}\label{eq:lazy-bdd-diff-lp}
        d^{p}(X_0,X_{\eps}) &\leq C_1 \norm{X_{\eps}-X_{0}}^{p} \leq C_2\left(\norm{\int_0^{\eps}\beta(X_t)dt}^{p}+\norm{\int_0^{\eps} g^{-1/2}(X_t)dB_t}^{p}\right)
    \end{align}
    Let $M_t = \int_0^{t} g^{-1/2}(X_t)dB_t$, which is a martingale. By the BDG inequality \cite[Chapter IV, Theorem 4.1]{revuz2004continuous} and bounds on $g$, there are constants $C_3,C_4 > 0$ such that $\Exp{}\norm{M_\eps}^{p} \leq C_3 \Exp{}\left[\langle M_{\eps}\rangle^{p/2}\right] \leq C_4 \eps^{p/2}$. For the deterministic term, it follows from Jensen's inequality, the bounds on $g^{-1}$ and the $\Gamma_{ij}^{k}$, and stationarity of $(X_t,t \geq 0)$ that for some constant $C_5 > 0$
    \begin{align*}
        \Exp{}\norm{\int_0^{\eps} \beta(X_t)dt}^{p} \leq C_5\eps^{p-1}\int_0^{\eps}\left(\Exp{}\left[\norm{\nabla_g U_{\mu}(X_t)}_g^{p}\right]+1\right)dt = C_5(\Exp{\mu}\left[\norm{\nabla_g U_{\mu}}_g^{p}\right]+1)\eps^p.
    \end{align*}
    Applying expectation to \eqref{eq:lazy-bdd-diff-lp} with the above bounds establishes the desired decay rate. 
\end{proof}

\subsection{Some Riemannian Calculations}
We present some useful geometric calculations. 

\begin{proposition}\label{prop:moment-bdd-brownian-bridge}
    Let $(M,g)$ satisfy Assumption \ref{assumption:manifold} (H2). Let $(W_{x,y}^{\eps}, x,y \in M, \eps > 0)$ denote the laws on $\cP(C([0,\eps],M))$ of the Brownian bridges on $M$, then there is an $\eps_0 < 1$ such that for each integer $N \geq 1$ there exists $C > 0$ such that for all $x,y \in M$ and $0 < t < \eps < \eps_0 (< 1)$
    \begin{align}
        \Exp{W_{x,y}^{\eps}}[d^{2N}(X_t,x)] \leq C\left(d^{2N}(x,y)+\eps\right).
    \end{align}
    This same inequality holds for bridges over $[0,1]$ of $\eps$-Brownian motion with $t \in [0,1]$.
\end{proposition}

\begin{proof}
    We fix a global chart as stipulated in the (H2) and identify $M$ with $\mathbb{R}^{d}$. The following argument mirrors \cite[Proposition 5.5.4]{hsu-stoch-analysis-manifold}, which establishes this bound in the case of $M$ compact on the basis of the SDE for the Brownian bridge \cite[Theorem 5.5.4]{hsu-stoch-analysis-manifold} and the following estimate of the gradient of the log of the fundamental solution to $\partial_t - \frac{1}{2}\Delta_g = 0$ for some $C > 0$ and all $\eps < 1$
\begin{align}\label{eq:grad-log-heat-kernel}
    \norm{\nabla_g \log p_{\eps}(x,y)}_{g} \leq C\left(\frac{1}{\eps}d(x,y)+\frac{1}{\sqrt{\eps}}\right).
\end{align}
On $(M,g)$ satisfying (H2), the same SDE for the Brownian bridge is established in \cite{engoulatov-grad-log-heat-06}. Similarly, as the volumes of all balls on $(M,g)$ are comparable to Euclidean balls with universal constants, \cite[Theorem 1]{engoulatov-grad-log-heat-06} establishes \eqref{eq:grad-log-heat-kernel} with a constant $C$ holding for all $x,y \in M$ and $\eps$ small enough, say less than some $\eps_0 \in (0,1)$. 

Fix some $x,y \in M (= \mathbb{R}^{d})$. The following is the main deviation from the argument in \cite[Proposition 5.5.4]{hsu-stoch-analysis-manifold}. Define the function $\eta: M \to [0,+\infty)$ by $\eta(z) = \norm{z-y}^{2N}$, where $\norm{\cdot}$ denotes the Euclidean norm on $\mathbb{R}^{d}$. Observe that $\eta \in C^{\infty}(M)$ and that, with $\alpha,\beta > 0$ as in assumption, $\alpha^{N} \eta(z) \leq d^{2N}(z,y) \leq \beta^{N} \eta(z)$. There then exists numerical constants $K_1, K_2 > 0$ such that for all $i,j \in\{1,\dots,d\}$
\begin{align}\label{eq:poly-bds}
    \abs{\frac{\partial}{\partial z^i} \eta(z)} \leq K_1 \eta(z)^{(2N-1)/2N} \text{ and } \abs{\frac{\partial^2}{\partial z^i \partial z^{j}}\eta(z)} \leq K_2 \eta(z)^{(2N-2)/2N}
\end{align}
Thus, it follows that
\begin{align}\label{eq:grad-est-euc-rad}
    \nabla_g \eta(z) &= g^{ij}(z)\frac{\partial}{\partial z^j} \eta(z) \implies \norm{\nabla_g \eta(z)}_{g} \leq \alpha^{-1}K_{1}\eta(z)^{(2N-1)/2N}. 
\end{align}
Let $(X_{t}, t \in [0,\eps])$ be the coordinate process. Applying $\Ito$'s formula to $\eta$ on the Brownian bridge $W_{x,y}^{\eps}$ as in \cite[Proposition 5.5.4]{hsu-stoch-analysis-manifold} gives that
\begin{align}
    t \in [0,\eps) \mapsto \eta(X_t) - \eta(x) -\frac{1}{2}\int_0^{t}\Delta_g \eta(X_s)ds - \int_0^{t}\langle \nabla_g \eta(X_s),\nabla_g \log p_{\eps-s}(X_s,y)\rangle_g ds
\end{align}
is a continuous local martinagle. Define a sequence of stopping times $S_R := \inf\{s > 0: \eta(X_s) > R \}$, then for each $R \geq 1$ it holds that
\begin{align}\label{eq:stopped-mgale-equality}
    \Exp{}[\eta(X_{t \wedge S_R})] = \eta(x)+ \frac{1}{2}\int_0^{t \wedge S_R} \Exp{}[\Delta_g \eta(X_s)]ds + \int_0^{t \wedge S_R} \Exp{}[\langle \nabla_g \eta(X_s),\nabla_g \log p_{\eps-s}(X_s,y)\rangle_g] ds. 
\end{align}
Recall the coordinate expression for the Laplace-Beltrami operator in \eqref{eq:laplace-beltrami-coords}. As $g$ and its first derivatives are bounded, by \eqref{eq:poly-bds} there is a constant $K_3 > 0$ such that
\begin{align*}
    \abs{\Delta_g f(z)} &= \abs{2g^{ij}(z)\frac{\partial^2}{\partial z^i \partial z^{j}}\eta(z)-g^{ij}(z)\Gamma_{ij}^{k}(z)\frac{\partial}{\partial z^{k}}\eta(z)} \\
    &\leq K_{3}\left(1+\eta(z)^{(2N-2)/2N}+\eta(z)^{(2N-1)/2N}\right) \leq K_{3}(3+2\eta(z)).
\end{align*}
By applying Cauchy-Schwarz with \eqref{eq:grad-log-heat-kernel} and \eqref{eq:grad-est-euc-rad}, there is some $K_4$ such that whenever $\eps < \eps_0$, it holds for all $z,y \in M$ and $0 < s \leq \eps/2$ that
\begin{align*}
    \abs{\langle \nabla_g \eta(z),\nabla_g \log p_{\eps-s}(z,y)\rangle_g} &\leq \norm{\nabla_g \eta(z)}_g\norm{\nabla_g \log p_{\eps-s}(z,y)}_g \\
    &\leq \alpha^{-1}K_1 \eta(z)^{(2N-1)/2N} \cdot C\left(\frac{\beta}{\eps-s}\eta(z)^{1/2N}+\frac{1}{\sqrt{\eps-s}}\right)\\
    &\leq K_4 \left(\frac{1}{\eps}\eta(z)+1\right).
\end{align*}
Now, we insist that $t \leq \eps/2$. These two previous bounds give from \eqref{eq:stopped-mgale-equality} that
\begin{align*}
    \Exp{}[\eta(X_{t \wedge S_R})] &\leq \eta(x) + \frac{1}{2}\int_0^{t\wedge S_R} K_3(3+2\Exp{}[\eta(X_s)])ds + \int_0^{t \wedge S_R} K_4\left(\frac{1}{\eps}\Exp{}[\eta(X_s)]+1\right) ds \\
    % &\leq \eta(x)+K_{3}(t \wedge S_N) + \frac{K_{4}}{\eps} \int_{0}^{t \wedge S_N} \Exp{}[\eta(X_s)] ds \\
    &\leq \eta(x)+K_{5}t + \frac{K_6}{\eps} \int_0^{t} \Exp{}[\eta(X_s)]ds,
\end{align*}
for positive constants $K_5, K_6 > 0$. Observe that this upper bound holds for all $R \geq 1$. As $S_{R} \to +\infty$ a.s.\ as $R \to +\infty$ from path continuity, it holds from Fatou's lemma that
\begin{align*}
    \Exp{}[\eta(X_{t})] \leq \eta(x) + K_{5}t + \frac{K_6}{\eps}\int_0^{t} \Exp{}[\eta(X_s)]ds.
\end{align*}
By Gr\"{o}nwall's Lemma, there exists $K_{6}$, which is independent of $\eps$, such that $\Exp{}[\eta(X_{t})] \leq K_{6}(\eta(x)+t)$. By equivalence of $d^2$ and $\eta$, for all $t \in (0,\eps/2)$ there is $K_{7} > 0$ such that
\begin{align*}
    \Exp{W_{x,y}^{\eps}}[d^{2N}(X_t,y)] \leq K_{7}(d^{2N}(x,y)+t).
\end{align*}
In the case that $t > \eps/2$, repeat the above argument with $W_{y,x}^{\eps}$, using that this is the measure obtained by reversing in time the paths of $W_{x,y}^{\eps}$. This establishes the desired inequality. 
\end{proof}

\begin{proposition}[Sufficient Condition for \eqref{eq:cu-integrability} under Assumption \ref{assumption:manifold} (H2)]\label{prop:h2-int-argument}
    Let $(M,g)$ satisfy Assumption \ref{assumption:manifold} (H2), and let $\diffref = \vol$ (so that $U_{\diffref}, \cU_{\diffref} = 0$) and $\mu \in \cP_2(M)$ be such that Assumption \ref{assumption:diffusions} is satisfied. Identify $M$ with $\mathbb{R}^{d}$. Suppose that $\mu$ has subexponential tails as a measure on $\mathbb{R}^{d}$ \cite[Definition 2.7.5]{vershynin-hdp} and there exists $C > 0$ and $N \geq 1$ such that
    \begin{align}\label{eq:polygrowth-cu}
        \norm{\nabla_g \cU_{\mu}(x)}_g^{2} \leq C(1+\norm{x}^{N}),
    \end{align}
    where $\norm{\cdot}$ denotes the Euclidean norm on $\mathbb{R}^{d}$. Then \eqref{eq:cu-integrability} and \eqref{eq:o-eps-2-limit} hold.
\end{proposition}

\begin{proof}
    As $\mu$ is subexponential, it has finite moments of all order. Moments can be computed with respect to either $\norm{\cdot}$ and $d$ as the two are equivalent under (H2). Let $\pi^{\eps}$ be the $\eps$-static $\Schro$ bridge from $\mu$ to itself computed with Wiener reference measure. 

    Let $\eps > 0$ be small enough that Proposition \ref{prop:moment-bdd-brownian-bridge} holds and fix an integer $M \geq 1$. Let $x_0 \in M$, then for all $t \in (0,1)$
    \begin{align*}
        \int_{M} d^{2M}(u,x_0) \mu_t^{\eps}(du) &= \int_{M \times M} \left(\Exp{W^{\eps}_{x,z}}d^{2M}(\omega_t,x_0)\right)\pi^{\eps}(dxdz) \\
        &\leq 2^{2M}\left(\int_{M \times M} \Exp{W^{\eps}_{x,z}}[d^{2M}(\omega_t,x)] \pi^{\eps}(dxdz)\right)+2^{2M}\int_{M} d^{2M}(x,x_0) \mu(dx). 
    \end{align*}
    By Proposition \ref{prop:moment-bdd-brownian-bridge}, there is $C_1 > 0$ such that
    \begin{align*}
        \int_{M \times M} \Exp{W^{\eps}_{x,z}}[d^{2M}(\omega_t,x)] \pi^{\eps}(dxdz) &\leq \int_{M \times M} C_1(d^{2M}(x,z)+\eps) \pi^{\eps}(dxdz) \\
        &\leq 2^{2M+1}C_1 \int_{M} d^{2M}(x,x_0) \mu(dx) + C_1\eps.  
    \end{align*}
    Altogether then, there is a constant $C_2 > 0$ such that, for all $\eps >0$ small enough,
    \begin{align}\label{eq:uniform-moment-bdd-entropic-int}
        \sup\limits_{t \in [0,1]} \int_{M}  d^{2M}(u,x_0)\mu_t^{\eps}(du) \leq C_2. 
    \end{align}
    Hence, for small enough $\eps$ each measure along the entropic interpolation has finite moments of all order. By \eqref{eq:uniform-moment-bdd-entropic-int}, there is a constant $C_3$ such that for all $\eps > 0$ small enough, $\norm{\nabla_g \cU_{\mu}}_{L^2(\mu_t^{\eps})}^2 \leq C_3$ for all $t \in [0,1]$. Hence, \eqref{eq:cu-integrability} holds and $\limsup\limits_{\eps \downarrow 0} \int_0^{1} \norm{\nabla_g \cU_{\mu}}_{L^2(\mu_t^{\eps})}^2 dt < +\infty$. As $\diffref = \vol$, Proposition \ref{prop:the-one-about-fi} applies and \eqref{eq:o-eps-2-limit} holds.  
\end{proof}

\begin{proposition}\label{prop:hess-dist-symm-breg}
    Consider the Hessian manifold satisfying Assumption \ref{assumption:hessian-manifold}, and identify it with $\mathbb{R}^{d}$ via its primal coordinates. Let $D_{\varphi}[\cdot|\cdot]$ denote the Bregman divergence with respect to $\varphi$. There exists $K > 0$ depending on the global bound of the third and fourth derivatives of $\varphi$ such that for all $x,z \in \mathbb{R}^{d}$, with $\norm{\cdot}$ denoting the Euclidean norm on $\mathbb{R}^{d}$,
    \begin{align*}
        \abs{d^{2}(x,z) - \left(D_{\varphi}[z|x]+D_{\varphi}[x|z]\right)} \leq K\norm{z-x}^4.
    \end{align*} 
\end{proposition}
\begin{proof}
This reduces to a Taylor series comparison between the two quantities at hand. Recall the identity \eqref{eq:symm-bregman-div-varphi}. 
Using the Taylor expansion for $\nabla \varphi(\cdot)$ about $x$, it holds that
\begin{align}\label{eq:symm-breg-expan}
    D_{\varphi}[z|x]+D_{\varphi}[x|z] &= \nabla^2 \varphi(x) (z-x)^2 + \frac{1}{2}\nabla^{3}\varphi(x)(z-x)^{3} +O(\|z-x\|^4). 
\end{align}
For fixed $x$, recall that $\vol$-a.e.\ $z \notin C_x$. For such $z$, it holds that $d^2(z,x) = \norm{\log_x(z)}^2_{g(x)}$. We now work in primal coordinates. In this case, $g(x) = \nabla^2 \varphi(x)$ and by \cite[Section 3]{brewin-expansions-09} there exists a constant $K > 0$ depending on the third and fourth derivatives of $\varphi$ such that
\begin{align}
    \norm{\log_x(z) - \left((z-x) + \frac{1}{2}\Gamma_{ij}^{k}(x)(z-x)^i(z-x)^j\right) } \leq  K\norm{z-x}^{3}. 
\end{align}
As $\varphi$ has globally bounded third and fourth derivatives, this constant $K$ can be chosen independently of $x$ and $z \notin C_x$. Hence, for $\vol$-a.e.\ $z \in \mathbb{R}^{d}$
\begin{align}
    \abs{\norm{\log_x(z)}^2_{g(x)} - \left( \nabla^2 \varphi(x)(z-x)^2 + \Gamma_{ij}^{k}(x)\partial^2_{k\ell}\varphi(x)(z-x)^i(z-x)^j(z-x)^\ell\right)} \leq K\|z-x\|^4. 
\end{align}
The Christoffel symbols in primal coordinates are given in \eqref{eq:cris-sym-primal}. Thus, we have
\begin{align}
    \abs{\norm{\log_x(z)}^2_{g(x)} - \left(\nabla^2 \varphi(x)(z-x)^2 + \frac{1}{2}\nabla^{3}\varphi(x)(z-x)^{3}\right)} \leq K\norm{z-x}^{4}.
\end{align}
The lemma follows by comparing this expansion with that of \eqref{eq:symm-breg-expan}.   
\end{proof}

\subsection{Entropic Cost Expansion}
The expansion of entropic cost in $\eps$ is well-established from works such that \cite[Theorem 1.6]{conforti21deriv} and \cite[Theorem 1]{chizat20-faster-wass}. However, in the case of fully supported densities on $\mathbb{R}^{d}$, this expansion is not present in the literature. As this setting is precisely the one considered in Section \ref{sec:diff-marg}, we present the following argument. 
\begin{proposition}\label{prop:sb-cost-exp-rd}
    Let $\mu,\nu \in \cP_{2}(\mathbb{R}^{d})$ be such that $ \Ent(\mu), \Ent(\nu)$ are finite and $\int_0^1 I(\rho_t^0|\mathrm{Leb}) dt < +\infty$, where $(\rho_t^{0}, t \in [0,1])$ is the McCann interpolation. Let $P^{\eps}$ denote the $\eps$-dynamic $\Schro$ bridge from $\mu$ to $\nu$ computed with respect to reversible Wiener measure $R^{\eps}$ on $\mathbb{R}^{d}$, then
    \begin{align}\label{eq:sb-cost-exp}
        \eps H(P^{\eps}|R^{\eps}) &= \frac{1}{2}\Was{2}^{2}(\mu,\nu) + \frac{\eps}{2}\left(\Ent (\mu) + \Ent (\nu)\right) + \frac{\eps^2}{8}\int_0^1 I(\rho_t^0|\mathrm{Leb})dt + o(\eps^2).
    \end{align}
\end{proposition}
\begin{proof}
    Fix $\eps > 0$ and let $(\rho_t^{\eps},v_t^{\eps}, t\in [0,1])$ denote the $\eps$-entropic interpolation from $\mu$ to $\nu$. By the suboptimality of the McCann interpolation for the $\eps$-entropic Benamou Brenier,
    \begin{align}\label{eq:ebb-mcann-subopt}
        \frac{1}{2}\int_0^1 \norm{v_t^{\eps}}^2_{L^2(\rho_t^{\eps})}dt + \frac{\eps^2}{8}\int_0^1 I(\rho_t^{\eps}|\mathrm{Leb})dt &\leq \frac{1}{2}\Was{2}^{2}(\mu,\nu)+\frac{\eps^2}{8}\int_0^1 I(\rho_t^0|\mathrm{Leb})dt. 
    \end{align}
    Similarly, by the suboptimality of the $\eps$-entropic interpolation for the (unregularized) Benamou-Brenier \cite{benamou2000computational}, it holds that
    \begin{align}\label{eq:bb-eint-subopt}
        \frac{1}{2}\Was{2}^{2}(\mu,\nu) \leq \frac{1}{2}\int_0^1 \norm{v_t^{\eps}}^2_{L^2(\rho_t^{\eps})}dt.
    \end{align}
    As these inequalities hold for all $\eps > 0$,
    \begin{align*}
        \int_0^1 I(\rho_t^{\eps}|\mathrm{Leb})dt \leq \int_0^1 I(\rho_t^0|\mathrm{Leb})dt \Rightarrow \limsup\limits_{\eps \downarrow 0}\int_0^1 I(\rho_t^{\eps}|\mathrm{Leb})dt \leq \int_0^1 I(\rho_t^0|\mathrm{Leb})dt.
    \end{align*}
    On the other hand, $\rho_t^{\eps}$ converges weakly as $\eps \downarrow 0$ to $\rho_t^0$ for each $t \in [0,1]$. The lowersemicontinuity of $I(\cdot|\mathrm{Leb})$ with respect to weak convergence \cite[Proposition 13.2]{bobkov-fisher-22} and Fatou's lemma establishes the matching lower bound. Thus, the integrated Fisher information is continuous at $\eps = 0$. Next, rearrange \eqref{eq:ebb-mcann-subopt} to obtain
    \begin{align*}
        \limsup\limits_{\eps \downarrow 0}\frac{8}{\eps^2}\left(\frac{1}{2}\int_0^1 \norm{v_t^{\eps}}^2_{L^2(\rho_t^{\eps})}dt - \frac{1}{2}\Was{2}^{2}(\mu,\nu)\right) &\leq \limsup\limits_{\eps \downarrow 0} \left(\int_0^1 I(\rho_t^0|\mathrm{Leb})dt -\int_0^1 I(\rho_t^{\eps}|\mathrm{Leb})dt \right) = 0.
    \end{align*}
    Similarly, \eqref{eq:bb-eint-subopt} furnishes the matching lower bound. Altogether then,
    \begin{align}\label{eq:sb-ke-lim}
        \lim\limits_{\eps \downarrow 0} \frac{8}{\eps^2}\left(\frac{1}{2}\int_0^1 \norm{v_t^{\eps}}^2_{L^2(\rho_t^{\eps})}dt - \frac{1}{2}\Was{2}^{2}(\mu,\nu)\right)  &= 0. 
    \end{align}
    The claimed expansion then follows from \eqref{eq:ent-cost-ent-interp} and applying \eqref{eq:sb-ke-lim} and the continuity at $\eps = 0$ of the integrated Fisher information. 
\end{proof}

\end{document}